\documentclass[11pt]{article} 

\usepackage{amsmath,amssymb,amsthm}
\usepackage{mathtools} 
\usepackage[utf8]{inputenc}

\newtheorem{thm}{Theorem} 
\newtheorem{lem}{Lemma} 
\newtheorem{cor}{Corollary} 
\newtheorem{prop}{Proposition}
\newtheorem{rem}{Remark} 
\newtheorem{defi}{Definition}

\newcommand{\eps}{\varepsilon} 
\newcommand{\ep}{\epsilon} 
\newcommand{\R}{\mathbb{R}} 

\newcommand{\N}{\mathbb{N}} 
\newcommand{\Z}{\mathbb{Z}} 

\newcommand{\mb}{\mathversion{bold}} 
\newcommand{\mn}{\mathversion{normal}} 
\newcommand{\ds}{\displaystyle}

\newcommand{\Cg}{\mathcal{C}_\gamma}
\newcommand{\dP}{d_\mathcal{P}}

\def\rest{\hskip 1pt{\hbox to 10.8pt{\hfill \vrule height 7pt width 0.4pt depth 0pt\hbox{\vrule height 0.4pt width
7.6pt depth 0pt}\hfill}}}
\def\evalu{\hskip 1pt{\hbox to 2pt{\hfill \vrule height -6pt width 0.4pt depth 0pt}}} 
\def\barint{\mathop{\vrule width 6pt height 3 pt depth -2.5pt \kern -8.8pt \intop}}
\title{On Schr\"odinger maps from $T^1$ to $S^2$} \author{{\sc Robert  L. Jerrard}  \& {\sc Didier Smets} }
\date{}
\begin{document}
\maketitle

\begin{abstract} We prove an estimate for the difference of two solutions of the Schr\"odinger map equation for maps from $T^1$ to $S^2.$  This estimate yields some
 continuity properties of the flow map for the topology of $L^2(T^1,S^2)$, provided one takes its quotient by the 
continuous group action of $T^1$ given by translations. We also prove that without taking this quotient, for any
$t>0$ the flow  map at time $t$ is discontinuous as a map from $\mathcal{C}^\infty(T^1,S^2)$, equipped with the 
weak topology of $H^{1/2},$ to the space of
distributions $(\mathcal{C}^\infty(T^1,\R^3))^*.$ 
\end{abstract}

%%%%%%%%%%%%%%%%%%%%%%%%%%%%%%%%%%%%%%%%%%%%%%%%%%%%%%%%%%%%%%%%%%%%%
\section{Introduction}                                              %
%%%%%%%%%%%%%%%%%%%%%%%%%%%%%%%%%%%%%%%%%%%%%%%%%%%%%%%%%%%%%%%%%%%%%

We are interested in the Schr\"odinger map equation, 
\begin{equation}\label{eq:schrodimap} 
\partial_t u = \partial_s \left( u\times \partial_s u\right), 
\end{equation} 
where $u : I \times T^1 \to S^2$, $0\in I \subset \R$ is some open interval, $T^1 \simeq \R/2\pi\Z$ is the flat one
dimensional torus, $S^2$ is the standard unit sphere in $\R^3$,  
and $\times$ is the vector product in $\R^3.$ The solutions we consider are all at least in
$L^\infty(I,H^{1/2}(T^1,S^2))$, so that equation \eqref{eq:schrodimap} has a 
well-defined distributional meaning in $I \times T^1.$ Such solutions are necessarily continuous
with values into
%\footnote{In the sequel we write $H^s_{\rm weak}$
%to denote the Sobolev spaces $H^s$ equipped with their weak topology.} 
$H^{1/2}_{\rm weak}(T^1,S^2)$, so that Cauchy problems are
well-defined too.

\medskip

Our main result provides an upper bound for the growth rate of the difference
between two solutions, one of which being required to be sufficiently smooth. 

\begin{thm}\label{thm:main} Let $u \in \mathcal{C}(I,H^3(T^1,S^2))$ be a solution of the Schr\"odinger map
equation \eqref{eq:schrodimap} on $I=(-T,T)$ for some $T>0.$  Given any other solution 
$v \in L^\infty(I,H^{1/2}(T^1,S^2))$ of \eqref{eq:schrodimap}, there exists a continuous  
function $\sigma\, :  I \to T^1$ such that for every $t\in I,$
\begin{equation*}\label{eq:contimod} 
\|v(t,\cdot) - u(t,\cdot + \sigma(t)) \|_{L^2(T^1,\,\R^3)} \leq C \|v(0,\cdot) - u(0,\cdot)\|_{L^2(T^1,\,\R^3)},
\end{equation*}
where the constant $C \equiv C(\|\partial_{sss}u(0,\cdot)\|_{L^2}, T)$, in particular $C$ does not depend on $v.$\footnote{See Section \ref{sect:th1} for a more explicit version of $C$.}
\end{thm}

For initial data in $H^k(T^1,S^2),$ with $k\geq 2,$ equation \eqref{eq:schrodimap} possesses 
 unique global solutions in $\mathcal{C}(\R,H^k(T^1,S^2)).$ Those solutions preserve the energy 
$$
E(u(t,\cdot)) := \int_{T^1} \big|\partial_s u \big|^2(t,s) \, ds. 
$$
For initial data in the energy space $H^1(T^1,S^2)$, global solutions in $L^\infty(\R,H^1(T^1,S^2))$ are also known to 
exist. The latter are continuous with values in $H^1_{\rm weak}(T^1,S^2)$, but are not known to be unique nor to preserve the energy. 

\smallskip

\begin{rem}\label{rem:flowcont}
 Let $D(\mathcal{F}) \subseteq H^{1/2}(T^1,S^2))$ be the domain of a hypothetical flow map $\mathcal{F}$ for
\eqref{eq:schrodimap} with values into $\mathcal{C}(I,H^{1/2}_{\rm weak}(T^1,S^2)),$ and let $\mathcal{G}$ be one of
the (possibly unique) existing flow maps from $H^1(T^1,S^2)$ with values into $\mathcal{C}(I,H^1_{\rm
weak}(T^1,S^2)).$  For $s\geq 0$, let $H^s/S^1$ and $H^s_{\rm weak}/S^1$ denote the topological spaces 
obtained as the quotients of $H^s(T^1,\R^3)$ and $H^s_{\rm weak}(T^1,\R^3)$ 
respectively by the continuous action group of $S^1$ through translations\footnote{These are of course no longer
vector spaces.}. Let finally $\mathcal{F}_{/S^1}$ and $\mathcal{G}_{/S^1}$ denote the
left composition of $\mathcal{F}$ and $\mathcal{G}$ respectively by the above quotient map\footnote{We mean pointwise in
$I$. In other words
$\mathcal{F}_{/S^1}$ and $\mathcal{G}_{/S^1}$ describe the profiles of the solutions, independently of their locations.}.    
As a consequence of Theorem \ref{thm:main}, we see that:

\smallskip

At any point in $H^3(T^1,S^2)$, $\mathcal{F}_{/S^1}$ is 
continuous as a map
\begin{align*}
\mathcal{F}_{/S^1}\: : \ \Big(\, D(\mathcal{F})\, ,\, {\rm top}\big(L^2\big)\, \Big) &\longrightarrow
\Big(\, \mathcal{C}\big(H^{1/2}_{\rm weak}\big)\, , \, {\rm top}\big( \mathcal{C}(
L^2/T^1\big)\,\big) \Big),\\
\shortintertext{and $\mathcal{G}_{/S^1}$ is continuous as a map}
\mathcal{G}_{/S^1}\: : \ \Big(\, H^1\, ,\, {\rm top}\big(H^1_{\rm weak}\big)\, \Big) &\longrightarrow
\Big(\, \mathcal{C}\big(H^{1}_{\rm weak}\big)\, , \, {\rm top}\big( \mathcal{C}(
H^1_{\rm weak }/T^1\big)\,\big) \Big).
\end{align*}
\end{rem}

\smallskip

We supplement Theorem \ref{thm:main} and Remark \ref{rem:flowcont} by the following ill-posedness type result, which shows in particular 
 that the function $\sigma$ in the statement of Theorem \ref{thm:main} could not be removed.

\begin{thm}\label{thm:illposed}
For any given $t\neq 0,$ the flow map for \eqref{eq:schrodimap} at time $t$ is not continuous 
as a map from $\mathcal{C}^\infty(T^1,S^2)$, equipped with the weak topology of $H^{1/2},$ to the space of
distributions $(\mathcal{C}^\infty(T^1,\R^3))^*.$ 
Indeed, for any $\sigma_0 \in \R$ there exist a 
sequence of smooth initial data $(u_{m,\sigma_0}(0,\cdot))_{m\in \N} \in \mathcal{C}^\infty(T^1,S^2)$ such that
$$
u_{m,\sigma_0}(0,\cdot) \rightharpoonup u^*(0,\cdot) \phantom{ + \sigma_0t} \qquad\text{ in } H^{\frac12}_{\rm
weak}(T^1,\R^3),
$$
where $u^*(0,s) := (\cos(s),\sin(s),0)$
is a stationary solution of \eqref{eq:schrodimap}, and for any $t\in \R$
$$
u_{m,\sigma_0}(t,\cdot) \rightharpoonup u^*(t,\cdot + \sigma_0t) \qquad\text{ in } H^{\frac12}_{\rm weak}(T^1,\R^3).
$$ 
\end{thm}

%%%%%%%%%%%%%%%%%%%%%%%%%%%%%%%%%%%%%%%%%%%%%%%%%%%%%%%%%%%%%%%%%%%%%%%%%%%%%%%%%%%%%%%%%%
\subsection{Three companions : Binormal flow, Schr\"odinger map, Cubic NLS}\label{sub:trium}
%%%%%%%%%%%%%%%%%%%%%%%%%%%%%%%%%%%%%%%%%%%%%%%%%%%%%%%%%%%%%%%%%%%%%%%%%%%%%%%%%%%%%%%%%%

The binormal curvature flow equation for $\gamma : I\times \R \to \R^3$, where $0\in I\subset \R$ is some open interval, is given by
\begin{equation}\label{eq:strongbfbis}
\partial_t \gamma = \partial_s \gamma \times \partial_{ss}\gamma,
\end{equation}
where $s$ is moreover required to be an arc-length parameter for the curves $\gamma(t,\cdot),$ for every $t\in I.$ 
At least for smooth solutions, the arc-length parametrization condition $|\partial_s
\gamma(t,s)|^2 = 1$
is compatible with equation \eqref{eq:strongbfbis}, since 
$$
\partial_t \bigl( |\partial_s \gamma|^2\bigr) = 2 \partial_s\gamma \cdot
\partial_{st}\gamma = 2 \partial_s\gamma \cdot \bigl( \partial_s\gamma \times
\partial_{sss} \gamma\bigr) =0
$$  
whenever \eqref{eq:strongbfbis} is satisfied. 
Denoting by $(\tau,n,b)\equiv (\tau(t,s),n(t,s),b(t,s))$ the 
Frenet-Serret frame along the curve $\gamma(t,\cdot)$ at the point $\gamma(t,s)$, so that $\partial_s \gamma =
\tau$ and $\partial_{ss}\gamma = \kappa n$ where
$\kappa$ is the curvature, 
equation \eqref{eq:strongbfbis} may  be rephrased as
$$
\partial_t \gamma = \kappa b,
$$
from which its name arises. In the sequel, we will only consider solutions of \eqref{eq:strongbfbis} which belong to
$L^\infty(I,H^{3/2}_{loc}(\R,\R^3))$, so that \eqref{eq:strongbfbis} has a well-defined distributional
meaning in $I \times \R$. 

\medskip

If $\gamma \in L^\infty(I,H^{3/2}_{loc}(\R,\R^3))$ is a solution to the binormal curvature flow equation
\eqref{eq:strongbfbis}, then the map $u := \partial_s \gamma \in L^\infty(I,H^{1/2}_{loc}(\R,S^2))$,
parametrizing the evolution in time of the unit tangent vectors to the curves $\gamma(t,\cdot)$, satisfies 
\begin{equation*}
\partial_t u = \partial_t \partial_s \gamma = \partial_s \partial_t \gamma = \partial_s \big( u \times \partial_s
u\big)
\end{equation*}
in the sense of distributions on $I\times \R.$ In other words, $u$ is a solution of the Schr\"odinger map equation
 \eqref{eq:schrodimap} for maps from $\R$ to $S^2$, and the binormal curvature flow equation is therefore a
primitive equation of the Schr\"odinger map equation.

\medskip

Conversely, let $u\in L^\infty(I,H^{1/2}_{loc}(\R,S^2))$ be a solution to the Schr\"odinger map equation
\eqref{eq:schrodimap} and define the function $\Gamma_u \in L^\infty(I,H^{3/2}_{loc}(\R,\R^3))$  by
\begin{equation}\label{eq:tildegamma}
\Gamma_u (t,s) := \int_0^{s} u(t,z)\, dz.
\end{equation}
In the sense of distributions on $I\times \R$, we have 
\begin{equation}\label{eq:almostbf}
\partial_s \Big( \partial_t \Gamma_u - \partial_s \Gamma_u \times \partial_{ss}\Gamma_u\Big)
=0.
\end{equation}
%Notice that if $u\in \mathcal{C}( I, H^{1/2}(T^1,S^2))$, then $\tilde \gamma_u$ may be considered as an element in
% $\mathcal{C}( I, H^{3/2}(T^1,X_u))$, where we recall that $X_u$ is the flat manifold defined by
%$X_u = {\R^3} / ( \int_{T^1} u(0,z)\, dz) \,\Z = {\R^3} / ( \int_{T^1} u(t,z)\, dz)
%\,\Z,$ $\forall t \in  I.$ One may also consider the restriction of $\tilde \gamma_u$ to $[0,2\pi)$ as a
%distribution on $I\times T^1,$ this distribution belongs to $\mathcal{C}( I,L^2(T^1,\R^3))$ but it does not
%belong to $ \mathcal{C}( I,H^{1/2}(T^1,\R^3))$ unless $\int_{T^1} u  = 0.$ 
By construction, the primitive curves $\Gamma_u(t,\cdot)$ all have their base point $\Gamma_u(t,0)$ fixed at the
origin. If they were smooth, equation \eqref{eq:almostbf} would directly imply the existence of a function $c_u$
depending only time only and such that
$$
\gamma_u(t,s) := \Gamma_u(t,s) + c(t)
$$
is a solution to the binormal curvature flow equation \eqref{eq:strongbfbis}. The function $c_u$ indeed represents 
the evolution in time of the actual base point of the curves. In Section \ref{sect:relation} we will turn this 
argument into a statement for quasiperiodic solutions in $H^{3/2}.$

\medskip

Our proofs of Theorem \ref{thm:main} and Theorem
\ref{thm:illposed} spend most of their time in the binormal curvature flow world rather than in the Schr\"odinger
map one.  

\medskip

Besides the bidirectional relation between the binormal curvature flow equation 
and the Schr\"odinger map equation presented just above, Hasimoto \cite{Has} exhibited in 1972 an
intimate relation between the binormal curvature flow equation
\eqref{eq:strongbfbis} and the cubic focusing nonlinear Schr\"odinger equation.
Let $\gamma : I \times \R \to \R^3$ be a smooth and biregular solution of the binormal
curvature flow equation \eqref{eq:strongbfbis}, and denote by $\kappa$ and $T$ respectively the curvature 
and torsion functions of $\gamma$. Then, 
the function $\Psi$ defined on $I\times \R$ by the Hasimoto transform
\begin{equation*}\label{eq:hasi}
\psi(t,s) := \kappa(t,s) \exp\left( i\int_0^s T(t,z)\, dz \right)
\end{equation*}
is a solution to 
\begin{equation*}\label{eq:schrodicubic}
\partial_t \psi + \partial_{ss}\psi +\frac{1}{2}(|\psi|^2-A(t))\psi = 0
\end{equation*}
where
$$
A(t) := \left( 2 \frac{\partial_{ss}\kappa - \kappa T^2}{\kappa} + \kappa^2\right)(t,0).
$$
If $\gamma$ is $2\pi$-periodic in $s$, that is if $\gamma : I\times T^1 \to
\R^3,$ then $\Psi$ is only quasiperiodic unless $\int_0^{2\pi} T(t,z)\, dz \in
2\pi \mathbb{Z}.$ Nevertheless, it is possible to recover a $2\pi$-periodic
function $\Psi$ by means of a Galilean transform. One can also get rid of the
$A(t)$ factor by means of a phase shift. More precisely, the function
$$
\Psi(t,s) := \psi(t,s-\frac{b}{2}t) \exp\left(i(bs-b^2t-\int_0^t \frac{1}{2}A(z)\, dz)\right), 
$$
where
$$
b := 1 - \frac{1}{2\pi}\int_0^{2\pi} T(0,z)\, dz = 1 - \frac{1}{2\pi}\int_0^{2\pi} T(t,z)\, dz
$$
is  well-defined on $I\times T^1$ and is a solution to the cubic focusing nonlinear Schr\"odinger equation
\begin{equation}\label{eq:nlsint}
\partial_t \Psi + \partial_{ss}\Psi +\frac{1}{2}|\Psi|^2\Psi = 0
\end{equation}
on $I\times T^1.$ 

\smallskip

Equation \eqref{eq:nlsint} is integrable and known to be solvable by the inverse
scattering method since the works of Zakharov and Shabat \cite{ZaSh} in 1971 for the vanishing case
and Ma and Ablowitz \cite{MaAb} in 1981 for the periodic case. Therefore the binormal curvature flow equation
\eqref{eq:strongbfbis} and the Schr\"odinger map equation \eqref{eq:schrodimap} are also integrable, in the weak
sense that they can be mapped to an integrable equation.
Notice however that the inverse of the Hasimoto transform, whenever it is well defined\footnote{Vanishing of $\Psi$
yields underdetermination.}, involves reconstructing a
curve from its curvature and torsion functions, an operation which is both highly 
nonlinear and nonlocal. At least for that reason, translating an estimate available for the cubic focusing NLS to the 
binormal curvature flow equation or to the Schr\"odinger map equation, or even vice versa, doesn't look like a 
straightforward task. 

\medskip 

Our proofs of Theorem \ref{thm:main} and Theorem \ref{thm:illposed} make no use of the Hasimoto transform, nor, at least directly, 
of the induced integrability. It would be of interest to better understand the implications, if any, of Theorem 
\ref{thm:main} for the cubic focusing NLS. 

%%%%%%%%%%%%%%%%%%%%%%%%%%%%%%%%%%%%%%%%%%%%%%%%%%%%%%%%%%%%%%
\subsection{Comments on the Cauchy theory}\label{sub:Cauchy}
%%%%%%%%%%%%%%%%%%%%%%%%%%%%%%%%%%%%%%%%%%%%%%%%%%%%%%%%%%%%%%

Schr\"odinger map equations arise in the literature not only for maps from $T^1$ to $S^2$, but 
more generally for maps from a Riemannian manifold to a K\"ahler manifold of arbitrary dimensions.  
It is not our purpose here to cover the growing literature about the Schr\"odinger map equations in
specific or general cases.  Instead, we will briefly focus on those works that, when they appeared, brought an
improvement at the level of the Cauchy theory when particularized (or adapted) to the case of maps from 
$T^1$  to $S^2$ only.   

\smallskip

As far as the Cauchy theory is concerned, it seems that \eqref{eq:schrodimap} was first studied by Zhou and Guo
\cite{ZhGu} in 1984 for maps $u$ from an interval to $S^2$, and by Sulem, Sulem and Bardos \cite{SuSuBa} in 1986
 for maps 
$u$ from $\R^N$ to $S^2$, for arbitrary $N\geq 1.$ Their proofs, which do not make use of any dispersive estimate, 
and therefore their results, can be translated almost word for word\footnote{Since $\R^N$ is non compact and
functions take values in $S^2$, the spaces involved in \cite{SuSuBa} are homogeneous Sobolev spaces.} in the case 
of maps from $T^1$ to $S^2.$ Both papers establish that \eqref{eq:schrodimap} possesses weak solutions in 
$L^\infty(H^1).$  The method in \cite{ZhGu} is based on a parabolic regularization and the Leray-Schauder fixed
point theorem, whereas the method in \cite{SuSuBa} is based on discrete approximation schemes and a weak compactness 
argument.  In particular none of them provide uniqueness in $H^1.$ The global character of their solutions, as
stated in \cite{SuSuBa}, follows from the global character of the approximating ones, since an energy is a Lyapunov
functional for the latter. 

\smallskip 

In a second step, differentiating  equation \eqref{eq:schrodimap} with 
respect to time and then performing energy estimates, it is proved in \cite{SuSuBa} that the growth rate of any $H^k$ norm of
a solution is controlled a priori by its $W^{2,\infty}$ norm. By Sobolev embedding and compactness, this yields local existence
in $L^\infty(H^k)$ for any $k>N/2+2,$ that is $k=3$ in dimension 1. Working out some of the details in
\cite{SuSuBa}, one realizes (see e.g. Appendix \ref{sect:cauchy}) that in dimension 1 the $H^2$ norm of the solution suffices to control
all the other higher norms, which yield local existence in $L^\infty(H^2)$ as well. Global existence in
$L^\infty(H^k)$ for any $k\geq 2$ then follows, as remarked in \cite{SuSuBa} or in Zhou, Guo and Tan
\cite{ZhGuTa} in 1991, by the a priori control on the $H^2$ norm given by the further conservation law for
$$
I(u(t,\cdot)) := \int_{T^1} \big|\partial_{ss} u + |\partial_s u|^2u\big|^2(t,s) + \big|\partial_s u\big|^2(t,s) -
\frac{3}{4}\big|\partial_s u\big|^4(t,s)\, ds,
$$  
which holds at least for solutions in $L^\infty(H^3)$. Pushing arguments of \cite{SuSuBa}\footnote{Both \cite{SuSuBa} 
and \cite{ZhGuTa} have let aside the study of the difference between two solutions, asserting that the
uniqueness followed from regularity. Whereas this is generally true, it is sometimes desirable, and we actually
needed that for Theorem \ref{thm:main}, to know which second norm controls what first norm of the difference.} a bit further, one also obtains 
 global well-posedness in $H^k$, for $k\geq 3$, as well as Lipschitz continuity of the flow map
in $H^3$, for the $H^1$ or the $H^2$ norm, with a constant depending only on the $H^3$ norm (see also 
Appendix \ref{sect:cauchy}).

\smallskip

In 2001, Ding and Wang \cite{DiWa} used the parabolic approximation approach and gave complete proofs of 
existence of global solutions  in $L^\infty(H^k),$ for any $k\geq 2$, and of uniqueness in $L^\infty(H^k)$, for any $k\geq 3.$   

\smallskip

Uniqueness in $L^\infty(H^2)$ was proved by Chang, Shatah and Uhlenbeck \cite{ChShUh} in 2000, 
 some part being clarified by Nahmod, Shatah, Vega and Zeng \cite{NaShVeZe} in 2006. The strategy used in those works 
consists in writing the components of the derivatives of a solution $u$ in a special orthonormal frame in the 
tangent plane $u^\perp$, a frame adapted to the solution itself, in which the equation turns out to be of NLS type.
 In essence, this is very much alike the Hasimoto transform, mentioned in the previous subsection, but it has the
advantage of additional flexibility. 
 Borrowing known uniqueness results at the NLS level, uniqueness for the Schr\"odinger map equation follows, provided 
one can pull it back (which requires some care since the frame depends on the solution).       

\smallskip

A different approach to uniqueness in $H^2$, based on more classical energy methods, was successfully used
 by McGahagan \cite{McG,McG2} in 2004. The key ingredient in her analysis, arguments earlier used by 
Shatah in related contexts (see e.g. \cite{Sh}), is to measure the discrepancy between the derivatives of two 
different solutions not by embedding the target manifold in a big euclidean space and using euclidean distances,
but by parallel transporting one derivative into the tangent plane of the other one and measuring there. 
Particularized to the case of $T^1$ and $S^2$, \cite{McG2} implies global well-posedness in $L^\infty(H^2)$, and Lipschitz continuity of
the flow map in $H^2$, for the $H^1$ norm, with a constant depending only on the $H^2$ norm.        

\smallskip
At the time of writing, there doesn't seem to be any local well-posedness result available below the $H^2$ regularity
 threshold. On the other hand, as expressed in particular in the introduction of \cite{NaShVeZe}, in terms of scalings
 equation \eqref{eq:schrodimap} is sub-critical in $H^s(T^1,S^2)$ for any $s>1/2.$  Therefore, even though one would be tempted to
believe that dimension $N=1$ is the easiest case as far as the Cauchy theory is concerned, there remain 
substantial gaps between the criticality at the level of $H^{1/2}$, the existence at the level of $H^1$, and the 
well-posedness at the level of $H^2.$

\medskip
Theorem \ref{thm:main} corresponds to a weak-strong type estimate, since one of the solutions is required to 
be smooth. Being such, there is no hope that it could be used alone to improve the Cauchy theory for 
\eqref{eq:schrodimap}. On the other hand, the estimate which it provides holds in a norm with very few 
regularity ($L^2$ is $1/2$ derivative below the critical space).  

\medskip
Theorem \ref{thm:illposed} focuses on ill-posedness\footnote{In higher dimensions Sch\"rodinger maps are
not always globally defined, and blow-up analysis is a current area of active investigation, but this is a
different issue.} for low regularity norms. Such considerations have attracted attention in related contexts of
dispersive equations, in particular since Bourgain \cite{Bo2} remarked that in cases it was possible to 
rule out regularity of the flow map even in spaces with more regularity than the critical space. A special
 attention was devoted to the cubic nonlinear Schr\"odinger equation in one spatial
dimension,  in particular in the works of Kenig-Ponce-Vega 2001 \cite{KePoVe}, Burq-G\'erard-Tzvetkov 2002 \cite{BuGeTz} and
Christ-Coliander-Tao 2003 \cite{ChCoTa}. In that case, the critical space in terms of scaling is $H^{-1/2}$, global 
well-posedness was proved in 1993 by Bourgain \cite{Bo} in the space $L^2$, and failure of continuity of the flow map
arises in any $H^s$ with $s<0.$ Counter examples to continuity were essentially constructed by 
superposition of a slow and a rapid modes, each of them being either explicit or well understood through the Galilean 
invariance of the equation. Our strategy for Theorem \ref{thm:illposed} is somewhat parallel though in a rather
different context, and it does not really involve new ideas. In 2007, Molinet \cite{Mo} gave further insight into the non continuity issue in the
focusing cubic NLS case, and was even able to recover some form of continuity in $L^2_{\rm weak}$, provided one is allowed to quotient by
the action group of $S^1$, through phase multiplication there. In some sense, Theorem \ref{thm:main} bears some 
common features with Molinet's result\footnote{One could actually argue that the focusing cubic NLS on $T^1$ and
equation \eqref{eq:schrodimap} are actually very much related, not to say equivalent, in particular in view of the Hasimoto
 transform (see Subsection \ref{sub:trium} above) or the related gauge transforms of Chang, Shatah and Uhlenbeck \cite{ChShUh}. 
Things are probably a little more complex though, since, as we already wrote, those transforms involve the solution itself 
and translating back and forth results in each framework requires at least some care.}.

%%%%%%%%%%%%%%%%%%%%%%%%%%%%%%%%%%%%%%%%%%%%%%%%%%%%%%%%%%%%%%%%%%%%%%%%%%%%%%
\subsection{Strategy of the proofs and Plan of the paper}\label{sub:strategy}
%%%%%%%%%%%%%%%%%%%%%%%%%%%%%%%%%%%%%%%%%%%%%%%%%%%%%%%%%%%%%%%%%%%%%%%%%%%%%%

The general strategy of the proof of Theorem \ref{thm:main} is as follows; we first voluntarily somewhat 
simplify the truth in order not to obscure the main ideas, and then indicate which
additional issues have to be overcome in a second step. 

\smallskip 

Given two solutions $u$ and $v$ of the Schr\"odinger map equation \eqref{eq:schrodimap},
we construct the two primitive solutions $\gamma_u$ and $\gamma_v$ of the binormal curvature
flow equation whose derivatives coincide with $u$ and $v$ respectively, as explained in Subsection
\ref{sub:trium} and proved in Section \ref{sect:relation}. Let us write $\gamma_u \equiv \gamma \equiv (\gamma_t)_{t\in
I}$ and $\gamma_v \equiv \Gamma \equiv (\Gamma_t)_{t\in I};$  one should keep in mind that $\gamma$ is the smooth solution
and $\Gamma$ is the rough one.  If $u(0,\cdot)$ and
$v(0,\cdot)$ are close in $L^2$, then the curve $\Gamma_0$ is close to the curve $\gamma_0$ 
in many respects, but in particular for the discrepancy measure
$$
F_{\gamma_0}(\Gamma_0) := \int_{\Gamma_0} 1 - f(d^2(x,\gamma_0)) \tau_{\Gamma_0}(x) \cdot \tau_{\gamma_0}(P(x))\,
d\mathcal{H}^1, 
$$ 
where $d$ is the distance function, $f(d^2) \simeq 1- d^2$ in a neighborhood of $d=0$, $P$ is the orthogonal
projection on the curve $\gamma_0$ and $\tau_{\Gamma_0}$ and $\tau_{\gamma_0}$ denote the oriented unit tangents to
$\gamma_0$ and $\Gamma_0$ respectively. It turns out, and this is the main ingredient in our proof, that since
$\Gamma$ and $\gamma$ both solve the binormal curvature flow equation, we have the Gronwall type inequality
\begin{equation}\label{eq:gronwallrough}
\frac{d}{dt} F_{\gamma_t}(\Gamma_t) \leq C F_{\gamma_t}(\Gamma_t)
\end{equation}
almost everywhere on $I$, where the constant $C$ is explicit in terms of the $L^\infty$ norms of 
$\partial_{ss}\gamma_t$ and $\partial_{sss}\gamma_t.$   In order to prove this inequality, we use in particular the 
fact that $\Gamma$ satisfies some weak formulation of \eqref{eq:strongbfbis}, namely
\begin{equation}\label{eq:weakrough}
\frac{d}{dt} \int_{\Gamma_t} X \cdot \tau_{\Gamma_t} \, d\mathcal{H}^1 = \int_{\Gamma_t} \partial_t X\cdot
\tau_{\Gamma_t} - D({\rm curl}\, X)\,:\,
(\tau_{\Gamma_t} \otimes \tau_{\Gamma_t})\, d\mathcal{H}^1,
\end{equation} 
for every smooth time-dependent vector field $X : I\times \R^3 \to \R^3.$ In view of the definition of $F$, we apply \eqref{eq:weakrough} for the choice $X(x) := d^2(x,\gamma_t)\tau_{\gamma_t}(P(x)).$  
Once \eqref{eq:gronwallrough} is proved, we obtain from Gronwall's inequality a control on  $F_{\gamma_t}(\Gamma_t)$
in terms of $F_{\gamma_0}(\Gamma_0),$ and therefore of $\|u(0,t)-v(0,t)\|_2^2.$ The last task we have to deal with is 
then to relate this discrepancy measure $F_{\gamma_t}(\Gamma_t)$ to the closeness in $L^2$ between
$u(t,\cdot)$ and $v(t,\cdot).$  Notice the important fact that, by its definition, $F$ doesn't care too much about 
the actual parametrizations of $\gamma_t$ and $\Gamma_t$ (only $P$ is involved).  In particular $F$ is insensible to 
constant shifts in the space variables of one or both of the curves, and therefore it does not control the 
squared $L^2$ distance between $u(t,\cdot)$ and $v(t,\cdot).$ It turns out that $F$ actually controls the squared 
$L^2$ distance between $v(t,\cdot)$ and a suitable translate $u(t,\cdot + \sigma(t))$ of $u(t,\cdot)$
, and that one can keep track of $\sigma(t)$ in time;  indeed, $\sigma(t)$ is essentially determined by
 $P(\Gamma_t(0))=\gamma_t(\sigma(t))$. The conclusion of Theorem \ref{thm:main} then follows.

As mentioned, there are some issues coming into the program. First, we need the curves $\gamma_u$ and $\gamma_v$ 
to have the same pitch in order to perform the analysis at the binormal curvature flow level. If $u$ and $v$ do not
have the same mean, we therefore first replace $u$ by another smooth solution $\tilde u$ which is sufficiently close 
to $u$ and has the same mean as $v$. We apply our strategy to $\tilde u$ and $v$ and then estimate the difference
between $\tilde u$ and $u$ by standard energy methods. Second, the orthogonal projection $P$ involved in the
definition of $F$ may not be well-defined, in particular if the smooth curve $\gamma$ has self-intersections or is close
to be so. To overcome this difficulty, we need to keep track of the fact that at the initial time, the curves
$\gamma_0$ and $\Gamma_0$ are not only close in terms of the Hausdorff distance or of $F$, but are also 
``parametrically'' close.  In doing so, we can guess at which place on $\gamma_t$ we should orthogonally project
each point of $\Gamma_t$, even when self-intersections are involved. The bottom line is that the actual definition
of $F$ gets somewhat more complex, as well as the weak formulation in \eqref{eq:weakrough}, but the overall strategy remains the
same.            

\medskip

The ideas behind the proof of Theorem \ref{thm:illposed} are part of the folklore on  
ill-posedness issues for dispersive equations in low regularity spaces. The functions $u_{m,\sigma_0}$ mentioned in
the statement of Theorem \ref{thm:illposed} are derived from explicit particular solutions $\gamma \equiv 
\gamma_{m,\tilde\sigma_0}$ of the binormal curvature flow equation. Those special solutions satisfy the property that at any positive time $t$ 
the solution curve $\gamma_t$ is obtained from the initial curve $\gamma_0$ through a rigid motion. Such solutions 
have been studied and classified in a beautiful paper by Kida \cite{Ki} in 1981. Among the many-parameters family of exact
solutions that Kida constructed, we carefully single out a sequence of initial data that have the form 
of a helix of small pitch and small radius wrapped around a fixed circle of unit radius. The pitch and radius of those helices tend
to zero along the sequence, with the radius being proportional to the power $3/2$ of the pitch. For each of these
initial data, the corresponding solution $\gamma_t$ at time $t$ is obtained by the superposition of two rigid
motions applied to $\gamma_0$: a translational motion of the base circle (this corresponds to the ``limit solution'', a circle being a
travelling wave solution of equation \eqref{eq:strongbfbis}) and a nontrivial
sliding motion along the base circle induced by the curvature of the small helix. Recast into the framework of the
Schr\"odinger map equation, this sliding motion gives us the continuity breaking parameter.            

\medskip

\noindent
We now shortly describe the plan of the paper. In Section \ref{sect:relation} we make precise the construction of a solution 
of \eqref{eq:strongbfbis} from a solution
of \eqref{eq:schrodimap}, requiring as little regularity as possible. In Section \ref{sect:meas}, we introduce the
adaptation of the functional $F$ above and relate it to different types of discrepancy measures between two curves.
 In Section \ref{sect:weakp}, we introduce and prove a weak formulation similar to \eqref{eq:weakrough} for 
parametrized solutions of the binormal curvature flow equation. In Section \ref{sect:key} we prove the basic
estimate which is at the heart of \eqref{eq:gronwallrough}. We gather all these informations in Section
\ref{sect:gron} to derive a discrepancy estimate for parametrized solutions of the binormal curvature flow
equation. Finally, we prove Theorem \ref{thm:main} in Section \ref{sect:th1} and Theorem \ref{thm:illposed} in
Section \ref{sect:th2}. In appendix we have gathered some presumably well-known continuity estimates for the 
flow map of \eqref{eq:schrodimap} at the level of smooth solutions. We couldn't find them with sufficient precision
about the constants dependence in the literature.

%%%%%%%%%%%%%%%%%%%%%%%%%%%%%%%%%%%%%%%%%%%%%%%%%%%%%%%%%%%%%%%%%%%%%%%%%%%%%%%%%%%%%%%%%%%%%%%%%%%%%%%%%%%%%%%%
\section{Schr\"odinger maps and Binormal curvature flow}\label{sect:relation}
%%%%%%%%%%%%%%%%%%%%%%%%%%%%%%%%%%%%%%%%%%%%%%%%%%%%%%%%%%%%%%%%%%%%%%%%%%%%%%%%%%%%%%%%%%%%%%%%%%%%%%%%%%%%%%%%

In this section, we make precise the relation between the binormal curvature flow equation and the Schr\"odinger
map equation which was outlined in Subsection \ref{sub:trium} in the Introduction. 
For later purposes, we wish to consider solutions of \eqref{eq:schrodimap} defined on tori of length possibly
different from the length $2\pi$ which we have arbitrarily chosen so far. Of course one can always use the scalings
of the equation to turn back to such a situation, but in the section following this one we would like to 
explicitly keep track of some constants depending on the length. For $\ell>0,$ we therefore denote by 
$T^1_\ell$, or actually simply $T^1$ if no confusion occurs, the flat torus
$$
T^1_\ell := \R \big/ \ell \Z.
$$

\begin{lem}\label{lem:basepoint}
Let $w\in L^\infty( I, H^{1/2}(T^1,S^2))$ be a solution of \eqref{eq:schrodimap}, and let $\Gamma_w$
be defined by \eqref{eq:tildegamma}. There exists a unique continuous function $c_w\,:\,  I \to \R^3$
satisfying $c_w(0)=0$ and such that the function $\gamma_w \in L^\infty(
I,H^{3/2}_{loc}(\R,\R^3))$ defined by
\begin{equation}\label{eq:plusbase}
\gamma_w (t,s) := \Gamma_w(t,s) + c_w(t),
\end{equation}
is a solution of the binormal curvature flow equation \eqref{eq:strongbfbis} on $I\times \R.$   
\end{lem}
\begin{proof}
Since the function $\Gamma_w$ is not a distribution on $I\times T^1$, we first consider instead its restriction $\tilde
\Gamma_w$ to $ I \times [0,\ell),$ which we view as a distribution on $I\times T^1.$  We have 
$\tilde \Gamma_w \in \mathcal{C}( I,L^2(T^1,\R^3))$, but in general $\tilde \Gamma_w \notin L^\infty(
I,H^{1/2}(T^1,\R^3))$, since  for each $t\in  I$  the function $\tilde \Gamma_w(t,\cdot)$ has a jump of 
height $-\int_{T^1} w(t,z)\, dz$ at the origin. We also define $S \in \mathcal{D}'( I \times T^1,\R^3)$ by 
$$
S := \partial_t \tilde\Gamma_w - w\times \partial_s w.
$$  

First, we have
\begin{equation}\label{eq:deun}\begin{split}
\partial_s S &= \partial_t \partial_s \tilde\Gamma_w - \partial_s\big( w \times \partial_s w\big)\\
& = \partial_t  w - \partial_s\big( w \times \partial_s w\big)  - \partial_t \left( (\int_{T^1}
w(t,z)\,dz)\, dt \otimes \delta_{0} \right)\\
& = 0, 	
\end{split}	
\end{equation}
since \eqref{eq:schrodimap} holds and because \eqref{eq:schrodimap} also implies that  $\int_{T^1} w(t,z)\, dz$
does not depend on time. 

Next, let $\varphi \in  \mathcal{D}(I\times T^1,\R^3).$ We may write 
$$\varphi(t,s) = \partial_s\chi (t,s) + m(t,s),$$
where 
$$
m(t,s)\equiv m(t) := \frac{1}{\ell} \int_{T^1} \varphi(t,s)\, ds \in \mathcal{D}(I\times T^1,\R^3)
$$
and
$$
\chi(t,s) =  :\int_0^s \varphi(t,z)-m(t)\, dz \in
 \mathcal{D}(I\times T^1,\R^3).
$$
In view of \eqref{eq:deun}, we obtain
\begin{equation}\label{eq:dedeux}
\langle S , \varphi \rangle_{I\times T^1}  = \langle S , \partial_s \chi + m\rangle_{I\times T^1} = -\langle \partial_s S , \chi \rangle_{I\times T^1} +
\langle S , m \rangle_{I\times T^1}
 = \langle S , m \rangle_{I\times T^1},
\end{equation}
where we denote by $\langle\cdot,\cdot\rangle_{I\times T^1}$ the duality between $\mathcal{D}'(I\times T^1,\R^3)$
and $\mathcal{D}(I\times T^1,\R^3).$ We expand
\begin{equation}\label{eq:detrois}\begin{split}
\langle S , m \rangle_{I\times T^1} &= - \langle \tilde\Gamma_w , \partial_t m\rangle_{I\times T^1} - \langle w
\times \partial_s w , m\rangle_{I\times T^1}\\
&= - \int_I \int_0^{\ell} \left(\int_0^s w(t,z)\, dz\right) \cdot m'(t)\, dsdt - \int_I \int_0^{\ell}
\big(w(t,s)\times\partial_s w(t,s)\big) \cdot m(t)\, dsdt\\
&= \int_I \int_0^{\ell}\left( (s-\ell)w(t,s)+\int_0^t \big(w(\tau,s)\times\partial_s
w(\tau,s)\big)\, d\tau \right)\, ds \cdot m'(t)\, dt\\
&= \int_I  c_w(t)\cdot m'(t)\, dt,
\end{split}	
\end{equation}
where
$$
c_w(t) := \int_0^{\ell}\left( (s-\ell)\big(w(t,s)-w(0,s)\big)+\int_0^t \big(w(\tau,s)\times\partial_s
w(\tau,s)\big)\, d\tau \right)\, ds.
$$
Notice that by construction $c_w(0)=0$, and, since $w\in \mathcal{C}(I,H^{1/2}_{\rm weak}(T^1,S^2))$, we have $c_w \in
\mathcal{C}( I,\R^3).$
 We will now prove that the function $\gamma_w$ defined on $ I \times \R$ by
$$
\gamma_w(t,s) := \Gamma_w(t,s) + c_w(t)
$$
is a solution to the binormal curvature flow equation \eqref{eq:strongbfbis} on $I\times \R.$ 
To that purpose, let $\psi \in \mathcal{D}(I\times\R,\R^3)$ be arbitrary. We first write 
\begin{equation}\label{eq:dequatre}
\langle \partial_t  \gamma_w - \partial_s \gamma_w \times \partial_{ss} \gamma_w , \psi \rangle_{I\times \R} = 
\langle \partial_t  \Gamma_w - w \times \partial_{s} w , \psi \rangle_{I\times \R} - \int_I c_w(t) \cdot n'(t)\, dt,
\end{equation}
where
$$
n(t) := \int_\R \psi(t,s) \, ds.
$$
In order to compute the first term on the right-hand side of \eqref{eq:dequatre}, we consider a  
partition of unity $1 = \sum_{j\in \Z} \lambda(\cdot+j\ell/2)$ on $\R$,  where $\lambda\in \mathcal{D}(\R,\R)$ has
compact support in
$(0,\ell)$, and we decompose $\psi$ as $\psi = \sum_{j\in \Z} \psi_j$ where $\psi_j \in \mathcal{D}( I \times
\R,\R^3)$ is defined by
$$
\psi_j(t,s) := \psi(t,s)\lambda(s-j\ell/2).
$$ 
For $j\in \Z,$ we also define $\tilde \psi_j \in \mathcal{D}( I \times T^1,\R^3)$ by
$$
\tilde \psi_j(t,s) := \sum_{k\in \Z} \psi_j(t,s+j\ell/2+k\ell).
$$
%For each $k$ in $\Z$, we have $\Gamma_u(t,s+2\pi k) = \Gamma_u(t,s) + k\int_{T^1}u.$   Therefore, for each $j \in
%\Z$, we may write
Since $\psi_j$ has its support in $ I \times [j\ell/2,(j/2+1)\ell]$, we have  
\begin{equation}\label{eq:taylor1}
\langle \partial_t \Gamma_w , \psi_j \rangle_{I\times \R} 
= -\int_I\int_{j\ell/2}^{(j/2+1)\ell} \tilde \Gamma_w(t,s) \cdot \partial_t \psi_j(t,s)\,dsdt.
\end{equation}
Changing variables and using the equality $\Gamma_w(t,s+\ell) = \Gamma_w(t,s) + \int_{T^1}w$, valid for any $t\in
 I$ and $s\in \R$, we obtain, for each $t\in  I,$
\begin{equation}\label{eq:chine}
\int_{j\ell/2}^{(j/2+1)\ell} \tilde \Gamma_w(t,s) \cdot \partial_t \psi_j(t,s)\,ds
= \int_{0}^{\ell} \tilde
\Gamma_w(t,s) + H_j(s) \cdot \partial_t \tilde\psi_j(t,s)\, ds
\end{equation}
where $H_j(s) \equiv (\int_{T^1}w)p$  for all $s\in [0,\ell)$ if $j=2p\in 2\Z$, whereas $H_j(s)=(\int_{T^1}w)(p+1)$
for $s\in [0,\ell/2)$ and $H_j(s)
=  (\int_{T^1}w)p$ for $s\in [\ell/2,\ell)$ if $j=2p+1\in 2\Z+1.$ Since in any case $H_j$ does not depend on time, integrating \eqref{eq:chine} in time
and using \eqref{eq:taylor1} yields 
\begin{equation}\label{eq:ping}
\langle \partial_t \Gamma_w , \psi_j \rangle_{I\times \R}= 
\langle \partial_t \tilde \Gamma_w , \tilde \psi_j \rangle_{I\times T^1}.
\end{equation}
By periodicity of $w$, we also have
\begin{equation}\label{eq:pong}
\langle -w\times\partial_s w , \psi_j \rangle_{I\times \R}= 
\langle -w\times \partial_s w , \tilde \psi_j \rangle_{I\times T^1}.
\end{equation}
Therefore, adding \eqref{eq:pong} to \eqref{eq:ping} and using \eqref{eq:dedeux} and \eqref{eq:detrois} with
the choice $\varphi = \tilde \psi_j$ we are lead to
\begin{equation}\label{eq:fuku}
\langle \partial_t \Gamma_w-w\times\partial_s w , \psi_j \rangle_{I\times \R} = \langle S ,
\tilde \psi_j \rangle_{I\times T^1} = \int_I c_w(t) \cdot n_j'(t)\, dt,
\end{equation}
where
$$
n_j(t) := \int_{T^1} \tilde\psi_j(t,s)\, ds = \int_{\R} \psi_j(t,s)\, ds.
$$
Summation in $j$ inside \eqref{eq:fuku} and substitution in \eqref{eq:dequatre} finally yields
$$
\langle \partial_t  \gamma_w - \partial_s \gamma_w \times \partial_{ss} \gamma_w , \psi
\rangle_{I\times \R} = \int_I \Big(
(\sum_j n_j'(t)) - n(t)\Big) \cdot c_w(t)\, dt = 0,
$$ 
so that the claim is proved and hence the existence of $c_w$. Since $c_w$ is required to be continuous 
with $c_w(0)=0$, and since its distributional derivative is completely determined by \eqref{eq:almostbf} and
\eqref{eq:plusbase}, its uniqueness follows.   
\end{proof}

Unless $a:=\int_{T^1} w = 0$, the ``curves'' $\gamma_w$ constructed in the previous lemma are not closed, but merely 
invariant under translation by $a.$ In the next section, we introduce various ways to measure the distance between such
curves, and briefly study their relations.

%%%%%%%%%%%%%%%%%%%%%%%%%%%%%%%%%%%%%%%%%%%%%%%%%%%%%%%%%%%%%%%%%%%%%%%%%
\section{Measuring the closeness between quasiperiodic curves}\label{sect:meas}
%%%%%%%%%%%%%%%%%%%%%%%%%%%%%%%%%%%%%%%%%%%%%%%%%%%%%%%%%%%%%%%%%%%%%%%%%

\begin{defi}\label{defi:curves}
A curve in $\R^3$ parametrized by arc-length is a Lipschitz map $\gamma\, :\,
\R \to \R^3$ such that 
$$
|\gamma'(s)| = 1 \qquad \text{for a.e. }\, s \in \R.
$$
A quasiperiodic curve in $\R^3$ of period $\ell > 0$ and pitch $a\in \R^3$ is
a curve in $\R^3$ parametrized by arc-length and such that
$$
\gamma(s+\ell) = \gamma(s) + a,\qquad \forall\, s\in \R.
$$
A closed curve in $\R^3$ of length $\ell > 0$ is a quasiperiodic curve in $\R^3$ of period $\ell$ and
zero pitch. 
\end{defi}

\begin{rem}
i) When $\gamma$ is a curve in $\R^3$ parametrized by arc-length, and when this does not lead to a confusion, 
we also denote by $\gamma$ the subset $\{\gamma(s),\, s\in \R\}$ of $\R^3.$\\
ii) If $\gamma$ is a quasiperiodic curve of period $\ell$ and pitch $a,$ then $\gamma$ is also a quasiperiodic curve of
period $n\,\ell$ and pitch $n\, a$, for any integer $n\geq 1.$ In the sequel, when we refer to ``\,the'' period
and ``\,the'' pitch 
of a quasiperiodic curve $\gamma$, we implicitly assume that $\gamma$ comes with a pre-assigned period $\ell$, which 
might not be its minimal one.
\end{rem}

\begin{defi}
Given two quasiperiodic curves $\gamma$ and $\Gamma$ in $\R^3,$ we define\footnote{The P in $\dP$ stands for
parametric.} 
$$
d_{\mathcal{P}}(\Gamma,\gamma) := \inf_{p\in \mathcal{P}(L,\ell)} \sup_{s \in \R} \Big| \Gamma\big(s\big) -
 \gamma\big(p(s)
\big)\Big| =  \inf_{p \in \mathcal{P}(\ell,L)} \sup_{s \in \R} \Big| \gamma\big(s\big)-
\Gamma\big(p(s)\big) 
\Big|,
$$ 
where $\ell$ and $L$ are the periods of $\gamma$ and $\Gamma$ respectively, and for $0<s_0,s_1$,  
$$
\mathcal{P}(s_0,s_1) = \{ p\in \mathcal{C}(\R,\R) \text{ s.t. } p(s+s_0)=p(s)+s_1,\ \forall s\in
\R\}.
$$
\end{defi}
\noindent One may verify that $\dP(\Gamma,\gamma)<+\infty$ if and only if $\Gamma$ and $\gamma$ have the same pitch, and that
the restriction of $\dP$ to quasiperiodic curves in $\R^3$ of fixed pitch is a distance on that space. Also, 
$$
d_\mathcal{H}(\Gamma,\gamma) \leq \dP(\Gamma,\gamma),
$$  
where $d_\mathcal{H}$ is the Hausdorff distance.

\smallskip

In the sequel, $\gamma \in \mathcal{C}^2(\R,\R^3)$ is a quasiperiodic curve of period $\ell>0$ and pitch $a\in \R^3,$ and
$\Gamma$ is a quasiperiodic curve of period $L>0$ and with the same pitch as $\gamma.$ 
We denote by $r_\gamma$ the minimal radius of curvature of $\gamma$\, :
$$
r_\gamma := \left(\max_{s\in [0,\ell]}|\gamma''(s)|\right)^{-1}\:  \in (0,+\infty],
$$
and by $\mathcal{C}_\gamma$ the tubular neighborhood\, :
$$
\mathcal{C}_\gamma := \left\{ x \in \R^3,\: \text{ s.t. }\: d(x,\gamma) < r_\gamma/8\right\}.
$$

%%%%%%%%%%%%%%%%%%%%%%%%%%%%%%%%%%%%%%%%%%%%%%%%%%%%%%%%%%%%%%%%%%%%
\begin{lem}\label{lem:uniqprojloc}
Let $s_0\in \R$ and $x_0\in \R^3$ be such that $|x_0-\gamma(s_0)|<r_\gamma/8.$ There
exists a unique $\xi_0=:\xi(s_0,x_0)\in \R$ such that
\begin{enumerate}
\item
$\xi_0 \in (s_0-r_\gamma/4,s + r_\gamma/4)$,
\item
$\Big( x_0-\gamma\big(\xi_0\big) \Big) \cdot \gamma'\big(\xi_0\big) = 0$.
\end{enumerate}
Moreover, the function $\xi\,:\, \Xi_\gamma \to \R,\ (s,x) \mapsto \xi(s,x)$ defined on the open set
$$
\Xi_\gamma := \left\{ (s,x)\in \R^4 \text{ s.t. } |x-\gamma(s)|<r_\gamma/8\right\} 
$$
is of class $\mathcal{C}^1$ on $\Xi_\gamma$ and for every $(s,x)\in \Xi_\gamma$ we
have\footnote{One may believe at first sight that {\it 3} and {\it 4} are in contradiction if $a=0,$ but notice that
 the domain $\Xi_\gamma$ does not contain the whole segment joining $(s,x)$ and $(s+\ell,x).$}
\begin{enumerate}
\setcounter{enumi}{2}
\item
$(s+\ell,x+a)\in \Xi_\gamma$ and $\xi(s+\ell,x+a) = \xi(s,x) +
\ell,$ 
\item
$\partial_s\xi(s,x) = 0,$ 
\item
$D_x \xi(s,x) = \left(1 + \frac{(x-\gamma(\xi(s,x)))\cdot
\gamma''(\xi(s,x))}{1-(x-\gamma(\xi(s,x)))\cdot \gamma''(\xi(s,x))}\right) 
\gamma'(\xi(s,x)).
$
\end{enumerate}
\end{lem}

%%%%%%%%%%%%%%%%%%%%%%%%%%%%%%%%%%%%%%%%%%%%%%%%%%%%%%%%%%%%%%%%%%%%
\begin{cor}\label{cor:repar}
Assume that $\dP(\Gamma,\gamma)<r_\gamma/8$ and let $p \in \mathcal{P}(L,\ell)$ be such that
$$
\sup_{s\in \R}|\Gamma(s)-\gamma(p(s))| <
r_\gamma/8.
$$
The function $\sigma\,:\, \R \to \R, \ s\mapsto \sigma(s):=\xi(p(s),\Gamma(s))$ is the unique continuous
function which satisfies, for any $s$ in $\R,$
\begin{enumerate}
\item
$\sigma(s+L) = \sigma(s)+\ell$,
\item
$\Big| \Gamma\big(s\big)-\gamma\big(\sigma(s)\big)\Big|
< r_\gamma/8$, 
\item
$\Big( \Gamma\big(s\big)-\gamma\big(\sigma(s)\big) \Big) \cdot \gamma'\big(\sigma(s)\big) = 0$,
\item
$\big| \sigma(s) - p(s) \big| < r_\gamma/4.$
\end{enumerate}
\end{cor}

Any continuous function $\sigma$ that satisfies $1$, $2$ and $3$ in Corollary \ref{cor:repar} is called a 
reparametrization of $\gamma$ for $\Gamma.$ It also follows from Lemma \ref{lem:uniqprojloc} that

%%%%%%%%%%%%%%%%%%%%%%%%%%%%%%%%%%%%%%%%%%%%%%%%%%%%%%%%%%%%%%%%%%%%%%
\begin{cor}\label{cor:reparestim} Assume that $\dP(\Gamma,\gamma)<r_\gamma/8.$
Any reparametrization $\sigma$ of $\gamma$ for $\Gamma$ is a Lipschitz function on $\R$, and for almost every $s\in
\R$ we have the formula
\begin{equation}\label{eq:ravine0}
\sigma'(s) = \left(1 + \frac{(\Gamma(s)-\gamma(\sigma(s)))\cdot
\gamma''(\sigma(s))}{1-(\Gamma(s)-\gamma(\sigma(s)))\cdot \gamma''(\sigma(s))}\right)  \Gamma'(s) \cdot
\gamma'(\sigma(s)).
\end{equation}
In particular,
\begin{equation}\label{eq:ravine1}
\sigma'(s) \leq 1+ \frac{2}{r_\gamma} \left|\Gamma(s)-\gamma(\sigma(s))\right| \leq
1+ \frac{2}{r_\gamma}\|\Gamma-\gamma\circ\sigma\|_{\infty}.
\end{equation}
\end{cor}

When each point $x$ in $\Cg$ has a unique orthogonal projection $P(x)$ on $\gamma$ (this can happen only if 
$\gamma$ has no self-intersection), then there exist a unique reparametrization $\sigma$ of $\gamma$ for
$\Gamma$ (modulo a constant multiple of $\ell$ if $a=0$), and it is determined
 by
$$
\gamma(\sigma(s)) = P\big(\Gamma(s)\big).
$$   

We also have

%%%%%%%%%%%%%%%%%%%%%%%%%%%%%%%%%%%%%%%%%%%%%%%%%%%%%%%%%%%%%%%%%%%%%%%%%%%%%%%%%%%%%%%%%%%%%%%%%%%%%%%%%%%
\begin{lem}\label{lem:estimhaus}
If $\sigma$ is a reparametrization of $\gamma$ for $\Gamma$,
then 
$$
\sup_{s\in \R} \big| \sigma(s) - (s+\sigma(0))\big| \leq 2\tfrac{L}{r_\gamma}\, \| \Gamma -
\gamma\circ\sigma\|_\infty 
+ (L - \ell)_+, 
$$
where 
$(L - \ell)_+ := \max(L - \ell,0).$
\end{lem}

We turn now to a different measure of the discrepancy between $\Gamma$ and $\gamma$, which is better 
suited to our needs than $\dP.$ 

%%%%%%%%%%%%%%%%%%%%%%%%%%%%%%%%%%%%%%%%%%%%%%%%%%%%%%%%%%%%%%%%%%%%%%%%%%%
\begin{defi}\label{defi:Fglobal}
For $r>0$, we set
\begin{equation}\label{eq:defF}
F_{\gamma,r}(\Gamma):=\inf_{\sigma}F_{\gamma,\sigma,r}(\Gamma) \equiv \inf_{\sigma}\int_0^L
\mathcal{F}(\Gamma,\gamma,\sigma,r,s)\, ds,    
\end{equation}
where
$$
\mathcal{F}(\Gamma,\gamma,\sigma,r,s) :=  1 -
f \big( | \Gamma(s) - \gamma(\sigma(s)) |^2\big) \gamma'(\sigma(s))\cdot\Gamma'(s), 
$$
the function $f\equiv f_r \,:\, [0,+\infty) \to [0,+\infty)$ is given by 
$$
f(d^2) := \left\{ 
\begin{array}{ll}
\displaystyle 1- \big(\tfrac{d}{r}\big)^2, & \ \text{for } \,0\:\leq d^2 \leq r^2,\\
\displaystyle \ 0, & \ \text{for } d^2\geq r^2,
\end{array}\right.
$$
and the infimum in \eqref{eq:defF} is taken over all possible reparametrizations of $\gamma$ for $\Gamma.$  
\end{defi}
\noindent It follows from Corollary \ref{cor:repar} that $F_{\gamma,r}(\Gamma) < +\infty$ whenever 
$\dP(\Gamma,\gamma)<r_\gamma/8.$ On the other hand, 
 since there does not exist any reparametrization of $\gamma$ for $\Gamma$ in that case, 
one has $F_{\gamma,r}(\Gamma) = +\infty$ whenever $d_\mathcal{H}(\Gamma,\gamma)\geq r_\gamma/8.$  If
$F_{\gamma,r}(\Gamma)  <+\infty$ and each point $x$ in $\Cg$ has a unique orthogonal projection $P(x)$ on $\gamma$,  
one can rewrite $F_{\gamma,r}(\Gamma)$ in the slightly more geometric form
$$
F_{\gamma,r}(\Gamma) = \int_{\Gamma} 1 - f(d^2(x,\gamma)) \tau_\gamma(P(x))\cdot\tau_{\Gamma}(x) \, d\mathcal{H}^1(x),
$$
where $\tau_\Gamma$ and $\tau_\gamma$ denote the oriented tangent vectors to $\Gamma$ and $\gamma$ respectively. 
Notice finally that $0\leq f\leq 1$ and that $f(d^2)=1$ if and only if $d^2=0.$ In particular, the integrand in
\eqref{eq:defF} is non negative and moreover
$$
\mathcal{F}(\Gamma,\gamma,\sigma,r,s) = 0 \quad \Longleftrightarrow \quad \left( \Gamma(s)=\gamma(\sigma(s)) \text{ and } 
\Gamma'(s)=\gamma'(\sigma(s))\right).
$$

\medskip

The functional $F_{\gamma,r}$, which plays a central role in our analysis, can be related to
more standard discrepancy measures, like the distances $d_\mathcal{H}$ or $\dP$ between the curves, 
or the $L^2$ distance between the parametrizations of their tangents. We devote the remaining part of this section to such relations, which we
state one after the other before actually proving them.

\begin{lem}\label{lem:estimhausbis}
If $\sigma$ is a reparametrization of $\gamma$ for $\Gamma$,
then for $r\geq \|\Gamma-\gamma\circ\sigma\|_\infty$
$$
\| \Gamma -
\gamma\circ\sigma\|_\infty^2 \leq \left( \sqrt{2}r + \tfrac{r^2}{L}\right)\,
F_{\gamma,\sigma,r}(\Gamma).
$$
In particular,
$$
\dP^2(\Gamma,\gamma) \leq \left( \sqrt{2}r + \tfrac{r^2}{L}\right)\,
F_{\gamma,r}(\Gamma).
$$ 
\end{lem}

Combining Lemma \ref{lem:estimhaus} and Lemma \ref{lem:estimhausbis} we will obtain

%%%%%%%%%%%%%%%%%%%%%%%%%%%%%%%%%%%%%%%%%%%%%%%%%%%%%%%%%%%%%%%%%%%%%%%%%%%%%%%%%%%%%%%%%%%%%%%%%%%%%%%
\begin{lem}\label{lem:aller}
If $\sigma$ is a reparametrization of $\gamma$ for $\Gamma$,
then for $r\geq \|\Gamma-\gamma\circ\sigma\|_\infty$ 
$$
\int_0^L \big| \Gamma'(s) - \gamma'(\tfrac{\ell}{L}s+\sigma(0)) \big|^2\, ds \leq \left( 4 +
16\tfrac{L^3}{r_\gamma^4}\left(\sqrt{2}r+\tfrac{r^2}{L}\right) \right) F_{\gamma,\sigma,r}(\Gamma) +
16\tfrac{L}{r_\gamma^2}|L-\ell|^2.
$$
In particular, if $L=\ell$,
$$
\int_0^\ell \big| \Gamma'(s) - \gamma'(s+\sigma(0)) \big|^2\, ds \leq \left( 4 +
16\tfrac{\ell^3}{r_\gamma^4}\left( \sqrt{2}r+\tfrac{r^2}{\ell}\right) \right) F_{\gamma,\sigma,r}(\Gamma).
$$
\end{lem}

%Combining Lemma \ref{lem:estimhausbis} and Lemma \ref{lem:aller} we obtain

%%%%%%%%%%%%%%%%%%%%%%%%%%%%%%%%%%%%%%%%%%%%%%%%%%%%%%%%%%%%%%%%%%%%%%%%%%%%%%%%%%%%%%%%%%%%%%%%%%%%%%%%
%\begin{lem}\label{lem:dinosaure}
%We have
%$$
%\dP^2(\Gamma,\gamma) \leq \left( 16L + \big(2+ 64\tfrac{L^4}{r_\gamma^4}\big)\big( \sqrt{2}r+\tfrac{r^2}{L}\big) \right)
%F_{\gamma,r}(\Gamma) + \big(4+64\tfrac{L^2}{r_\gamma^2}\big)|L-\ell|^2.
%$$
%\end{lem}

Combining Corollary \ref{cor:repar}, the Cauchy-Schwarz inequality and Lemma \ref{lem:aller} we will obtain

%%%%%%%%%%%%%%%%%%%%%%%%%%%%%%%%%%%%%%%%%%%%%%%%%%%%%%%%%%%%%%%%%%%%%%%%%%%%%%%%%%%%%%%%%%%%%%%%%%%%%%%%
\begin{lem}\label{lem:retour}
Assume that 
\begin{equation}\label{eq:miles}
\Gamma(0)=\gamma(0) \quad\text{and}\quad \sqrt{L}\big(\int_0^L |\Gamma'(s)-
\gamma'(\tfrac{\ell}{L}s)|^2\, ds\Big)^{1/2}  + |L-\ell| < r_\gamma/8.
\end{equation}
Then $\dP(\Gamma,\gamma)<r_\gamma/8$, there exists a unique reparametrization $\sigma$ of $\gamma$ for $\Gamma$ 
 that satisfies $\sigma(0)=0$, and we have, for any $r>0,$  
$$
F_{\gamma,\sigma,r}(\Gamma) \leq \left(1+ 2\tfrac{L^2}{r^2} + 16\tfrac{L^4}{r_\gamma^4} \right)\int_0^L |\Gamma'(s)-
\gamma'(\tfrac{\ell}{L}s)|^2\, ds +
\left( \tfrac{2}{r^2} + \tfrac{8}{r_\gamma^2} + 16\tfrac{L^2}{r_\gamma^4}\right)L|L-\ell|^2.
$$
\end{lem}

To complete this section, we now provide the proofs of Lemmas \ref{lem:uniqprojloc} to \ref{lem:retour} and
Corollaries \ref{cor:repar} and \ref{cor:reparestim}.   

\medskip 

%%%%%%%%%%%%%%%%%%%%%%%%%%%%%%%%%%%%%%%%%%%%%%%%%%%%%%%%%%%%%%%%%%%%%%%%%%%%%%%%%%%%%%%%%%%%%%%%%%%%%%
\noindent{\bf Proof of Lemma \ref{lem:uniqprojloc}.} If $r_\gamma = +\infty$, then $\gamma$ is a straight line and the conclusion is straightforward. We assume next
 that $r_\gamma \in (0,+\infty).$  

\noindent{\bf Step 1: Existence.}
 If $(x_0-\gamma(s_0)) \cdot \gamma'(s_0) = 0,$ 
it suffices to choose $\xi_0=s_0.$ If not, assume that $(x_0-\gamma(s_0)) \cdot \gamma'(s_0) > 0,$
the other case being treated in a very similar manner. Define $h\,:\, \R \to [0,+\infty)$ 
by $h(z) :=
|\gamma(z)-x_0|.$ 
For $z\in \R$ verifying $h(z)\neq 0$, 
we compute
$$
h'(z) = \gamma'(z) \cdot \frac{\gamma(z)-x_0}{|\gamma(z)-x_0|}
$$
and
\begin{equation}\label{eq:hseconde}
h''(z) = \gamma''(z) \cdot \frac{\gamma(z)-x_0}{|\gamma(z)-x_0|} + \frac{1}{|\gamma(z)-x_0|} -
\frac{\big((\gamma(z)-x_0)\cdot\gamma'(z)\big)^2}{|\gamma(z)-x_0|^3}.
\end{equation}
By assumption, we have $h(s_0)>0$ and $h'(s_0)<0.$ If $h(s_1)=0$ for some $s_1\in
[s_0,s_0+r_\gamma/4),$ it suffices to choose $\xi_0:=s_1.$ Assume next that $h(z)\neq 0$ for all $z\in
[s_0,s_0+r_\gamma/4).$ If $h'(z)<-1/2$ for all $z\in [s_0,s_0+r_\gamma/4],$
then we would have $h(s_0+r_\gamma/4) < h(s_0)-r_\gamma/8 <0,$ which is impossible. Let therefore $s_2 := \min\{
z\geq s_0,\, h'(z)\geq -1/2\}<s_0+r_\gamma/4.$ We have $h(s_2) \leq h(s_0)-\frac{1}{2}(s_2-s_0)$,
and $h'(s_2)=-\frac{1}{2}$ if $s_2\neq s_0$ or $-\frac{1}{2}\leq h'(s_2) <0$ if $s_2=s_0.$ We
claim that there exits $s_3\in[s_2,s_0+r_\gamma/4)$ such that $h'(s_3)=0,$ in which case it suffices to
choose  $\xi_0:=s_3.$ Assume by contradiction that 
$h'(z)\neq 0$ for any $z\in [s_2,s_0+r_\gamma/4),$ so that $h'(z)< 0$ for all $z\in
[s_2,s_0+r_\gamma/4),$ and in particular $h(z) \leq h(s_2) <r_\gamma/8,$
for all $z\in [s_2,s_0+r_\gamma/4).$ For $z \in [s_2,s_0+r_\gamma/4)$ such that $h'(z)\geq
-\frac{1}{2}$, since we also have by assumption $h'(z)<0$, 
we infer from \eqref{eq:hseconde} and the definition of $r_\gamma$ that
\begin{equation}\label{eq:chocolat1}
h''(z)\geq  - \frac{1}{r_\gamma} + \frac{1}{h(z)} - \frac{1}{4h(z)} > \frac{1}{2h(z)} \geq
\frac{1}{2h(s_2)}>0.
\end{equation}
Since $h'(s_2)\geq -1/2,$ \eqref{eq:chocolat1} first implies that in fact $h'(z)\geq -\frac{1}{2}$ for 
all $z\in [s_2,s_0+r_\gamma/4),$ 
and then that
$$
h'(s_0+r_\gamma/4) > h'(s_2) + \frac{1}{2h(s_2)}(s_0+r_\gamma/4-s_2) \geq -\frac{1}{2} + 
\frac{s_0+r_\gamma/4-s_2}{2\big(h(s_0)-(s_2-s_0)/2\big)} > 0, 
$$
since $h(s_0)<r_\gamma/8.$ This is not compatible with the fact that $h'(z)<0$ for all 
$z\in [s_2,s_0+r_\gamma/4),$ and we therefore have 
reached our contradiction. 

\smallskip
\noindent{\bf Step 2: Uniqueness.} Let $s \in (s_0-r_\gamma/4,s_0+r_\gamma/4)$ be such
that 
\begin{equation}\label{eq:gh1}
\big( x_0-\gamma(s) \big) \cdot \gamma'(s) = 0.
\end{equation}
On the one hand, if $h(s)\neq 0,$ then $h'(s)=0$, and therefore from \eqref{eq:hseconde} 
$$
h''(s) = \gamma''(s) \cdot \frac{\gamma(z)-x}{|\gamma(s)-x|} + \frac{1}{|\gamma(s)-x|} \geq -\frac{1}{r_\gamma} +
\frac{2}{r_\gamma } > 0.
$$
On the other hand, if $h(s)=0$ then $\lim_{z\to s,z>s}h'(z)=1.$ 
In both cases, we infer that $h'(z)>0$ for all $z>s$ in a sufficiently small neighborhood of $s.$ 
 Let $s^+ := \sup\{u >s,\, h'(z)>0 \text{ for } z\in(s,u)\}>s.$ If $s^+\neq +\infty$, then we must have
$h'(s^+)=0$ 
and $h''(s^+)\leq 0.$ Therefore
$$
0 \geq h''(s^+) =  \gamma''(s^+) \cdot \frac{\gamma(s^+)-x_0}{|\gamma(s^+)-x_0|} + \frac{1}{|\gamma(s^+)-x_0|} \geq -\frac{1}{r_\gamma} +
\frac{1}{|\gamma(s^+)-x_0|},
$$  
so that $ |\gamma(s^+)-x_0|\geq r_\gamma$. Since $s \in (s_0-r_\gamma/4,s_0+r_\gamma/4)$, since $\gamma$ is 
parametrized by arc-length, and since $|\gamma(s_0)-x_0|<r_\gamma/8$, we first infer that 
$|\gamma(s)-x_0)|<3r_\gamma/8$, and next that $s^+\geq s + 5r_\gamma/8.$ Hence $h(z)>h(s)$ for any $z\in
(s,s+5r_\gamma/8).$ A parallel argument leads to the conclusion that $h(z)>h(s)$ for any $z\in
(s-5r_\gamma/8,s).$ If $\tilde s \in (s_0-r_\gamma/4,s_0+r_\gamma/4)$ were such that $\tilde s \neq
s$ and $\tilde s$ is a solution to \eqref{eq:gh1} with $s$ replaced by $\tilde s$,
then we would infer, since then $|\tilde s-s|< r_\gamma/2<5r_\gamma/8$, that
$h(\tilde s)>h(s)$ and $h(s)>h(\tilde s),$ which is not possible.

\smallskip
\noindent{\bf Step 3: Regularity and Periodicity.}
In Step 1 and Step 2, we have constructed a function $\xi$ that satisfies and is uniquely defined by $\it 1$ 
and $\it 2$ for any $s\in \R.$ Condition $\it 3$ also follows from the quasiperiodicity assumption on $\gamma$ 
and the uniqueness of Step 2. Concerning regularity, let $(s_0,x_0) \in \Xi_\gamma:$ The function
$$
\psi\,:\, \Xi_\gamma \times \R \to \R,\ (s,x,z) \mapsto \big(x - \gamma(z)\big)\cdot \gamma'(z)
$$
vanishes at $(s,x,z)=(s_0,x_0,\xi(s_0,x_0))$ and we have 
$$
\partial_{z} \psi(s_0,x_0,\xi(s_0,x_0)) = -|\gamma'(\xi(s_0,x_0))|^2 + \big(x_0 - \gamma(\xi(s_0,x_0))\big)\cdot
\gamma''(\xi(s_0,x_0)) \leq -1 + \frac{r_\gamma}{8}\,\frac{1}{r_\gamma} < 0.
$$  
It follows from the implicit function theorem that there exists a continuously differentiable function $\tilde\xi,$ defined in
a neighborhood of $(s_0,x_0)$ in $\Xi_\gamma$ such that $\psi(s,x,\tilde\xi(s,x))\equiv 0.$ From Step 2, we 
infer that $\tilde\xi = \xi.$ The implicit function theorem also yields, differentiating the relation
$$
\big(x - \gamma(\xi(s,x))\big)\cdot \gamma'(\xi(s,x)) = 0
$$
with respect to $s$ and $x$ respectively, the formulas $\it 4$ and $\it 5.$ \qed

%%%%%%%%%%%%%%%%%%%%%%%%%%%%%%%%%%%%%%%%%%%%%%%%%%%%%%%%%%%%%%%%%%%%%%%%%%%%%%%%%%%%%%%%%%%%%%%%%%%
\noindent{\bf Proof of  Corollaries \ref{cor:repar} and \ref{cor:reparestim}.}

Corollary \ref{cor:repar} is straightforward. For the regularity stated in Corollary \ref{cor:reparestim},
notice that if $\sigma$ is any reparametrization of $\gamma$ for $\Gamma$ then for all $s\in \R$ we 
have $\sigma(s)= \xi(\sigma(s),\Gamma(s)).$ Since the function $\xi$ is locally constant in $s$ and
continuously differentiable, and since $\Gamma$ is Lipschitzian, it follows that $\sigma$ is Lipschitzian. 
Formula \eqref{eq:ravine0} is then a direct consequence of $\it 5$ in Lemma \ref{lem:uniqprojloc} and the
chain rule, and inequality \eqref{eq:ravine1} follows from \eqref{eq:ravine0}, $\it 2$ in Corollary \ref{cor:repar}
and the definition of $r_\gamma.$ \qed

\medskip

%%%%%%%%%%%%%%%%%%%%%%%%%%%%%%%%%%%%%%%%%%%%%%%%%%%%%%%%%%%%%%%%%%%%%%%%%%%%%%%%%%%%%%%%%%%%%%%%%%%
\noindent{\bf Proof of Lemma \ref{lem:estimhaus}.}
From \eqref{eq:ravine1}, we deduce in particular that for $s\in [0,L],$
\begin{equation}\label{eq:wagner1}
\sigma(s) \leq \sigma(0) + s + \frac{2}{r_\gamma} \|\Gamma - \gamma\circ\sigma\|_\infty s,
\end{equation}
and 
\begin{equation}\label{eq:wagner2}
\sigma(L) \leq \sigma(s) + (L-s) + \frac{2}{r_\gamma} \|\Gamma - \gamma\circ\sigma\|_\infty(L- s).
\end{equation}
By assumption on $\sigma$, we have $\sigma(L) = \sigma(0)+\ell,$
and therefore we deduce from \eqref{eq:wagner2} that
\begin{equation}\label{eq:wagner3}
\sigma(s) \geq \sigma(0) + s + (\ell-L) - \frac{2}{r_\gamma} \|\Gamma - \gamma\circ\sigma\|_\infty(L- s).
\end{equation}
Combining \eqref{eq:wagner1} and \eqref{eq:wagner3} the conclusion follows. 
\qed
\medskip

%%%%%%%%%%%%%%%%%%%%%%%%%%%%%%%%%%%%%%%%%%%%%%%%%%%%%%%%%%%%%%%%%%%%%%%%%%%%%%%%%%%%%%%%%%%%%%%%%%%
\noindent{\bf Proof of Lemma \ref{lem:estimhausbis}.}
For almost every  $s \in \R$ such that $\Gamma'(s) \cdot \gamma'(\sigma(s)) \geq 0,$ we have
\begin{equation*}
\begin{split}
\mathcal{F}(\Gamma,\gamma,\sigma,r,s) &\geq 1 - \Gamma'(s)\cdot \gamma'(\sigma(s)) \\
&=
\frac{|\Gamma'(s)|^2}{2}+\frac{|\gamma'(\sigma(s))|^2}{2} - \Gamma'(s)\cdot
\gamma'(\sigma(s))\\
&= \frac{1}{2}|\Gamma'(s)-\gamma'(\sigma(s))|^2, 
\end{split}
\end{equation*}
while for almost every $s \in \R$ such that $\Gamma'(s) \cdot \gamma'(\sigma(s)) < 0,$ we have
\begin{equation*}
\mathcal{F}(\Gamma,\gamma,\sigma,r,s) \geq 1 \geq  \frac{1}{2}|\Gamma'(s)-\gamma'(\sigma(s))|^2.
\end{equation*}
It follows in particular from the Pythagorean theorem that for almost every $s\in \R$,
\begin{equation}\label{eq:mozart}
\mathcal{F}(\Gamma,\gamma,\sigma,r,s) \geq \frac{1}{2}|\Gamma'(s)-\gamma'(\sigma(s))|^2\geq  \frac{1}{2} \big|
Q_{s}(\Gamma'(s))\big|^2,
\end{equation}
where $Q_{s}$ denotes the orthogonal projection onto the plane
$[\gamma'(\sigma(s))]^\perp.$
For any $s\in \R$, we also have, since $r\geq \|\Gamma-\gamma\circ\sigma\|_\infty,$ 
\begin{equation}\label{eq:verdi}
\mathcal{F}(\Gamma,\gamma,\sigma,r,s) \geq 1 - f(|\Gamma(s)-\gamma(\sigma(s))|^2) \geq \frac{1}{r^2}
|\Gamma(s)-\gamma(\sigma(s))|^2.
\end{equation}
Combining \eqref{eq:mozart} and \eqref{eq:verdi}, we are led, for almost every $s\in \R,$ to
\begin{equation}\label{eq:nabuco}
\mathcal{F}(\Gamma,\gamma,\sigma,r,s) \geq \frac{1}{\sqrt{2}r}  \big|Q_{s}(\Gamma'(s))\big|\,
\big|\Gamma(s)-\gamma(\sigma(s))\big| \geq
\frac{1}{2\sqrt{2}r} \big| \partial_s |\Gamma(s)-\gamma(\sigma(s))|^2 \big|.
\end{equation}
Taking into account the $L$-periodicity of the map $s\mapsto \Gamma(s)-\gamma(\sigma(s))$, integration of
\eqref{eq:nabuco} yields    
\begin{equation}\label{eq:saintsaens}
{\rm osc} \Big( \big|\Gamma(\cdot)-\gamma\circ\sigma(\cdot)\big|^2 \Big) \leq \sqrt{2}r F_{\gamma,\sigma,r}(\Gamma).
\end{equation}
Integration of \eqref{eq:verdi} and the Tchebichev inequality yield
\begin{equation}\label{eq:ravel}
\inf_{s\in \R} \big|\Gamma(s)-\gamma(\sigma(s))\big|^2 \leq \frac{r^2}{L} F_{\gamma,\sigma,r}(\Gamma),
\end{equation}
and combining \eqref{eq:saintsaens} and \eqref{eq:ravel} we therefore obtain
\begin{equation*}\label{eq:strauss}
\sup_{s\in \R} \big|\Gamma(s)-\gamma(\sigma(s))\big|^2 \leq \left( \sqrt{2}r + \frac{r^2}{L}\right)
F_{\gamma,\sigma,r}(\Gamma),
\end{equation*}
which corresponds to the first claim.

For the second claim, notice that if $F_{\gamma,r}(\Gamma)=+\infty$ there is nothing to prove, while if
 $F_{\gamma,r}(\Gamma)<+\infty$ it suffices to use the first claim and then to minimize over all the possible
reparametrizations of $\gamma$ for $\Gamma.$      
\qed

\medskip 

%%%%%%%%%%%%%%%%%%%%%%%%%%%%%%%%%%%%%%%%%%%%%%%%%%%%%%%%%%%%%%%%%%%%%%%%%%%%%%%%%%%%%%%%%%%%%%%%%%%
\noindent{\bf Proof of Lemma \ref{lem:aller}.}
Integration of \eqref{eq:mozart} for $s\in [0,L]$ yields
$$%\begin{equation}\label{eq:faure1}
\int_0^L |\Gamma'(s)-\gamma'(\sigma(s))|^2 \, ds \leq 2 F_{\gamma,\sigma,r}(\Gamma),
$$%\end{equation} 
so that
\begin{equation}\label{eq:faure2}
\int_0^L |\Gamma'(s)-\gamma'(\sigma(0)+\tfrac{\ell}{L}s)|^2 \, ds \leq 4 F_{\gamma,\sigma,r}(\Gamma) +
2\int_0^L |\gamma'(\sigma(s))-\gamma'(\sigma(0)+\tfrac{\ell}{L}s)|^2 \, ds.
\end{equation}
By Lemma \ref{lem:estimhaus}, we deduce
$$
\sup_{s\in [0,L]} |\sigma(s) - (\sigma(0)+\tfrac{\ell}{L}s)| \leq
\tfrac{2L}{r_\gamma}\|\Gamma-\gamma\circ\sigma\|_\infty + 2 |L-\ell|,
$$ 
so that
\begin{equation}\label{eq:faure3}
\int_0^L |\gamma'(\sigma(s))-\gamma'(\sigma(0)+\tfrac{\ell}{L}s)|^2 \, ds\leq L \left(
\tfrac{2L}{r_\gamma}\|\Gamma-\gamma\circ\sigma\|_\infty + 2 |L-\ell| \right)^2 \|\gamma''\|_\infty^2. 
\end{equation}
By Lemma \ref{lem:estimhausbis} and replacing $\|\gamma''\|_\infty$ by $1/r_\gamma$, we
deduce from \eqref{eq:faure3}
\begin{equation}\label{eq:faure4}
\int_0^L |\gamma'(\sigma(s))-\gamma'(\sigma(0)+\tfrac{\ell}{L}s)|^2 \, ds\leq \tfrac{8L}{r_\gamma^2} \left(
\tfrac{L^2}{r_\gamma^2}\left( \sqrt{2}r + \tfrac{r^2}{L}\right)\, F_{\gamma,\sigma,r}(\Gamma) + |L-\ell|^2 \right). 
\end{equation}
Combining \eqref{eq:faure2} and \eqref{eq:faure4} the conclusion follows.\qed

\medskip

%%%%%%%%%%%%%%%%%%%%%%%%%%%%%%%%%%%%%%%%%%%%%%%%%%%%%%%%%%%%%%%%%%%%%%%%%%%%%%%%%%%%%%%%%%%%%%%%%%%%%
%\noindent{\bf Proof of Lemma \ref{lem:dinosaure}.} 
%If $F_{\gamma,r}(\Gamma)=+\infty$ there is nothing to prove. If $\sigma$ is a reparametrization of $\gamma$ for
%$\Gamma,$ we write, for $s\in \R,$
%\begin{equation*}\begin{split}
%|\Gamma(s) - \gamma(\sigma(0)+\tfrac{\ell}{L}s)|^2 &\leq \left( \big|\Gamma(0)-\gamma(\sigma(0))\big| +
% \int_0^s |\Gamma'(z) -
%\tfrac{\ell}{L}\gamma'(\sigma(0)+\tfrac{\ell}{L}z) |\, dz\right)^2\\ 
%&\leq 2\|\Gamma-\gamma\circ\sigma\|_\infty^2 + 4L \int_0^s |\Gamma'(z) -
%\gamma'(\sigma(0)+\tfrac{\ell}{L}z) |^2\, dz + 4|L-\ell|^2. 
%\end{split}\end{equation*}
%From Lemma \ref{lem:estimhausbis}, Lemma \ref{lem:aller}, and the definition of $\dP,$ we therefore obtain
%$$
%\dP^2(\Gamma,\gamma) \leq \left( 16L + \big(2+ 64\tfrac{L^4}{r_\gamma^4}\big)\big( \sqrt{2}r+\tfrac{r^2}{L}\big) \right)
%F_{\gamma,r}(\Gamma) + \big(4+64\tfrac{L^2}{r_\gamma^2}\big)|L-\ell|^2.
%$$
%Minimizing over $\sigma$ the conclusion follows.
%\medskip

%%%%%%%%%%%%%%%%%%%%%%%%%%%%%%%%%%%%%%%%%%%%%%%%%%%%%%%%%%%%%%%%%%%%%%%%%%%%%%%%%%%%%%%%%%%%%%%%%%%%%
\noindent{\bf Proof of Lemma \ref{lem:retour}.} 
Set 
$$
G:= \int_0^L |\Gamma'(s) - \gamma'(\tfrac{\ell}{L}s)|^2\, ds.
$$
By Cauchy-Schwarz inequality and \eqref{eq:miles}, we have, for any $\tau \in [0,L]$,
\begin{equation*}\label{eq:tra1}\begin{split}
\left| \Gamma(\tau) - \gamma(\tfrac{\ell}{L}\tau)\right| &\leq  \int_0^\tau |\Gamma'(s)-
\gamma'(\tfrac{\ell}{L}s)|\, ds + |1-\tfrac{\ell}{L}|\int_0^\tau |\gamma'(\tfrac{\ell}{L}s)|\, ds\\
&\leq \sqrt{L G} + |L-\ell|<r_\gamma/8.
\end{split}\end{equation*}
Hence $\dP(\Gamma,\gamma)\leq\sqrt{L G} + |L-\ell|<r_\gamma/8.$  Moreover, the unique
reparametrization $\sigma$ of $\gamma$ for $\Gamma$ given by Corollary \ref{cor:repar} for 
the choice $p(s)=Ls/\ell$ satisfies $\sigma(0)=0,$ and we have 
\begin{equation}\label{eq:tra2}
\|\Gamma-\gamma\circ\sigma\|_\infty \leq \sqrt{L G} + |L-\ell|.
\end{equation} 
From Lemma \ref{lem:estimhaus} and \eqref{eq:tra2}, we obtain
\begin{equation}\label{eq:tra3}
\sup_{s\in [0,L]} |\sigma(s) - s| \leq \tfrac{2L}{r_\gamma}\,\left( \sqrt{L G} + |L-\ell|\right)
+ (L - \ell)_+. 
\end{equation}
We write, for almost every $s\in [0,L],$
\begin{equation}\label{eq:tra4}\begin{split}
\mathcal{F}(\Gamma,\gamma,\sigma,r,s) &= 1 - \Gamma'(s)\cdot\gamma'(\sigma(s)) +
(1-f(|\Gamma(s)-\gamma(\sigma(s))|^2))\Gamma'(s)\cdot\gamma'(\sigma(s))\\
& \leq \frac{1}{2}|\Gamma'(s)-\gamma'(\sigma(s))|^2 + \tfrac{1}{r^2}|\Gamma(s)-\gamma(\sigma(s))|^2\\
&\leq |\Gamma'(s)-\gamma'(\tfrac{\ell}{L}s)|^2 +
|\gamma'(\sigma(s))-\gamma'(\tfrac{\ell}{L}s)|^2  +
\tfrac{1}{r^2}\|\Gamma-\gamma\circ\sigma\|^2_\infty\\
&=: |\Gamma'(s)-\gamma'(\tfrac{\ell}{L}s)|^2 + R.
\end{split}\end{equation}
By \eqref{eq:tra2} and \eqref{eq:tra3}, we infer that
\begin{equation}\label{eq:tra5}\begin{split}
R\ &\leq (\sup_{\tau\in [0,L]}
\big|\sigma(\tau)-\tfrac{\ell}{L}\tau\big|)^2\|\gamma''\|_\infty^2  + \tfrac{1}{r^2}(\sqrt{L G} +
|L-\ell|)^2\\
&\leq \tfrac{1}{r_\gamma^2}\left(\tfrac{2L}{r_\gamma}\,\Big( \sqrt{L G} + |L-\ell|\Big)
+ 2|L - \ell|\right)^2  + \tfrac{1}{r^2}\left(\sqrt{L G} + |L-\ell|\right)^2\\
&\leq  \left(\tfrac{2}{r^2} + 16\tfrac{L^2}{r_\gamma^4}\right)LG +
\left( \tfrac{2}{r^2} + \tfrac{8}{r_\gamma^2} + 16\tfrac{L^2}{r_\gamma^4}\right) |L-\ell|^2.
\end{split}\end{equation}
Integrating \eqref{eq:tra4} for $s\in [0,L]$ and using \eqref{eq:tra5}, we are led to
\begin{equation*}\label{eq:think}
F_{\gamma,\sigma,r}(\Gamma) \leq \left(1 + 2\tfrac{L^2}{r^2} + 16\tfrac{L^4}{r_\gamma^4}\right)G 
+\left( \tfrac{2}{r^2} + \tfrac{8}{r_\gamma^2} + 16\tfrac{L^2}{r_\gamma^4}\right)L|L-\ell|^2,
\end{equation*}
and the proof is complete.\qed
%%%%%%%%%%%%%%%%%%%%%%%%%%%%%%%%%%%%%%%%%%%%%%%%%%%%%%%%%%%%%%%%%%%%%%%%%%%%%%%%%%%%%%%%%%%%%%%
\section{Integral Formula for parametrized binormal curvature flows}\label{sect:weakp}
%%%%%%%%%%%%%%%%%%%%%%%%%%%%%%%%%%%%%%%%%%%%%%%%%%%%%%%%%%%%%%%%%%%%%%%%%%%%%%%%%%%%%%%%%%%%%%%

The next proposition is the second main ingredient in our proof of Theorem \ref{thm:main}.
A very similar result\footnote{Without an $s$ dependence for the vector field $X$ but 
valid for binormal curvature flows of co-dimension $2$ in arbitrary dimension $N\geq 3$.} was obtained 
ten years ago by the first author in \cite{Je1}, where it was also used
to define a weak notion of binormal curvature flow for objects that are only rectifiable currents 
of co-dimension 2 (notice than \eqref{eq:weakbfp} only
involves first order derivatives of $\gamma$).  In \cite{JeSm2} we provide an existence theory for
a notion of binormal curvature flow related to the one introduced in \cite{Je1}, and also strongly based
on formula \eqref{eq:weakbfp}. 

\begin{prop}\label{prop:weakp}
Let $I\subset \R$ be some open interval and let $\Gamma \in L^\infty(I, H^\frac32_{\rm loc}(\R,\R^3))$ be a solution
of the binormal curvature flow equation on $I\times \R.$ Assume moreover that $\Gamma$ is 
$L$-quasiperiodic in the sense that  
\begin{equation}\label{eq:perio0}
\Gamma(t,s+L) = \Gamma(t,s) + a,\qquad\forall t\in I,\, \forall\, s\in \R,   
\end{equation}
for some period $L>0$ and pitch $a\in \R^3.$ 

Then for any parametrized vector field $X \in \mathcal{C}^1(I\times \R, \mathcal{C}^{2}(\R^3,\R^3))$ 
such that
\begin{align}
\label{eq:ced1}
&X(t,s+L,x+a) = X(t,s,x)  
&\ &\forall t\in I,\, s\in \R,\text{ and }x \in \R^3, \\
\label{eq:ced2}
&\partial_s X(t,s,\Gamma(t,s))= \partial_{si}X(t,s,\Gamma(t,s))= 0 
&\ & \forall i\in\{1,2,3\},\, t\in I, \text{ and }s\in \R,
\end{align}
we have
\begin{equation}
\label{eq:weakbfp}
\frac{d}{dt} \int_0^L \left(X\circ \Gamma\right) \cdot \partial_s\Gamma \, ds  
= \int_0^L \left(\partial_t  X \circ \Gamma\right) \cdot \partial_s\Gamma - \left(D\, {\rm curl} X \circ
\Gamma\right)  : \left( \partial_s\Gamma \otimes
\partial_s\Gamma\right)
\, ds
\end{equation} 
in the sense of distributions on $I.$
\end{prop}
\noindent In \eqref{eq:weakbfp}, by abuse of notations we have used the expression $X\circ \Gamma$ to denote the function 
$(t,s)\mapsto X(t,s,\Gamma(t,s))$, and similarly for $\partial_s X \circ \Gamma$ and $D\,{\rm curl}X\circ \Gamma$,
where $D$ and ${\rm curl}$ apply to the $x$ variable only.

\begin{proof}
\noindent{\bf Step 1.} Assume that $\Gamma$ is a smooth solution. We then write in coordinates, and with the same
abuse of notations as the one mentioned here above, 
\begin{equation*}\begin{split}
\frac{d}{dt} &\int_0^L (X\circ \Gamma)\cdot \partial_s\Gamma\, ds 
 - \int_0^L \left(\partial_t  X \circ \Gamma\right) \cdot \partial_s\Gamma\, ds\\
&=  \int_0^L  \sum_{i,j = 1}^3 (\partial_iX^j \circ \Gamma) \partial_t \Gamma^i  \partial_s\Gamma^j\, ds
+\int_0^L \sum_{j= 1}^3 (X^j\circ \Gamma)\partial_{st}\Gamma^j\, ds\\ 
&= \int_0^L \sum_{i,j = 1}^3 (\partial_iX^j \circ \Gamma) \left( \partial_t \Gamma^i  \partial_s\Gamma^j -
\partial_s \Gamma^i \partial_t \Gamma^j\right)\, ds - \int_0^L (\partial_sX\circ \Gamma) \cdot
\partial_t\Gamma\,ds + \Big[(X\circ\Gamma)\cdot \partial_t\Gamma \Big]_0^L\\
&= \int_0^L \sum_{i,j = 1}^3 (\partial_iX^j \circ \Gamma) \left( \partial_t \Gamma^i  \partial_s\Gamma^j -
\partial_s \Gamma^i \partial_t \Gamma^j\right)\, ds, 
\end{split}
\end{equation*}
where we have used \eqref{eq:perio0}, \eqref{eq:ced1} and \eqref{eq:ced2} for the last identity.   
By definition of the vector product, 
$$
\left( \partial_t \Gamma^i  \partial_s\Gamma^j - \partial_t \Gamma^j \partial_s
\Gamma^i\right) = \sum_{k=1}^3 \ep_{ijk} \left( \partial_t \Gamma \times \partial_s \Gamma\right)^k
$$
where $\ep_{ijk}$ is the Levi-Civita symbol, therefore  
\begin{equation}\label{eq:ep1}
\frac{d}{dt} \int_0^L (X\circ \Gamma)\cdot \partial_s\Gamma\, ds 
 = \int_0^L \left(\partial_t  X \circ \Gamma\right) \cdot \partial_s\Gamma \, ds
+ \sum_{i,j,k=1}^3 \ep_{ijk}\int_0^L (\partial_iX^j \circ \Gamma)\left( \partial_t \Gamma \times
\partial_s \Gamma\right)^k\,ds.
\end{equation}
Next, we also write in coordinates
$$
\int_0^L \left( D\, {\rm curl} X \circ \Gamma\right) : \left(
\partial_s\Gamma \otimes \partial_s\Gamma\right) \, ds 
= 
\int_0^L \sum_{k,l=1}^3 (\partial_l({\rm curl} X)^k)\circ \Gamma) \partial_s\Gamma^l \partial_s\Gamma^k\, ds.
$$
By definition of the rotational and the chain rule, 
\begin{equation*}\begin{split}
\int_0^L \sum_{k,l=1}^3 (\partial_l({\rm curl} X)^k)\circ \Gamma) \partial_s\Gamma^l \partial_s\Gamma^k\, ds
&
= \sum_{i,j,k,l=1}^3 \ep_{ijk}\int_0^L  (\partial_{il}X^j\circ \Gamma)
\partial_s\Gamma^l \partial_s\Gamma^k\, ds\\
= \sum_{i,j,k=1}^3 \ep_{ijk}\int_0^L \Big( \partial_s\big(\partial_{i}X^j\circ
\Gamma\big)-\partial_{si}X^j\circ\Gamma\Big) \partial_s\Gamma^k\, ds
&
=\sum_{i,j,k=1}^3 \ep_{ijk}\int_0^L \partial_s(\partial_{i}X^j)\circ
\Gamma)\partial_s\Gamma^k\, ds
\end{split}
\end{equation*}
where we have used \eqref{eq:ced2} for the last identity.
Integration by parts and \eqref{eq:perio0},\eqref{eq:ced1} therefore yields
\begin{equation}\label{eq:ep2}
\int_0^L \left( D\,  {\rm curl} X\circ \Gamma \right): \left(
\partial_s\Gamma \otimes \partial_s\Gamma\right) \, ds = - \sum_{i,j,k=1}^3 \ep_{ijk}\int_0^L  
(\partial_{i}X^j\circ \Gamma) \partial_{ss}\Gamma^k\, ds.
\end{equation}  
Combining \eqref{eq:ep1} and \eqref{eq:ep2}, and noticing 
that $\partial_t \Gamma \times \partial_s \Gamma = \partial_{ss}\Gamma$ whenever $\Gamma$ solves the binormal
curvature flow equation, the conclusion follows.

\medskip
\noindent{\bf Step 2.} The general case. We consider a family of space-time radially-symmetric mollifiers $(\rho_\eps)_{\eps>0}$ and
define the smooth functions $\Gamma_\eps := \rho_\eps * \Gamma.$ Even though the $\Gamma_\eps$ are no longer
solutions of the binormal curvature flow equation, we may reproduce the computations of Step 1 for $\Gamma_\eps$ up
 to before the conclusion, so as to obtain
\begin{equation}\label{eq:preconclu}
\frac{d}{dt} \int_0^L\!\! (X\circ \Gamma_\eps)\cdot \partial_s\Gamma_\eps\, ds 
= \int_0^L \!\!\left(\partial_t  X \circ \Gamma_\eps\right) \cdot \partial_s\Gamma_\eps\, ds 
- \int_0^L \!\!\left( D\,  {\rm curl} X\circ \Gamma_\eps \right): \left(
\partial_s\Gamma_\eps \otimes \partial_s\Gamma_\eps\right) \, ds + R_\eps 
\end{equation}   
where
$$
R_\eps := \sum_{i,j,k=1}^3 \ep_{ijk}\int_0^L (\partial_iX^j \circ \Gamma_\eps)\left( \partial_t \Gamma_\eps \times
\partial_s \Gamma_\eps - \partial_{ss}\Gamma_\eps\right)^k\,ds.
$$
Using the (local in time) uniform convergence of $\Gamma_\eps$ to $\Gamma$, and the 
strong convergence in (e.g.) $\mathcal{C}_{\rm loc}(I,L^2_{\rm loc}(\R,\R^3))$ of $\partial_s\Gamma_\eps$ to $\partial_s\Gamma$, 
 we can pass to the limit in the left-hand side of \eqref{eq:preconclu} in the sense of
distributions on $I$, as well as in the first and second terms 
on the right-hand side of \eqref{eq:preconclu} (pointwise in that case), and get the equivalent terms where
$\Gamma_\eps$ is replaced by $\Gamma.$ To conclude, we need 
therefore to prove that $R_\eps$ converges to zero in the
sense of distributions on $I,$ as $\eps\to 0.$  First notice that since $\partial_{ss}\Gamma_\eps$ converges 
in (e.g.) $\mathcal{C}_{\rm loc}(I,H^{-1}_{\rm loc}(\R,\R^3))$ to $\partial_{ss}\Gamma$ and since
 $\partial_iX^j \circ \Gamma_\eps$ converges in (e.g.) $\mathcal{C}_{\rm loc}(I,H^{1}_{\rm loc}(\R,\R^3))$ to 
$\partial_iX^j \circ \Gamma$ (indeed $\partial_iX^j$ is Lipschitzian), we have
\begin{equation}\label{eq:secconv}
\sum_{i,j,k=1}^3 \ep_{ijk}\int_0^L  (\partial_iX^j \circ \Gamma_\eps)\partial_{ss}\Gamma_\eps^k\,ds
 \xrightarrow{\eps\to 0} \sum_{i,j,k=1}^3 \ep_{ijk}\int_0^L(\partial_iX^j \circ \Gamma_\eps)
\partial_{ss}\Gamma^k\,ds, 
\end{equation}
in the sense of distributions on $I.$ We will next show that in the sense of distributions on $I,$ 
\begin{equation}\label{eq:premconv}
\sum_{i,j,k=1}^3 \ep_{ijk}\int_0^L (\partial_iX^j \circ \Gamma_\eps) \left( \partial_t \Gamma_\eps \times
\partial_s \Gamma_\eps\right)^k\,ds \xrightarrow{\eps \to 0}  \sum_{i,j,k=1}^3 \ep_{ijk}\int_0^L(\partial_iX^j
\circ \Gamma_\eps)  \left(
\partial_t \Gamma \times
\partial_s \Gamma\right)^k\,ds ,
\end{equation}
after we first establish that the right hand side of \eqref{eq:premconv} is indeed a distribution on $I.$ 
Our argument is based on the fact that 
$$
\partial_s\Gamma\in L^\infty(I,H^{1/2}(T^1,\R^3)) \cap
L^\infty(I\times T^1) = L^\infty\big(I, H^{1/2}(T^1,\R^3) \cap L^\infty(T^1,\R^3)\big) =: L^\infty(I,Y),
$$
and on the algebra properties of this last space. Since $\partial_{ss}\Gamma \in
L^\infty(I,H^{-1/2}(T^1,\R^3)) \subseteq L^\infty(Y^*)$, we firstly infer that 
$$
\partial_t \Gamma = \partial_s\Gamma \times \partial_{ss}\Gamma \in L^\infty(I,Y^*),
$$ 
and then secondly that
$$
\partial_t \Gamma \times \partial_s \Gamma \in L^\infty(I,Y^*),
$$
the latter being therefore a well-defined distribution. 

Let now $m\in \mathcal{D}(I,\R)$ be arbitrary, and set
$$ 
\chi_\eps(t,s) := ( \partial_iX^j
\circ \Gamma_\eps)(t,s) m(t) \in \mathcal{C}(I,H^{3/2}(T^1,\R)).
$$
For $p,q \in \{1,2,3\},$ we write 
\begin{equation}\label{eq:vivakida0}
\partial_t \Gamma_\eps^p \partial_s \Gamma_\eps^q = \partial_t \Gamma^p \partial_s \Gamma^q
+ \partial_t(\Gamma_\eps^p -\Gamma^p)\partial_s \Gamma_\eps^q
 + \partial_t\Gamma^p \partial_s(\Gamma_\eps - \Gamma)^q.
\end{equation}
Denoting by $\langle\cdot,\cdot\rangle$ the $L^2(I,Y^*)-L^2(I,Y)$
duality, we have
\begin{equation}\label{eq:sautemouton}
\langle \partial_t(\Gamma_\eps^p -\Gamma^p ) \partial_s \Gamma_\eps^q , \chi_\eps\rangle 
= \langle\chi_\eps \partial_t(\Gamma_\eps^p -\Gamma^p ) , \partial_s \Gamma_\eps^q\rangle
= - \langle \rho_\eps * \big( \chi_\eps\partial_t(\Gamma_\eps^p -\Gamma^p )\big)  , \partial_s \Gamma^q\rangle.
\end{equation}
The family $\big(\rho_\eps * \big( \chi_\eps\partial_t(\Gamma_\eps^p -\Gamma^p )\big)\big)_{\eps}$ being bounded
 in the dual Banach space $L^2(I,Y^*),$ the Banach-Alaoglu theorem yields a subsequence that converges weakly-* 
to some limit in $L^2(I,Y^*).$ Since it also converges to $0$ in the sense of distributions on $I\times T^1$
(recall that $\chi_\eps$ is strongly convergent in $\mathcal{C}(I,H^{s}(T^1,\R))$ for any $s<3/2$), that
limit must be zero. Hence, $m\in \mathcal{D}(I,\R)$ being arbitrary, passing to the limit in the last term of
\eqref{eq:sautemouton} we deduce that
\begin{equation}\label{eq:vivakida1}
\int_0^L(\partial_iX^j \circ \Gamma_\eps) \partial_t(\Gamma_\eps^p -\Gamma^p)\partial_s \Gamma_\eps^q\, ds \to 0
\end{equation}
in the sense of distributions on $I.$ Similarly we write
\begin{equation*}\label{eq:sautemoutonbis}
\langle \partial_t\Gamma^p  \partial_s (\Gamma_\eps^q-\Gamma^q) , \chi_\eps\rangle = 
\langle \partial_t\Gamma^p  \chi_\eps, \partial_s (\Gamma_\eps^q-\Gamma^q)\rangle 
= -\langle ({\rm Id - \rho_\eps *})( \partial_t\Gamma^p  \chi_\eps), \partial_s\Gamma^q\rangle,
\end{equation*}
and we deduce as above that
\begin{equation}\label{eq:vivakida2}
\int_0^L(\partial_iX^j \circ \Gamma_\eps) \partial_t\Gamma^p \partial_s (\Gamma_\eps^q-\Gamma^q)\,ds \to 0
\end{equation}
in the sense of distributions on $I.$ Combining \eqref{eq:vivakida1} and \eqref{eq:vivakida2} for various $i$, $j$,
$p$ and $q$ in $\{1,2,3\}$ according to \eqref{eq:vivakida0}, we obtain \eqref{eq:premconv}. 
Combining \eqref{eq:premconv} and \eqref{eq:secconv} we are led to
$$
R_\eps \to \sum_{i,j,k=1}^3 \ep_{ijk}\int_0^L (\partial_iX^j \circ \Gamma)\left( \partial_t \Gamma \times
\partial_s \Gamma - \partial_{ss}\Gamma\right)^k\,ds = 0,
$$
and the proof is complete.
\end{proof}  
%%%%%%%%%%%%%%%%%%%%%%%%%%%%%%%%%%%%%%%%%%%%%%%%%%%%%%%%%%%%%%%%%%%%%%%%%%%%%%%%%%%%%%%%%%%%%%%
\section{The heart of the matter : A pointwise estimate}\label{sect:key}
%%%%%%%%%%%%%%%%%%%%%%%%%%%%%%%%%%%%%%%%%%%%%%%%%%%%%%%%%%%%%%%%%%%%%%%%%%%%%%%%%%%%%%%%%%%%%%%

In this section we present the key estimate that will allow us to control
the growth of the discrepancy between a smooth solution $\gamma$ and a less
smooth solution $\Gamma$  of the binormal curvature flow equation \eqref{eq:strongbfbis}.
 Nearly the same
estimate plays an important role in a somewhat different setting in \cite{JeSm2},
where we study a class of non-parametrized weak solutions of the binormal
curvature
flow. For the sake of completeness, and because it is reasonably short, we
present the proof both here and in \cite{JeSm2}, in full (and nearly identical)
detail.

\medskip

Let $(t_0,s_0)\in \R^2$ and $\gamma \,:\, \mathcal{O} \to \R^3$, where $\mathcal{O}$ is some open
neighborhood of $(t_0,s_0).$ Assume that $\partial_{ts}\gamma$ and $\partial_{sss}\gamma$ 
are continuous in $\mathcal{O},$ and that for any $(t,s) \in \mathcal{O},$
\begin{equation}\label{eq:paramarc}
|\partial_s\gamma(t,s)| = 1. 
\end{equation}
   
Let $x_0\in \R^3$ be such that
\begin{equation}\label{eq:proche1} 
\big| x_0 - \gamma(t_0,s_0) \big| \big|\partial_{ss}\gamma(t_0,s_0)\big| < 1/2 
\end{equation}
and
\begin{equation}\label{eq:ortho1}
\big( x_0 - \gamma(t_0,s_0)\big) \cdot \partial_s \gamma(t_0,s_0)=0.
\end{equation}

The mapping $\Psi : \R^3 \times \mathcal{O} \to \R$
$$
(x,t,s) \mapsto \big( x - \gamma(t,s)\big) \cdot \partial_s \gamma(t,s) 
$$
satisfies $\Psi(x_0,t_0,s_0)=0$ and
\begin{equation}\label{eq:ibm1}
\partial_s \Psi(x_0,t_0,s_0) = - |\partial_s\gamma(t_0,s_0)|^2 +\big( x_0 - \gamma(t_0,s_0)\big) \cdot
\partial_{ss}\gamma(t_0,s_0) < -1/2 , 
\end{equation}
where we have used \eqref{eq:paramarc} and \eqref{eq:proche1} for the last inequality. From the implicit function
theorem, we infer that there exist an open neighborhood $\mathcal{U}$ of $(x_0,t_0)$ in $\R^4,$ 
and a function $\zeta\,:\, \mathcal{U}\to \R$ such that
\begin{equation*}
\Psi(x,t,\zeta(x,t)) = 0 \qquad \forall (x,t) \in \mathcal{U}.
\end{equation*}   
Moreover $\partial_t\zeta$ and $\partial_{{x_i}{x_j}}\zeta$ exist and are continuous in $\mathcal{U}$, for
every $i,j\in \{1,2,3\}.$

\medskip

We fix $r>0$ and define the vector field $X\equiv X_{\gamma,r}$ on $\mathcal{U}$ by
\begin{equation}\label{eq:defX}
X(x,t) = f\Big(\big|x-\gamma(t,\zeta(x,t))\big|^2\Big) \partial_s\gamma(t,\zeta(x,t)),
\end{equation}
where
the function $f\equiv f_r\,:\, [0,+\infty) \to [0,+\infty)$ is given by 
$$
f(d^2) := \left\{ 
\begin{array}{ll}
\displaystyle 1- \big(\tfrac{d}{r}\big)^2, & \ \text{for } \,0\:\leq d^2 \leq r^2,\\
\displaystyle \ 0, & \ \text{for } d^2\geq r^2.
\end{array}\right.
$$

The following pointwise estimate is the core of our analysis. It allows to control
the terms which appear on the right-hand side of the identity in Proposition \ref{prop:weakp}.

\begin{prop}\label{prop:core}
Assume that
\begin{equation}\label{eq:labonne}
\big( \partial_t\gamma - \partial_s\gamma\times\partial_{ss}{\gamma}\big)(t_0,s_0) \ = \ \partial_s \big(
\partial_t\gamma - \partial_s\gamma\times\partial_{ss}{\gamma}\big)(t_0,s_0) = 0.
\end{equation}
Then for any $r> |x_0-\gamma(t_0,s_0)|$ and $V\in S^2\subset \R^3$ the estimate 
\begin{equation*}\label{eq:waou}
\bigl|\partial_t X \cdot V - D({\rm curl} X) : (V \otimes V)  \bigr|   \leq  K \left( 1 -
X \cdot V\right)
\end{equation*}
holds at the point $(x_0,t_0),$ where
$$
K \equiv K(\gamma,r) := \frac{8}{r^2} + \frac{32}{r_0^2} + \left(2+32\frac{r}{r_0}\right)\Sigma_0,
$$
%\left(48+24\tfrac{r_0^2}{r^2}\right)\big|\partial_{ss}\gamma(t_0,s_0)\big|^{2}  + \left(2+32
%\tfrac{r}{r_0}\right) \big|\partial_{sss}\gamma(t_0,s_0)\big|, 
$$
r_0 := 1/|\partial_{ss}\gamma(t_0,s_0)|\in (0,+\infty],\qquad \text{ and }\qquad \Sigma_{0} =
|\partial_{sss}\gamma(t_0,s_0)|\in [0,+\infty). 
$$
\end{prop}
\noindent(In case $r_0=+\infty$, one can take $r=+\infty$ and then $K$ should be understood as $K =  2\Sigma_0$)

\begin{proof}
We start by fixing some notations that will help keeping expressions of reasonable size in the sequel. 
 For a function $Y$ with values in $\R^3,$ and $i\in\{1,2,3\},$ we write $Y^i$ to denote the i-th
component of $Y.$ We write  $d^2$ to denote the function 
$(x,t)\mapsto |x-\gamma(t,\zeta(x,t))|^2$, $\gamma(\zeta)$ to denote the function $(x,t)\mapsto \gamma(t,\zeta(x,t))$, and similarly for
$\partial_t\gamma(\zeta),$  $\partial_{ts}\gamma(\zeta),$ $\partial_s\gamma(\zeta),$ $\partial_{ss}\gamma(\zeta)$ and
$\partial_{sss}\gamma(\zeta).$ When this does not lead to a possible confusion, we also denote by $x$ the function
$(x,t)\mapsto x.$  Each of these functions are being defined on $\mathcal{U}.$

\medskip

In view of \eqref{eq:ibm1}, shrinking $\mathcal{U}$ if necessary,
we may assume that\footnote{In particular we won't need to deal with the lack of regularity of $f$ at $d=r$.}  
\begin{equation}\label{eq:tp2}
d^2(x,t) < r,\qquad \forall (x,t) \in \mathcal{U},
\end{equation}
and 
\begin{equation}\label{eq:ibm2}
\rho(x,t):= 1 - \big( x - \gamma(t,\zeta(x,t))\big) \cdot
\partial_{ss}\gamma(t,\zeta(x,t)) > 1/2,\qquad \forall (x,t) \in \mathcal{U}. 
\end{equation} 

{\bf Step 1: First computation of \mb$D({\rm curl} X) : (V \otimes V).$\mn}    
Differentiating \eqref{eq:defX} we obtain, pointwise on $\mathcal{U}$ and for $i,j \in\{1,2,3\},$  
\begin{equation*}\label{eq:djXgamma}
\partial_{j}X^i = \partial_{j}(f(d^2) )\partial_s \gamma(\zeta)^i
+ f(d^2) \, \partial_{ss}\gamma(\zeta)^i \, \partial_{j}\zeta
\end{equation*}
for the space derivatives, and 
\begin{equation}
\label{eq:dtXgamma}
\partial_{t}X^i = \partial_{t}(f(d^2) )\, \partial_s \gamma(\zeta)^i
+ f(d^2) \, \big[ \partial_{ss}\gamma(\zeta)^i \, \partial_{t}\zeta + 
\partial_{ts}\gamma(\zeta)\big] 
\end{equation}
for the time derivative. Also, for $i,j,\ell \in\{1,2,3\},$ 
\begin{equation}\begin{split}
\partial_{\ell j} X^i =  & 
\partial_{\ell j} (f(d^2 ))\, \partial_s \gamma(\zeta)^i
+\partial_{ss}\gamma(\zeta)^i \big[ \partial_{\ell} (f(d^2)) \, \partial_{j}\zeta + \partial_{j} (f(d^2)) \,
\partial_{\ell}\zeta\big]\nonumber
\\
&
+ f(d^2) \, \partial_{sss}\gamma(\zeta)^i \,\partial_{\ell}\zeta \,\partial_{j}\zeta
+ f(d^2) \, \partial_{ss}\gamma(\zeta)^i \,\partial_{\ell j} \zeta.
\label{eq:dldjXgamma}
\end{split}\end{equation}
%%%%%%%%%%%%%%%
%%%%%%%%%%%%%%%
In particular, we may write
\begin{equation*}\label{eq:Dcurl1}
D({\rm curl} X) : (V \otimes V) 
\ =: \  A 
 \ = \ A_1 + A_2 + A_3 + A_4, 
\end{equation*}
where
\begin{equation*}\label{eq:ugly1}
\begin{aligned}
A_1& :=  \ep_{ijk} \partial_{\ell i} (f(d^2 ))\, \partial_s \gamma(\zeta)^j \,V^k V^\ell \,  ,\\
A_2&:= 
\ep_{ijk} 
\partial_{ss}\gamma(\zeta)^j \big[ \partial_{\ell} (f(d^2)) \, \partial_{i}\zeta  + \partial_{i} (f(d^2)) \
\partial_{\ell}\zeta  \big]V^k V^\ell \, ,
\\
A_3&
:= \ep_{ijk} \, f(d^2)\, \partial_{sss}\gamma(\zeta)^j \,\partial_{\ell} \zeta  \, \partial_{i}\zeta V^k V^\ell \, , \\
A_4&
:=  \ep_{ijk} f(d^2) \ \partial_{ss}\gamma(\zeta)^j \,\partial_{\ell i}\zeta
V^kV^\ell,
\end{aligned}
\end{equation*}
in which $\ep_{ijk}$ is the Levi-Civita symbol and we sum over repeated indices. 

\medskip

{\bf \mb Step 2: Expressing derivatives of $\zeta$ in terms of $\gamma$.\mn} 
Recall that by definition of $\zeta$, we have
\begin{equation}\label{eq:sigma1}
(x - \gamma(t,\zeta(t,x)) )\cdot \partial_s \gamma(t,\zeta(t,x)) = 0
\end{equation}
for every $(x,t)\in \mathcal{U}.$ For $j\in \{1,2,3\},$ differentiating \eqref{eq:sigma1} 
with respect to $x_j$  and using \eqref{eq:paramarc} we find 
\begin{equation}\label{eq:sigma1.5}
\partial_s \gamma(\zeta)^j- \partial_{j}\zeta + (x-\gamma(\zeta))\cdot
\partial_{ss}\gamma(\zeta) \ \partial_{j}\zeta = 0.
\end{equation}
In view of \eqref{eq:ibm2}, we may rewrite \eqref{eq:sigma1.5} as
\begin{equation}\label{eq:sigma2}
\partial_{j}\zeta  = \frac 1 {\rho} \partial_s \gamma(\zeta)^j.
\end{equation}
For $\ell\in \{1,2,3\},$ differentiating \eqref{eq:sigma1.5} with respect to $x_\ell$ and using 
\eqref{eq:paramarc}, we obtain
\begin{equation}\label{eq:sigma3}
\partial_{\ell j}\zeta
= \frac 1 \rho \Big(
\partial_{ss}\gamma(\zeta)^j \frac {\partial_s \gamma(\zeta)^\ell}\rho 
+ 
\partial_{ss} \gamma(\zeta)^\ell \frac {\partial_s \gamma(\zeta)^j}\rho 
+
(x-\gamma(\zeta))\cdot \partial_{sss}\gamma(\zeta) \frac
{\partial_s\gamma(\zeta)^j \partial_s\gamma(\zeta)^\ell}{\rho^2}
\Big).
\end{equation}
Finally, differentiating \eqref{eq:sigma1} with respect to $t$ we obtain
\begin{equation}
\partial_t\zeta \ = \ 
\frac 1 \rho( - \partial_t\gamma(\zeta) \cdot \partial_s \gamma(\zeta) + (x -
\gamma(\zeta))\cdot\partial_{ts} \gamma(\zeta)).
\label{eq:sigma4}\end{equation}
In particular, taking into account \eqref{eq:labonne} it follows from \eqref{eq:sigma4} that, {\it at the point} $(x_0,t_0)$,
\begin{equation*}\label{eq:sigma4point}
\partial_t\zeta \ = \ 
\frac {1}{\rho} (x - \gamma(\zeta))\cdot\ \big(\partial_s \gamma(\zeta)\times
\partial_{sss} \gamma(\zeta)\big).
\end{equation*}

\medskip

{\bf \mb Step 3: Expressing derivatives of $d^2$ in terms of $\gamma.$\mn} 
In view of the definition of $d^2$, we have for $j\in \{1,2,3\},$  
\begin{equation}\label{eq:Ddsquared}
\partial_{j} d^2  = 2(x - \gamma(\zeta))^j - 2(x - \gamma(\zeta))\cdot
\partial_s\gamma(\zeta)\ 
\partial_{j}\zeta \  = \  2(x - \gamma(\zeta))^j, 
\end{equation}
where the last equality follows from \eqref{eq:sigma1}. 
For $\ell\in \{1,2,3\},$ differentiating \eqref{eq:Ddsquared} with respect to $x_\ell$ and using \eqref{eq:sigma1.5}, 
we obtain  
\begin{equation}\label{eq:D2dsquared}
\partial_{\ell j} d^2 = -2\big( \delta_{j\ell } - \partial_s \gamma(\zeta)^j\partial_{\ell}\zeta\big) =
-2\Big(\delta_{j\ell} - \frac {\partial_s \gamma(\zeta)^j\, \partial_s \gamma(\zeta)^\ell}\rho\Big),
\end{equation}
where $\delta_{j\ell}$ is the Kronecker symbol. Also from the definition of $d^2$, we have 
\begin{equation}\label{eq:dtdsquared}
\partial_{t} d^2 =  - 2(x - \gamma(\zeta))\cdot (\partial_t\gamma(\zeta) + \partial_s \gamma(\zeta) \,  \partial_t\zeta).
\end{equation}
In particular, taking into account \eqref{eq:sigma1} and \eqref{eq:labonne} it follows from \eqref{eq:dtdsquared} that, {\it at the point} $(x_0,t_0)$,
\begin{equation}\label{eq:dtdsquaredpoint}
 \partial_{t} d^2 =  - 2(x - \gamma(\zeta))\cdot (\partial_s\gamma(\zeta) \times \partial_{ss}
   \gamma(\zeta)).
\end{equation}

\medskip

{\bf \mb Step 4: A reduced expression for \mb$D({\rm curl} X) : (V \otimes V).$\mn} We substitute,
in the terms $A_1, A_2, A_3$ and $A_4$ defined in Step 1, the expressions for the derivatives of 
$d^2$ and $\zeta$ which we obtained in Step 2 and Step 3. Some cancellations occur.

\smallskip

Examining $A_1$, we first expand:
\begin{equation*}\begin{split}
\partial_{\ell i} (f(d^2 )) &=   f''(d^2) \partial_{\ell}d^2 \,
\partial_{i}d^2 + f'(d^2) \partial_{\ell i} d^2\\
&=  \tfrac{2}{\rho}f'(d^2)\partial_s \gamma(\zeta)^\ell\,
\partial_s \gamma(\zeta)^i - 2f'(d^2)\delta_{\ell i},
\end{split}
\end{equation*}
where we have used $f''(d^2)\equiv 0$ by \eqref{eq:tp2} as well as \eqref{eq:D2dsquared} for the second equality. 
Next, 
\begin{align*}
%\ep_{ijk} 
%(x-\gamma(\zeta))^\ell \, (x-\gamma(\zeta))^i
%\partial_s \gamma(\zeta)^j\, V^k V^\ell &= \big(\ep_{ijk}(x-\gamma(\zeta))^i\partial_s
%\gamma(\zeta)^jV^k\big)\big((x-\gamma(\zeta))^\ell V^\ell\big)\\ 
%&= \big((x - \gamma(\zeta))\cdot (\partial_s\gamma(\zeta)\times V) \big)\big( (x - \gamma(\zeta)) \cdot
%V\big).\\
%\shortintertext{Similarly,} 
\ep_{ijk} \partial_s \gamma(\zeta)^\ell\,
\partial_s \gamma(\zeta)^i\, \partial_s \gamma(\zeta)^j\, V^k V^\ell &= \big(\partial_s \gamma(\zeta)\cdot
(\partial_s\gamma(\zeta)\times V) \big)\big( \partial_s\gamma(\zeta) \cdot V\big) = 0,\\
\shortintertext{and}
\ep_{ijk} \delta_{\ell i} \, \partial_s \gamma(\zeta)^j\, V^k V^\ell &=  \ep_{ijk} V^i \, \partial_s
\gamma(\zeta)^j\, V^k = V\cdot \big( \partial_s\gamma(\zeta)\times V\big) = 0.
\end{align*}
Hence,
\begin{equation*}\label{eq:A1}
A_1 =  0.
%4f''(d^2)\big((x - \gamma(\zeta))\cdot (\partial_s\gamma(\zeta)\times V) \big)\big( (x - \gamma(\zeta)) \cdot
%V\big). 
\end{equation*}
Concerning $A_2$, \eqref{eq:sigma2} and \eqref{eq:Ddsquared} yield 
\begin{equation*}\label{eq:A2}\begin{split}
A_2 = &\ \tfrac 2 \rho f'(d^2) \ep_{ijk} V^k V^\ell 
\partial_{ss}\gamma(\zeta)^j  \big( (x-\gamma(\zeta))^\ell \partial_s
\gamma(\zeta)^i  +(x-\gamma(\zeta))^i   \partial_s \gamma(\zeta)^\ell  \big)\\
= & \   
\tfrac 2 \rho f'(d^2)  \partial_{ss}\gamma(\zeta)\cdot  (\partial_s
\gamma(\zeta) \times V)  (V \cdot (x-\gamma(\zeta))) \ +
 \\
& \ \tfrac 2 \rho f'(d^2) (x-\gamma(\zeta))\cdot(\partial_{ss}\gamma(\zeta) \times V) (V \cdot
\partial_s\gamma(\zeta))
\\
 =&\!: \ A_{2,1} + A_{2,2}.
\end{split}\end{equation*}
For $A_3,$ we invoke \eqref{eq:sigma2} to substitute $\partial_\ell \zeta$ and $\partial_i\zeta$ and obtain 
\begin{equation*}\label{eq:A3}
A_3   =   \frac 1{\rho^2}f(d^2) (\partial_s\gamma(\zeta)\cdot V) \
\partial_s\gamma(\zeta)
\cdot(\partial_{sss}\gamma(\zeta)\times V).
\end{equation*}
For $A_4$ finally, we invoke \eqref{eq:sigma3} to substitute $\partial_{\ell i}\zeta$ and obtain
\begin{equation*}\begin{split}
A_4 =&\ \tfrac 1{\rho^2} f(d^2)\partial_{ss}\gamma(\zeta)\cdot\big(\partial_{ss}\gamma(\zeta)\times V\big)
(\partial_{s}\gamma(\zeta)\cdot V)
\ +\\
&\ \tfrac 1{\rho^2} f(d^2)  \partial_s \gamma(\zeta)\cdot
(\partial_{ss}\gamma(\zeta)\times V )
(\partial_{ss}\gamma(\zeta)\cdot V) 
\ + \\
&\ \tfrac 1{\rho^3} f(d^2) 
((x-\gamma(\zeta))\cdot \partial_{sss}\gamma(\zeta))\, \partial_s
\gamma(\zeta)\cdot( \partial_{ss}\gamma(\zeta)\times V)(\partial_s \gamma(\zeta)\cdot V)
\\
=&\!:\ 0 + A_{4,1} + A_{4,2}.
\end{split}
\end{equation*}

{\bf \mb Step 5: Computation of $\partial_t X \cdot V$.\mn}
We expand \eqref{eq:dtXgamma} as
\begin{equation*}
\label{eq:dtXgammabis}
\partial_{t}X^i = f'(d^2)\partial_{t}d^2 \, \partial_s \gamma(\zeta)^i
+ f(d^2) \, \big[ \partial_{ss}\gamma(\zeta)^i \, \partial_{t}\zeta + 
\partial_{ts}\gamma(\zeta)\big]. 
\end{equation*}
Therefore, {\it at the point} $(x_0,t_0),$ we obtain from \eqref{eq:labonne}, \eqref{eq:sigma4} and
\eqref{eq:dtdsquaredpoint}
\begin{equation*}\label{eq:dtXdotV}
\partial_t X \cdot V =: B = B_1 + B_2 + B_3, 
\end{equation*} 
where
\begin{equation*}\label{eq:notsobad}
\begin{aligned}
B_1 
&:= -2  f'(d^2) ((x - \gamma(\zeta))\cdot (\partial_s \gamma(\zeta) \times
\partial_{ss}\gamma(\zeta))) \, ( \partial_s\gamma(\zeta) \cdot V)
\, , 
\\
B_2
&:= \tfrac 1 {\rho} f(d^2)\,  ((x-\gamma(\zeta))\cdot (\partial_s\gamma(\zeta)\times
\partial_{sss}\gamma(\zeta))) \,( \partial_{ss}\gamma(\zeta) \cdot V)\, ,
\\
B_3 
&:= f(d^2)  \, (\partial_s\gamma(\zeta) \times \partial_{sss}\gamma(\zeta))\cdot V\, . 
\end{aligned}
\end{equation*}

{\bf \mb Step 6: Proof of Proposition \ref{prop:core} completed.\mn} 
We write, {\it at the point} $(x_0,t_0),$ 
\begin{equation}\label{eq:rassemble}\begin{split}
\bigl| B - A \bigr| &= \bigl|\partial_t X \cdot V - D({\rm curl} X) : (V \otimes V)  \bigr| \\
& \leq |A_1| + |A_{2,1}| + |A_{2,2} - B_1| +  |A_3 - B_3| + |A_{4,1}| + |A_{4,2}| + |B_2|,
\end{split}\end{equation}
and we will estimate each of the terms in the last line separately. 
We first observe the following elementary facts that hold at the point $(x_0,t_0)$ (when they involve functions):
$$
\begin{array}{lll}
a) &\ds |V^\perp| \, , \, |V|\, ,\, |\partial_s\gamma(\zeta)| \leq 1,  &(\text{indeed } V\in S^2 \text{ and
}\eqref{eq:paramarc} \text{ holds}), \\
b) &\ds |f(d^2)|\leq 1,\: |f'(d^2)| = 1/r^2,  & (\text{follows from the definition of }f),\\
c) &\ds \rho\geq 1/2, |1-1/\rho|\leq 2d/r_0, |1 - 1/\rho^2|\leq 6d/r_0, &(\text{from }\eqref{eq:proche1} \text{ and
the definition of }\rho). 
\end{array}
$$
Also recall that by definition
$$
r_0 = |\partial_{ss}\gamma(t_0,s_0)|^{-1},\qquad\text{and}\qquad \Sigma_0 = |\partial_{sss}\gamma(t_0,s_0)|.
$$
Taking into account \eqref{eq:ortho1}, direct inspection yields  
\begin{equation}\label{eq:premieres}
|A_1| =  0, \qquad |A_{2,1}| \le 4 d (r^{2}r_0)^{-1} |V^\perp|, \qquad 
|A_{4,1}| \le 4r_0^{-2}|V^\perp|^2,
\end{equation}
as well as
\begin{equation*}\label{eq:secondes}
|A_{4,2}| \le  8 dr_0^{-1} \Sigma_0|V^\perp|,
\quad\quad 
|B_2| \le 2 dr_0^{-1}\Sigma_0 |V^\perp|.
\end{equation*}
Next, we write  
\begin{equation*}\label{eq:tierces}\begin{split}
| B_1 - A_{2,2}|  &= \big| \tfrac 2 \rho f'(d^2) (\partial_s\gamma(\zeta)\cdot V)
((x-\gamma(\zeta))\times \partial_{ss}\gamma(\zeta) )\cdot
 ( \partial_s\gamma(\zeta) - \tfrac V\rho )\big|\\
&\le 4 d(r^2r_0)^{-1} \big(|\partial_s\gamma(\zeta) - V| + 2 dr_0^{-1}\big),
\end{split}\end{equation*}
and
\begin{equation}\label{eq:quartes}\begin{split}
|B_3 - A_3| &=  \big| f(d^2)[(\partial_s\gamma(\zeta)\times \partial_{sss}\gamma(\zeta))\cdot V]( 1
 - \tfrac 1 {\rho^2}V\cdot \partial_s\gamma(\zeta))
\big|\\
&\le \Sigma_0 |V^\perp| \ \big( (1- \partial_s\gamma(\zeta)\cdot V) +
6dr_0^{-1}\big)\\
&\le \Sigma_0  \big( (1- \partial_s\gamma(\zeta)\cdot V) + 6dr_0^{-1}|V^\perp|\big).
\end{split}\end{equation}
It remains to bound $d$, $|V^\perp|$, $|\partial_s\gamma(\zeta) - V|$ and $|1- \partial_s\gamma(\zeta)\cdot V|$
in terms of $1-X\cdot V.$ To that purpose, first recall from the definition of $f$, from the fact that
$|V|=|\partial_s\gamma(\zeta)| = 1$, and from the assumption $d\leq r$, that 
\begin{equation}\label{eq:dominun}
1 - X \cdot V \geq  1- f(d^2) \geq d^2r^{-2}.
\end{equation}
Also, if $X\cdot V\geq 0$ then $1-X\cdot V \geq 1-\partial_s\gamma(\zeta)\cdot V,$ and if $X\cdot V <0$ then
$1-X\cdot V \geq 1 \geq (1-\partial_s\gamma(\zeta)\cdot V)/2.$ In any case, we have
\begin{equation}\label{eq:domindeux}
1 - X \cdot V \geq  \frac{1}{2} \left( 1 - \partial_s\gamma(\zeta)\cdot V\right).
\end{equation}
Finally, by Hilbert's projection theorem 
\begin{equation}\label{eq:domintrois}
|V^\perp|^2 \le |V - \partial_s\gamma(\zeta)|^2 = 2\big(1 - \partial_s\gamma(\zeta)\cdot
V\big) \leq 4 \big(1-X\cdot V\big).
\end{equation}
Inserting \eqref{eq:dominun}, \eqref{eq:domindeux}, or \eqref{eq:domintrois} in
\eqref{eq:premieres}-\eqref{eq:quartes}, and writing $x\preceq y$ for $x\leq y(1-X\cdot V)$, we obtain
$$
|A_1|\preceq 0,\quad |A_{2,1}|\preceq \frac{8}{r_0r},\quad |A_{4,1}|\preceq \frac{16}{r_0^2},
\quad |A_{4,2}|\preceq 16\frac{r}{r_0}\Sigma_0,
$$
$$|B_2|\preceq 4\frac{r}{r_0}\Sigma_0,\quad |B_1-A_{2,2}|\preceq
(1+\frac{r}{r_0})\frac{8}{r_0r},\quad |B_3-A_3|\preceq (2+12\frac{r}{r_0})\Sigma_0,  
$$
and summation according to \eqref{eq:rassemble} yields the claim.
\end{proof}

%%%%%%%%%%%%%%%%%%%%%%%%%%%%%%%%%%%%%%%%%%%%%%%%%%%%%%%%%%%%%%%%%%%%%%%%%%%%%%%%%%%%%%%%%%%%%%%%%%%%%%%%%%%%%%%%
\section{A Gronwall estimate for \mb$F_{\gamma,\sigma,r}(\Gamma)$\mn }\label{sect:gron}
%%%%%%%%%%%%%%%%%%%%%%%%%%%%%%%%%%%%%%%%%%%%%%%%%%%%%%%%%%%%%%%%%%%%%%%%%%%%%%%%%%%%%%%%%%%%%%%%%%%%%%%%%%%%%%%%

Let $\gamma \in L^\infty(I, H^4_{\rm loc}(\R,\R^3))$ 
and $\Gamma \in L^\infty(I, H^\frac32_{\rm loc}(\R,\R^3))$ be two solutions
of the binormal curvature flow equation \eqref{eq:strongbfbis} on $I\times \R,$ where $I=(-T,T)$ for some $T>0.$  
Assume moreover that $\gamma$ and 
$\Gamma$ are both quasiperiodic and that both have the same pitch, that is  
\begin{equation}\label{eq:perio0bis}
\begin{split}
&\gamma(t,s+\ell)\ = \,\gamma(t,s) + a,\qquad\forall t\in J,\, \forall\, s\in \R,\\
&\Gamma(t,s+L   ) = \,\Gamma(t,s) + a,\qquad\forall t\in J,\, \forall\, s\in \R,
\end{split} 
\end{equation}
for some $\ell>0,$ $L >0,$ and $a\in \R^3.$ It follows from \eqref{eq:strongbfbis} and our assumptions
on $\gamma$ and $\Gamma$ that\footnote{Actually $\gamma$ could be globally extended from $I$ to $\R$ as a solution, see
Subsection \ref{sub:Cauchy}, but the existence of $\Gamma$ has to be assumed a priori.}    $\gamma \in \mathcal{C}(\bar I,\mathcal{C}^3(\R,\R^3))$ and $\Gamma \in 
\mathcal{C}(I,\mathcal{C}(\R,\R^3)).$ 
For $t\in I,$ we set
$$
\gamma_t:=\gamma(t,\cdot),\qquad \Gamma_t:=\Gamma(t,\cdot),
$$
and define\footnote{We have already defined a quantity called $r_\gamma$ for a curve with no time dependence. One is the
natural extension of the other and there shouldn't be any possible confusion.}  
$$
r_\gamma := \max_{t\in \bar I} \left(\max_{s\in \R}|\partial_{ss}\gamma(t,s)|\right)^{-1}\:  \in (0,+\infty].
$$
Assume that at time zero we have 
\begin{equation*}\label{eq:toutpres0}
d_\mathcal{P}(\Gamma_0,\gamma_0)< r_\gamma/8,
\end{equation*}
and let therefore $\sigma_0$ be a reparametrization of $\gamma_0$ for $\Gamma_0$ (the existence of which being
ensured by Corollary \ref{cor:repar}).   

\begin{thm}\label{thm:binormalestim}
Fix $0<r\leq r_\gamma/8$ and assume that 
\begin{equation}\label{eq:layton0}
F(0):= F_{\gamma_0,\sigma_0,r}(\Gamma_0) < F_r := r \left(\sqrt{2} + \frac{r}{L}\right)^{-1}.
\end{equation}
Define then
$$
T_r := \tfrac{1}{K} \log\left(\tfrac{F_r}{F(0)}\right),
$$
where
$$
K :=\frac{8}{r^2} + \frac{32}{r_\gamma^2} + \left(2+32\frac{r}{r_\gamma}\right)
\|\partial_{sss}\gamma\|_{L^\infty(I \times \R)},
$$
and set $J_r=[-T_r,T_r]\cap (-T,T).$
There exists a unique $\sigma\in \mathcal{C}(J_r\times\R,\R)$ such that
$\sigma(0,\cdot)=\sigma_0$ and $\sigma_t:=\sigma(t,\cdot)$ is a reparametrization 
of $\gamma_t$ for $\Gamma_t,$ for every $t\in J_r.$ The function $F$ defined
 on $J_r$ by $F(t) := F_{\gamma_t,\sigma_t,r}(\Gamma_t)$  is
Lipschitz continuous on $J_r$ and satisfies for almost every $t\in J_r$
the inequality
$$
\big| F'(t) \big| \leq K F(t).
$$
Therefore
$$
F(t) \leq \exp(K|t|)F(0),\qquad\forall \: t\in J_r, 
$$
and in particular
$$
\dP(\Gamma_t,\gamma_t) < r, \qquad\forall \: t\in J_r. 
$$
\end{thm}
\begin{proof}
It follows from \eqref{eq:layton0} and Lemma \ref{lem:estimhausbis} that $\|\Gamma_0-\gamma_0\circ\sigma_0\|_\infty
<r.$ Since $\Gamma$ and $\gamma$ depend continuously on time, we have $\|\Gamma_t-\gamma_t\circ\sigma_0\|_\infty <
r$ for any $t$ in a sufficiently small neighborhood $J$ of $0$ in $I.$ For $t\in J,$ applying Corollary \ref{cor:repar} 
to the curves $\gamma_t$ and $\Gamma_t$ with the choice $p=\sigma_0$ yields a reparametrization $\sigma_t$ of $\gamma_t$
 for $\Gamma_t$ that satisfies $\|\Gamma_t-\gamma_t\circ\sigma_t\|_\infty \leq
\|\Gamma_t-\gamma_t\circ\sigma_0\|_\infty<r$ and which is the only reparametrization of
$\gamma_t$ for $\Gamma_t$ that satisfy $\|\sigma_t-\sigma_0\|_\infty < r_\gamma/8.$  
Applying the implicit function theorem (exactly as we did in to the proof of Lemma \ref{lem:uniqprojloc} but with an additional
time dependence) to the function
$$
(t,s,x,z) \mapsto \big(x - \gamma(t,z)\big)\cdot \gamma'(t,z),
$$  
we infer that $\sigma_t=\sigma(t,\cdot)$ for some continuous function $\sigma$ on $J\times\R$ which is also 
uniquely determined locally and hence globally on $J\times \R.$ In particular, $T_c>0,$ where
\begin{equation*}\begin{split}
T_c :=\sup \big\{ & 0<\tau< T \text{ s.t. } \big[ \exists \, \sigma\in \mathcal{C}((-\tau,\tau)\times\R,\R) \text{ s.t. }
 \sigma(0,\cdot)=\sigma_0 \text{ and  }
\forall t\in (-\tau,\tau)\\ 
&\: \sigma_t:=\sigma(t,\cdot)
 \text{ is a reparametrization of } \gamma_t \text{ for } \Gamma_t \text{ satisfying }
\|\Gamma_t-\gamma_t\circ\sigma_t\|_\infty <r \big]\big\}.
\end{split}\end{equation*}
We claim that either $T_c>T_r$ or $T_c=T.$ We argue by contradiction and assume that it is not the case.  
Fix an arbitrary $0<T_1<T_c, $ set 
$R:= \max_{t\in [-T_1,T_1]}\|\Gamma_t-\gamma_t\circ \sigma_t\|_\infty < r,$ and fix $\chi_R \in
\mathcal{C}^\infty(\R^+,[0,1])$ such that $\chi_R\equiv 1$ on $[0,R]$ and $\chi_R\equiv 0$ outside
$[r,+\infty).$ Define the parametrized vector field $X\in \mathcal{C}((-T_1,T_1)\times
\R,\mathcal{C}^2(\R^3,\R^3))$ by
\begin{equation}\label{eq:defXbis}
X(t,s,x) = \chi_R\big(x-\gamma(\sigma(t,s))\big) f_{r_\gamma}\big(|x-\gamma(\sigma(t,s))|^2\big)
\partial_s\gamma(t,\sigma(t,s)).
\end{equation}
Note that $X$ satisfies \eqref{eq:ced1} and \eqref{eq:ced2}, by Corollary \ref{cor:repar} conclusion $\mathit 1$ and
Lemma \ref{lem:uniqprojloc} conclusion $\mathit 4.$ 
Applying Proposition \ref{prop:weakp} to $\Gamma$ and $X$ on $(-T_1,T_1)$ and using Proposition \ref{prop:core}
to estimate the right-hand side of \eqref{eq:weakbfp}\footnote{Indeed, the $X$ defined in \eqref{eq:defXbis} 
is locally independent of $s$ (this is again Lemma \ref{lem:uniqprojloc} conclusion $\mathit 4$), and, 
removing locally the variable $s$, coincides there with the vector field $X$ 
defined in \eqref{eq:defX}.}, we obtain that $F$ is Lipschitzian on $[-T_1,T_1]$ and   
\begin{equation}\label{eq:gronici}
\left|\frac{d}{dt} F(t) \right| = \left| \frac{d}{dt} \int_0^L   \left(X\circ \Gamma\right) \cdot \partial_s\Gamma
\, ds \right | \leq K F(t),  
\end{equation}
where $K$ is given in the statement and where we have also used the fact that by definition of $R,$ 
 $\chi_R(\Gamma(t,s)-\gamma(t,\sigma(t,s))) \equiv 1.$ %Since $T_1<T_c$ was arbitrary, the same holds on $(-T_c,T_c).$  
By Gronwall inequality and 
\eqref{eq:gronici}, we obtain $F(T_1)\leq \exp(KT_1)F(0)\leq\exp(KT_c)F(0)< F_c,$ independently of $T_1.$ Lemma
\ref{lem:estimhausbis} applied to $\sigma_{T_1}$, $\gamma_{T_1}$ and $\Gamma_{T_1}$ with the $r$ of the 
the statement  then implies that
\begin{equation*}\begin{split}
\big\|\Gamma_{T_1}-\gamma_{T_1}\circ\sigma_{T_1}\big\|_\infty^2 &\leq \big(\sqrt{2}r +
\tfrac{r^2}{L}\big)F(T_1)
\\
&\leq \big(\sqrt{2}r +
\tfrac{r^2}{L}\big)\exp(KT_c)F(0)
\\
&< \big(\sqrt{2}r +
\tfrac{r^2}{L}\big)\exp(KT_r)F(0) \leq r^2,
\end{split}
\end{equation*}
independently of $T_1.$ Applying Corollary \ref{cor:repar} as above for $T_1$ sufficiently close to $T_c$ and 
with the choice $p_0:=\sigma_{T_1}$, and taking into account the fact that $\Gamma$ and $\gamma$ are uniformly
continuous on $[-T_c,T_c]\times \R$ we would then be able to extend $\sigma$ past
 $T_c$ with the required properties, which is contradictory to the definition of $T_c.$     
\end{proof}

%%%%%%%%%%%%%%%%%%%%%%%%%%%%%%%%%%%%%%%%%%%%%%%%%%%%%%%%%%%%%%%%%%%%%%%%%%%%%%%%%%%%%%%%%%%%%%%%%%%%%%%%%%%%%%%%
\section{Proof of Theorem \ref{thm:main}}\label{sect:th1}
%%%%%%%%%%%%%%%%%%%%%%%%%%%%%%%%%%%%%%%%%%%%%%%%%%%%%%%%%%%%%%%%%%%%%%%%%%%%%%%%%%%%%%%%%%%%%%%%%%%%%%%%%%%%%%%%

We distinguish three cases. Case 1 is the base scenario where one can directly rely on the binormal flow estimate
of Section \ref{sect:gron}.  In Case 2 and Case 3, we will first modify the solution $u$ into solutions $\tilde u$ or
$\check u$ for which Case 1 may be applied, and then estimate the difference between $u$ and $\tilde u$ or $\check u$
separately.  

\medskip 

\noindent
{\bf Case 1: Assume that \mb$\int_{T^1} u = \int_{T^1} v.$\mn} We first apply Lemma \ref{lem:basepoint} to $u$ and $v$
successively. This yields solutions $\gamma_u\in \mathcal{C}(I,H^4_{\rm loc}(\R,\R))$ and $\gamma_v\in
\mathcal{C}(I,H^\frac32_{\rm loc}(\R,\R))$ of the binormal curvature flow equation, which for convenience we
will call $\gamma$ and $\Gamma$ respectively. Since $u$ and $v$ are $2\pi$-periodic with respect to the space variable, 
and since $\int_{T^1} u = \int_{T^1} v,$ it follows that $\gamma$ and $\Gamma$ satisfy conditions
\eqref{eq:perio0bis} with $\ell=L:=2\pi$ and $a:=\int_{T^1}u=\int_{T^1}v.$ As in Section \ref{sect:gron}, for $t\in I$ 
we set
$$
\gamma_t:=\gamma(t,\cdot),\qquad \Gamma_t:=\Gamma(t,\cdot),
$$
and define  
$$
r_u := \max_{t\in \bar I} \left(\sup_{s\in \R} |\partial_{ss}\gamma(t,s)|\right)^{-1} = \max_{t\in \bar I} \left(\sup_{s\in
T^1} |\partial_s u(t,s)|\right)^{-1}\in (0,+\infty].
$$
By construction, $\gamma_0(0)=\Gamma_0(0)=0.$ Applying Lemma \ref{lem:retour} with $\Gamma:=\Gamma_0$, $\gamma:=\gamma_0$
 and $L=\ell:=2\pi$, we infer that if
\begin{equation}\label{eq:acote0}
\int_{T^1}|v(0,s)-u(0,s)|^2\, ds < \frac{r_u^2}{128\pi}
\end{equation}
then $\dP(\gamma_0,\Gamma_0)<r_u/8,$ and there exists a reparametrization $\sigma_0$ of $\gamma_0$ for
$\Gamma_0,$ such that, setting $r:=r_u/8,$ we have
$$
F(0):= F_{\gamma_0,\sigma_0,r}(\Gamma_0) \leq \big( 1 + \tfrac{8^3\pi^2}{r_u^2} + \tfrac{2^8\pi^4}{r_u^4}\big)
\int_{T^1}|v(0,s)-u(0,s)|^2\, ds.
$$ 

Assume first that  
\begin{equation}\label{eq:acote1}\begin{split}
\int_{T^1}|v(0,s)-u(0,s)|^2\, ds &< \frac{r_u}{8}\big( 1 +
\tfrac{8^3\pi^2}{r_u^2}+ \tfrac{2^8\pi^4}{r_u^4}\big)^{-1}\big(\sqrt{2}+\tfrac{r_u}{16\pi}\big)^{-1}\\
&= \big( \tfrac{1}{2\pi} +
\tfrac{8\sqrt{2}}{r_u} + \tfrac{256\pi}{r_u^2} + \tfrac{512\sqrt{2}\pi^2}{r_u^3}+\tfrac{2^7\pi^3}{r_u^4} +
\tfrac{2^{11}\sqrt{2}\pi^4}{r_u^5}\big)^{-1}.
\end{split}
\end{equation}
Then \eqref{eq:layton0} holds in addition to \eqref{eq:acote0}, and therefore Theorem \ref{thm:binormalestim} provides the
existence of $\sigma \in \mathcal{C}(J_r\times \R,\R)$ such that  $\sigma_t:=\sigma(t,\cdot)$ is
a reparametrization of $\gamma_t$ for $\Gamma_t$ for all $t\in J_r$, and the estimate   
\begin{equation}\label{eq:quasifini0}
F_{\gamma_t,\sigma_t,r} \leq \exp(K_u|t|) F(0)\leq \exp(K_u|t|)\big( 1 + \tfrac{8^3\pi^2}{r_u^2}+ \tfrac{2^8\pi^4}{r_u^4}\big)
\int_{T^1}|v(0,s)-u(0,s)|^2\, ds
\end{equation}
for every $t\in J_r,$ where
$$
K_u := \frac{544}{r_u^2} + 6\|\partial_{ss}u\|_{L^\infty([-t,t]\times T^1)},\quad J_r:=[-T_r,T_r]\cap (-T,T),
\quad \text{and}\quad T_r :=  \tfrac{1}{K_u}
\log\left(\tfrac{F_r}{F(0)}\right).
$$
Using Lemma \ref{lem:aller}, we infer from \eqref{eq:quasifini0} that for every $t\in J_r,$ 
\begin{equation*}\label{eq:quasifini1}
\|v(t,\cdot)-u(t,\cdot+\sigma(t,0))\|^2_2 \leq C_{u,1}\exp(K|t|)\|v(0,\cdot)-u(0,\cdot)\|^2_2,
\end{equation*}
where
$$
C_{u,1}:= \big(4 + \tfrac{(1+2^{11})\pi^2}{r_u^2} +
\tfrac{16\sqrt{2}\pi^3}{r_u^3} + \tfrac{2^93\pi^4}{r_u^4} + \tfrac{2^{13}\sqrt{2}\pi^5}{r_u^5}
+\tfrac{2^8\pi^6}{r_u^6} + \tfrac{2^{12}\sqrt{2}\pi^7}{r_u^7}\big).
$$
On the other hand, if $T_r<T$ then for every $t$ in $(0,T)\setminus J_r$ we simply write  
\begin{equation}\label{eq:souris1}\begin{split}
\|v(t,\cdot)-u(t,\cdot+\sigma(T_r,0))\|^2_2 &\leq 8\pi = 8\pi\exp(-K_u|t|)\exp(K_u|t|)\\
&\leq 8\pi\exp(-KT_r)\exp(K_u|t|) = 8\pi\frac{F(0)}{F_r}\exp(K_u|t|)\\
&\leq C_{u,2} \exp(K_u|t|)\|v(0,\cdot)-u(0,\cdot)\|^2_2,
\end{split}
\end{equation}
where
$$
C_{u,2} := \big( 4 + \tfrac{2^6\sqrt{2}\pi}{r_u} + \tfrac{2^{11}\pi^2}{r_u^2}
+ \tfrac{2^{15}\sqrt{2}\pi^3}{r_u^3}+ \tfrac{2^{10}\pi^4}{r_u^4}+ \tfrac{2^{14}\sqrt{2}\pi^5}{r_u^5} \big).
$$
For $t$ in $(-T,0)\setminus J_r,$ we obtain the equivalent of \eqref{eq:souris1} replacing 
$\sigma(T_r,0)$ by $\sigma(-T_r,0).$ 

It remains to consider the situation where \eqref{eq:acote1} does not hold. There again, we simply write,
for any $t\in (-T,T),$
\begin{equation*}\label{eq:souris2}\begin{split}
\|v(t,\cdot)-u(t,\cdot)\|^2_2 &\leq 8\pi = 8\pi\|v(0,\cdot)-u(0,\cdot)\|^{-2}_2\|v(0,\cdot)-u(0,\cdot)\|^2_2\\
&\leq  C_{u,2} 
\|v(0,\cdot)-u(0,\cdot)\|^2_2\\
&\leq C_{u,2}\exp(K_u|t|) \|v(0,\cdot)-u(0,\cdot)\|^2_2.
\end{split}
\end{equation*}

Inspecting the constants $C_{u,1}$ $C_{u,2}$ and $K_u$, we conclude that for any $t\in (-T,T)$
\begin{equation*}
\|v(t,\cdot)-u(t,\cdot+\sigma(t))\|^2_2 \leq C_u \exp\big(K_u|t|\big)\|v(0,\cdot)-u(0,\cdot)\|^2_2,
\end{equation*}
where 
$$
C_u :=  \big(4 + \tfrac{2^6\sqrt{2}\pi}{r_u}+ \tfrac{(1+2^{11})\pi^2}{r_u^2} +
\tfrac{2^{15}\sqrt{2}\pi^3}{r_u^3} + \tfrac{2^93\pi^4}{r_u^4} + \tfrac{2^{14}\sqrt{2}\pi^5}{r_u^5}
+\tfrac{2^8\pi^6}{r_u^6} + \tfrac{2^{12}\sqrt{2}\pi^7}{r_u^7}\big)
$$
and where $\sigma$ is the continuous function on $(-T,T)$ defined by $\sigma\equiv0$ if \eqref{eq:acote1} does not hold
and $\sigma(t):=\sigma(t,0)$ for $t\in J_r$, $\sigma(t):=\sigma(-T_r,0)$ for $t\in(-T,-T_r),$ and
$\sigma(t):=\sigma(T_r,0)$ for $t\in (T_r,T)$ if \eqref{eq:acote1} holds. 

\smallskip

The constants $C_u$ and $K_u$ above are completely explicit, but at this stage they involve estimates on the 
solution $u$ on the whole  interval $[-T,T].$ In view of the Cauchy theory at the level of smooth solutions 
(see Section \ref{sect:cauchy}), we may then estimate both $C_u$ and $exp(K_u|t|)$ 
by a function $C(\|\partial_{sss}u(0,\cdot)\|_2,T).$ This completes the proof of Theorem \ref{thm:main} in
Case 1. 

\noindent
Notice that there wouldn't be any particular difficulty in turning $C(\|\partial_{sss}u(0,\cdot)\|_2,T)$
 into an explicit expression as well, but these expressions quickly become even more complicated (in
particular they involve double exponentials in time). For Case 2 and Case 3 in the sequel, we will rely on
 additional levels of approximation;  for convenience we will therefore abandon explicit expressions there 
even sooner than for Case 1. Without loss of generality, we will also consider $u$ as being globally defined in time.

\medskip

%%%%%%%%%%%%%%%%%%%%%%%%%%%%%%%%%%%%%%%%%%%%%%%%%%%%%%%%%%%%%%%%%%%%%%%%%%%%%%%%%%%%%%%%%%%%%%%%%%%%%%%%%%%%%%%%
\noindent{\bf Case 2: Assume \mb$|\int_{T^1}u\, |\geq 1/2.$\mn} In that case we will take advantage of the
symmetries of \eqref{eq:schrodimap} and reduce it to a case similar to Case 1. More precisely, we first infer from
the Cauchy-Schwarz inequality that
\begin{equation}\label{eq:mazette}
\Big | \int_{T^1}u - \int_{T^1} v \Big| \leq \sqrt{2\pi} \|v(0,\cdot)-u(0,\cdot)\|_2.  
\end{equation}
As in Case 1, if the inequality 
\begin{equation}\label{eq:acote1bis}
 \|v(0,\cdot)-u(0,\cdot)\|_2 < \frac{1}{4\sqrt{2\pi}} 
\end{equation}
does not hold, then we simply write for any $t\in (-T,T),$
\begin{equation*}\label{eq:souris6}\begin{split}
\|v(t,\cdot)-u(t,\cdot)\|^2_2 &\leq 8\pi = 8\pi\|v(0,\cdot)-u(0,\cdot)\|^{-2}_2\|v(0,\cdot)-u(0,\cdot)\|^2_2\\
&\leq  2^8\pi^2\|v(0,\cdot)-u(0,\cdot)\|^2_2.
\end{split}
\end{equation*}
Assume next that \eqref{eq:acote1bis} holds. 
Then  $| \int_{T^1}u - \int_{T^1} v|\leq 1/4$ and there exist $R_\theta\in SO(3)$ and $\kappa\in (1/2,3/2)$ 
 such that the function 
$$
\tilde u_0(s) := R_\theta u(0,\tfrac{s}{\kappa})
$$
defined on $T^1_\ell$ with $\ell = 2\pi\kappa$ satisfies $\int_{T^1_\ell}\tilde u = \int_{T^1}v$ and
moreover\footnote{These can be checked from easy planar geometry.} 
\begin{equation}\label{eq:vindeglace0}
\|R_\theta - {\rm Id}\|_{\mathcal{L}(\R^3,\R^3)} \leq 2\sqrt{2\pi} \|v(0,\cdot)-u(0,\cdot)\|_2, \qquad |\kappa - 1|
\leq 2\sqrt{2\pi} \|v(0,\cdot)-u(0,\cdot)\|_2.
\end{equation}

The unique (global) solution $\tilde u$ of \eqref{eq:schrodimap} with initial datum $\tilde u_0$ is given by
\begin{equation}\label{eq:presse0}
\tilde u(t,s) = R_\theta u\big(\tfrac{t}{\kappa^2},\tfrac{s}{\kappa}\big).
\end{equation}  
We apply Lemma \ref{lem:basepoint} to $\tilde u$ and $v$
successively, which yield solutions $\gamma_{\tilde u} \in \mathcal{C}(I,H^4_{\rm loc}(\R,\R))$ and $\gamma_v\in
\mathcal{C}(I,H^\frac32_{\rm loc}(\R,\R))$ of the binormal curvature flow equation. We then follow the exact same 
lines as in Case 1, the only minor difference being that here $2\pi=:L\neq \ell$ and we merely have, in view of
\eqref{eq:vindeglace0}, 
\begin{equation*}\label{eq:longproche}
|L-\ell|^2 \leq 32\pi^3\|v(0,\cdot)-u(0,\cdot)\|_2^2,
\end{equation*}
whose only effect is to affect the value of some numerical constants (when we apply Lemma \ref{lem:aller} and 
Lemma \ref{lem:retour} only). Doing so, we obtain, for every $t\in (-T,T),$  
\begin{equation}\label{eq:vindeglace1}
\int_0^{2\pi} \big|v(t,s)-\tilde u(t,\kappa s + \sigma(t))\big|^2\, ds \leq C_{\tilde u}\exp(K_{\tilde u}|t|)\int_0^{2\pi}
\big|v(0,s)-\tilde u(0,\kappa s )\big|^2\, ds
\end{equation}
for some function $\sigma\in\mathcal{C}(I,T^1_\ell),$ where
$$
C_{\tilde u} \leq C(\|\tilde u_0'''\|_2,T) \leq C(\|\partial_{sss}u(0,\cdot)\|_2,T). 
$$
In view of \eqref{eq:presse0}, \eqref{eq:vindeglace1} translates into
\begin{equation}\label{eq:cosma1}
\int_0^{2\pi} \big|v(t,s)-R_\theta u\big(\tfrac{t}{\kappa^2},s + \tfrac{1}{\kappa}\sigma(t)\big)\big|^2\, ds \leq
 C(\|\partial_{sss}u(0,\cdot)\|_2,T)\int_0^{2\pi}
\big|v(0,s)- R_\theta u(0,s)\big|^2\, ds.
\end{equation}
Now, by \eqref{eq:vindeglace0} we estimate
\begin{equation*}\label{eq:cosma2}
\int_0^{2\pi} \big|(R_\theta-{\rm Id}) u\big(\tfrac{t}{\kappa^2},s+ \tfrac{1}{\kappa}\sigma(t)\big)\big|^2\, ds
\leq 2\pi \|R_\theta - {\rm Id}\|_{\mathcal{L}(\R^3,\R^3)}^2 \leq 16\pi^2 \|v(0,\cdot)- u(0,\cdot)\|_2^2,
\end{equation*}
and
\begin{equation*}\label{eq:cosma3}
\int_0^{2\pi} \big|(R_\theta-{\rm Id}) u\big(0,s\big)\big|^2\, ds
\leq 2\pi \|R_\theta - {\rm Id}\|_{\mathcal{L}(\R^3,\R^3)}^2 \leq 16\pi^2 \|v(0,\cdot)-u(0,\cdot)\|_2^2.
\end{equation*}
Also, 
\begin{equation}\label{eq:cosma4}
\begin{split}
\int_0^{2\pi} \big| u\big(\tfrac{t}{\kappa^2},s+ \tfrac{1}{\kappa}\sigma(t)\big)-u\big(t,s+
\tfrac{1}{\kappa}\sigma(t)\big) \big|^2\, ds 
&\leq t^2 (1-\frac{1}{\kappa^2})^2 \|\partial_t
u\|_{L^\infty((-4T,4T),L^2(T^1))}^2 \\
&\leq C t^2 |1-\kappa|^2 \|\partial_{ss}
u\|_{L^\infty((-4T,4T),L^2(T^1))}^2 \\
&\leq C(\|\partial_{sss}u(0,\cdot)\|_2,T)  \|v(0,\cdot)-u(0,\cdot)\|_2^2.
\end{split}
\end{equation}
Combining \eqref{eq:cosma1}-\eqref{eq:cosma4} we complete the proof of Theorem \ref{thm:main} in Case 2.

\medskip

%%%%%%%%%%%%%%%%%%%%%%%%%%%%%%%%%%%%%%%%%%%%%%%%%%%%%%%%%%%%%%%%%%%%%%%%%%%%%%%%%%%%%%%%%%%%%%%%%%
\noindent{\bf Case 3: Assume \mb$|\int_{T^1}u\, |< 1/2.$\mn}
As in Cases 1 and 2, if the inequality 
\begin{equation}\label{eq:jklm}
 \|v(0,\cdot)-u(0,\cdot)\|_2 < \min\Big(\frac{1}{60}\sqrt{\frac{\pi}{2}}\|\partial_s u(0,\cdot)\|_\infty^{-1},
\frac{\sqrt{2\pi}}{5}\Big) 
\end{equation}
does not hold, then we simply write for any $t\in (-T,T),$
\begin{equation*}\label{eq:souris7}\begin{split}
\|v(t,\cdot)-u(t,\cdot)\|^2_2 &\leq 8\pi = 8\pi\|v(0,\cdot)-u(0,\cdot)\|^{-2}_2\|v(0,\cdot)-u(0,\cdot)\|^2_2\\
&\leq \max\big(2^83^25^2 \|\partial_s u(0,\cdot)\|_\infty^{2}, 2^35^2\big) \|v(0,\cdot)-u(0,\cdot)\|^2_2.
\end{split}
\end{equation*}
Assume next that \eqref{eq:jklm} holds. 

For $e\in S^2$ and $r\geq 0,$ we denote by $M_{e,r}$ the M\"obius transform on $S^2$ given by $M_{e,r} :=
S_e^{-1}\circ D_\alpha \circ S_e,$ where $\alpha:=1+r$, $S_e$ is the stereographic projection from $S^2\setminus\{e\}$ to 
the equatorial plane of $S^2$ (with $e$ as north pole), and
$D_\alpha$ is the dilation of magnitude $\alpha$ in that equatorial plane. More explicitly, in spherical coordinates
$(\varphi,\theta)$ on $S^2$ such that $\varphi=0$ corresponds the point $e$ we have 
$$
M_{e,r}(\varphi,\theta) := \left( 2 \arctan \Big( \frac{1}{\alpha}\tan\big(\frac{\varphi}{2}\big)\Big)\,,\,
\theta\, \right).
$$    
{\bf Claim:} There exist $e_0\in S^2$ and $r_0\geq 0$ such that $\check u_0 := M_{e_0,r_0}\circ u(0,\cdot)$ satisfies 
$\int_{T^1} \check u_0 = \int_{T^1} v$ and moreover we have the estimate
\begin{equation}\label{eq:riquiqui}
r_0 \leq \max\big(60\sqrt{2/\pi}\|\partial_s u(0,\cdot)\|_\infty,5/\sqrt{2\pi}\big) \|v(0,\cdot)-u(0,\cdot)\|_2.
\end{equation}
\begin{proof}[Proof of the claim]
Fix first $e\in S^2$ and $0 \leq  r \leq 1.$ Because of the assumption $| \int_{T^1}u \, |<1/2,$ the image of
$u(0,\cdot)$ cannot be contained entirely in the spherical cap $\{0\leq \varphi \leq \arccos(1/2) = \pi/3\},$
nor in its opposite spherical cap $\{2\pi/3\leq \varphi\leq \pi\}.$ Let therefore $s_e\in T^1$ be such that $\pi/3 \leq \varphi(u(0,s_e))\leq
2\pi/3.$ For every $s\in T^1$ such that $|s-s_e| \leq \pi/(12 \|\partial_su(0,\cdot)\|_\infty),$ we have
$\pi/4 \leq \varphi(u(0,s))\leq 3\pi/4.$ For every point $p$ on $S^2$, we have $M_{e,r}(p)\cdot e \geq p\cdot e$,
and for those points $p$ such that $\pi/4 \leq \varphi(p)\leq 3\pi/4$ we even have 
$ M_{e,r}(p)\cdot e \geq p\cdot e + cr$, where\footnote{One can compute the value of $c$ noticing that the 
worst case scenario corresponds to $r=1$ and $\varphi(\bar p)=\pi/4$, so that $M_{e,1}(\bar p)\cdot e = \cos(\varphi(M_{e,1}(\bar p)))=:\cos(\bar
\varphi)$ is the solution of the equation
$\sin(\bar\varphi)/(1-\cos(\bar\varphi))=2\sin(\pi/4)/(1-\cos(\pi/4))=2(\sqrt{2}-1)$ which one obtains by taking
squares of both sides and easy algebra.} 
$$
c:= \frac{4-\sqrt{17-12\sqrt{2}}}{7-2\sqrt{2}} - \frac{\sqrt{2}}{2} \geq \frac{1}{5}. 
$$   
In particular, we have
\begin{equation*}\label{eq:cabougepasmal}
e \cdot \int_{T^1} M_{e,r}(u(0,s))\, ds \geq e \cdot \int_{T^1} u(0,s)\, ds +
\min\Big(\frac{\pi}{60}\|\partial_su(0,\cdot)\|_\infty^{-1},\frac{2\pi}{5}\Big) r,  
\end{equation*}
and therefore by \eqref{eq:mazette}
 \begin{equation}\label{eq:cabougepasmalbis}
e \cdot \int_{T^1} M_{e,r}(u(0,s))- v(0,s)\, ds \geq -\sqrt{2\pi} \|v(0,\cdot)-u(0,\cdot)\|_2  +
\min\Big(\frac{\pi}{60}\|\partial_su(0,\cdot)\|_\infty^{-1},\frac{2\pi}{5}\Big)r. 
\end{equation}
Set $\bar r := \max\big(60\sqrt{2/\pi}\|\partial_s u(0,\cdot)\|_\infty,5/\sqrt{2\pi}\big) \|v(0,\cdot)-u(0,\cdot)\|_2$;
in particular from \eqref{eq:jklm} we see that $\bar r \leq 1.$ The function
$$
G\,:\, B[0,\bar r] \subset \R^3 \to \R^3,\ x \mapsto \left\{ 
\begin{array}{ll}
\ds \int_{T^1} M_{x/|x|,|x|}(u(0,s))- v(0,s)\, ds & \text{if } x\neq 0,\\[10pt]
\ds \int_{T^1} u(0,s)- v(0,s)\, ds & \text{if } x=0,
\end{array}
\right.
$$
is continuous on $B[0,\bar r]$, and in view of \eqref{eq:cabougepasmalbis}, it satisfies
$$
G(x)\cdot x \geq 0 \qquad \forall x\in \partial B[0,\bar r]. 
$$
If follows from the Poincar\'e-Bohl theorem \cite{Had} that $G$ has at least one zero $x_0$ in $B[0,\bar r].$  Our 
claim is proved choosing $r_0:=|x_0|$ and $e_0:=x_0/|x_0|$ (If $|x_0|=0 $ we are in Case 1).  
\end{proof}
Let $\check u \in \mathcal{C}(\R,H^3(T^1,S^2))$ denote the solution of \eqref{eq:schrodimap} with initial datum
$\check u_0.$ We first apply Theorem \ref{thm:main} Case 1 to $\check u$ and $v$. This yields the estimate
\begin{equation}\label{eq:vindeglace2}
\int_0^{2\pi} \big|v(t,s)-\check u(t, s + \sigma(t))\big|^2\, ds \leq C_{\check u}\exp(K_{\check u}|t|)\int_0^{2\pi}
\big|v(0,s)-\check u(0, s )\big|^2\, ds
\end{equation}
for every $t\in I$ and for some function $\sigma\in\mathcal{C}(I,T^1),$ where, since $M_{e,r} \in \mathcal{C}^\infty(S^2,S^2),$
$$
C_{\check u} \exp(K_{\check u}|t|)\leq C\big(\|\partial_{sss}\check u_0\|_2,T\big) \leq
C\big(\|\partial_{sss}u(0,\cdot)\|_2,T\big) .
$$
Next, using \eqref{eq:riquiqui} we infer that 
$$\|u(0,\cdot)-\check u_0\|_{H^1}^2 \leq
C\big( \|\partial_{s}u(0,\cdot)\|_\infty\big)\|v(0,\cdot)-u(0,\cdot)\|_2^2,
$$
and therefore from Corollary \ref{cor:controldiff} we obtain, for every $t\in I,$
\begin{equation}\label{eq:vindeglace3}\begin{split}
\|u(t,\cdot+\sigma(t))-\check u(t,\cdot+\sigma(t)\|_{2}^2 &=
\|u(t,\cdot)-\check u(t,\cdot)\|_{2}^2\\
& \leq \|u(t,\cdot)-\check u(t,\cdot)\|_{H^1}^2\\ 
&\leq C\big(\|\partial_{sss}u(0,\cdot)\|_2,T\big)\|v(0,\cdot)-u(0,\cdot)\|_2^2.
\end{split}
\end{equation}
Combining \eqref{eq:vindeglace2} and \eqref{eq:vindeglace3}, the proof of Theorem \ref{thm:main} in Case 3
is completed. \qed

%%%%%%%%%%%%%%%%%%%%%%%%%%%%%%%%%%%%%%%%%%%%%%%%%%%%%%%%%%%%%%%%%%%%%%%%%%%%%%%%%%%%%%%%%%%%%%%%%%%%%%%%%%%%%%%%
\section{Proof of Theorem \ref{thm:illposed}}\label{sect:th2}
%%%%%%%%%%%%%%%%%%%%%%%%%%%%%%%%%%%%%%%%%%%%%%%%%%%%%%%%%%%%%%%%%%%%%%%%%%%%%%%%%%%%%%%%%%%%%%%%%%%%%%%%%%%%%%%%

We begin with a simpler example than the one stated in Theorem \ref{thm:illposed}. This example in itself doesn't
imply the kind of non-continuity of the flow map which we point at, but hopefully it raises the intuition sufficiently
 so that the actual example will appear natural.  

\medskip

Consider, for $\eps$ and $p$ positive, the helix $\gamma_0\equiv \gamma_{0,\eps,p}$ of radius $\eps$ and pitch $2\pi p$ oriented along 
the $e_1$  direction:
$$
\gamma_0(s) = \left(\, p \mathfrak{s}  \, , \, \eps \cos\big( \mathfrak{s}\big)\,
 , \, \eps \sin\big( \mathfrak{s}\big)\, \right),
$$
where $\mathfrak{s}:= s /\sqrt{\eps^2+p^2}$ so that $\gamma_0$ is parametrized by arc-length. Explicit computations
 yield the identity
$$
\Big(\partial_{s}\gamma_0 \times \partial_{ss} \gamma_0\Big)(s) = -\frac{p}{\eps^2+p^2}\partial_s\gamma_0(s) +
\frac{1}{\sqrt{\eps^2+p^2}}(1,0,0).
$$
If we denote by $\gamma\equiv \gamma_{\eps,p}$ the solution of the binormal curvature flow equation \eqref{eq:strongbfbis} with initial datum
$\gamma_0$, then it follows that
$$
\gamma(t,s) = \gamma_0(\xi) + \frac{t}{\sqrt{\eps^2+p^2}}(1,0,0),
$$ 
where $\xi := s - p\,t/(\eps^2+p^2).$ In particular\footnote{The two fractions inside the large parenthesis
correspond to what will be called $-C$ and $\Omega$ in the actual example later.}
\begin{equation}\label{eq:deplacement}
\gamma(t,0) = \left( - \frac{p^2}{(\eps^2+p^2)^\frac32}  +  \frac{1}{(\eps^2+p^2)^\frac12}\right)(t,0,0) =
\frac{\eps^2}{(\eps^2+p^2)^\frac32}(t,0,0).
\end{equation}
Given $\sigma_0 >0$, if we consider now sequences $(\eps_m)_{m\in \N}$ and $(p_m)_{m\in \N}$ such that
$$
\eps_m\to0,\qquad p_m\to 0,\quad \text{and} \quad \eps_m^2/p_m^3 \to \sigma_0,
$$
as $m\to +\infty,$ then the sequence $(\gamma_{0,\eps_m,p_m})_{m\in \N}$ is bounded in $H^{\frac32}_{\rm loc}(\R,\R^3)$ and weakly
converges in $H^{\frac32}_{\rm loc}(\R,\R^3)$ to the function $\gamma_0^*(s):=(s,0,0)$, which is a stationary solution
of equation \eqref{eq:strongbfbis}. On the other hand, from \eqref{eq:deplacement} we infer that for every $t\in \R$ 
the sequence $(\gamma_{\eps_m,p_m}(t,\cdot))_{m\in \N}$ 
weakly converges in  $H^{\frac32}_{\rm loc}(\R,\R^3)$ to the function $\gamma_0^*(s+\sigma_0t).$ If we require
moreover that $1/\sqrt{\eps_m^2+p_m^2} \in \N$ (e.g. by choosing $\eps_m^2 := \sigma_0 p_m^3$ and $p_m$ as the
unique positive solution of $\sigma_0 p_m^3+p_m^2 = m^{-2}$) then the functions $u_n:=\partial_s
\gamma_{\eps_m,p_m}$ are solutions of the Schrodinger map equation \eqref{eq:schrodimap} such that
$u_m(0,\cdot)$ weakly converge in $H^\frac12(T^1,S^2)$ as $m\to +\infty$ to the stationary solution $u^*(0,s):=(1,0,0)$ and
such that $u_m(t,\cdot)$ weakly  converge in $H^\frac12(T^1,S^2)$ as $m\to +\infty$ to $u^*(t,\cdot + \sigma_0t).$

\smallskip

Of course here, $u^*$ is a constant function of space and time, and in particular  $u^*(t,\cdot + \sigma_0t) =
u^*(t,\cdot).$  In order to exhibit a sequence having the properties stated in Theorem \ref{thm:illposed}, 
 we build up on the previous example modifying it in such a way that the limit solution $u^*$ is no longer
constant but satisfies $u^*(t,s) = (\cos(s),\sin(s),0).$  Roughly speaking, it suffices to wrap the original helix
$\gamma_0$ around a fixed unit circle in the $(e_1,e_2)$-plane instead of letting it extend along the $e_1$-axis.
 One can probably do so just by a simple ansatz and then some possibly more delicate analysis. Instead, we will
rely on known explicit traveling wave solutions of \eqref{eq:strongbfbis} that behave exactly as wrapped helices.
We describe them now.       

\smallskip 
  
In 1981, Kida \cite{Ki} studied the set of initial curves in $\R^3$ for which the solution map of the binormal curvature
flow equation reduces to a family of rigid motions. Using the symmetries of \eqref{eq:strongbfbis} and its
conservations laws, such motions are
necessarily the superposition of a constant speed rotation around a fixed axis (which we may always assume to be
the $x_3$-axis after a fixed rotation) and a constant speed translation parallel to that same axis. Following
\cite{Ki}, we denote
by $\Omega$ and $V$ the speeds of rotation around $e_3$ and of translation along $e_3$ respectively, and by $C$ 
the speed of the slipping motion of the curve along itself\footnote{Even though the speed given by the binormal curvature flow equation
is perpendicular to the tangent vector, in a non-orthogonal frame like $(\partial_s\gamma,\gamma \times e_3 ,e_3)$
the component $C$ of $\partial_t \gamma$ along $\partial_s\gamma$ may not be zero. For the rigid motions we
consider, $C$ is a constant function of space and time.}, so that
$$
\gamma(t,s) = \gamma(0,s-Ct)\cdot \left(\begin{smallmatrix}
\cos(\Omega t) & \sin(\Omega t) & 0\\
-\sin(\Omega t) & \cos(\Omega t) & 0\\
0 & 0 & 1 
\end{smallmatrix}
\right)
 + V(0,0,t)
$$
and therefore $\gamma$ satisfies the additional equation
\begin{equation}\label{eq:vitessedecomp}\partial_t \gamma = -C \partial_s \gamma + \Omega e_3\times \gamma + Ve_3.
\end{equation}
Combining \eqref{eq:strongbfbis} and \eqref{eq:vitessedecomp} we write
\begin{equation}\label{eq:firstnation0}
\partial_s\gamma \times \partial_{ss} \gamma =  -C \partial_s \gamma + \Omega e_3\times \gamma + Ve_3.
\end{equation}
Taking the scalar product of \eqref{eq:firstnation0} with $\partial_s\gamma$  yields
\begin{equation}\label{eq:firstnation1}
- C + \Omega r^2\theta' + Vz' = 0,
\end{equation}
where we wrote $(\gamma^1(0,\cdot),\gamma^2(0,\cdot)) =: (r(\cdot)\cos(\theta(\cdot)),r(\cdot)\sin(\theta(\cdot)))$
and $\gamma^3(0,\cdot) =: z(\cdot).$ Taking the vector product of \eqref{eq:firstnation0} with 
$\partial_s \gamma$ instead yields 
\begin{equation}\label{eq:firstnation2}
z'' = -\Omega r r' = - \frac{\Omega}{2} R',
\end{equation}
where $R = r^2.$  After integration, \eqref{eq:firstnation2} leads to 
\begin{equation}\label{eq:firstnation3}
z' = \frac{\Omega}{2}\left[ A - R\right]
\end{equation}
for some integration constant $A\in \R.$ Assume that $\Omega\neq 0.$ Combining \eqref{eq:firstnation1} and
 \eqref{eq:firstnation3} we obtain
\begin{equation}\label{eq:firstnation4}
\theta ' = \frac{1}{2} V + (C-\frac{1}{2}AV\Omega)/(\Omega R),
\end{equation}
and then combining \eqref{eq:strongbfbis} with \eqref{eq:firstnation3} and \eqref{eq:firstnation4} we finally
obtain
\begin{equation}\label{eq:firstnation5}
(R')^2 + f(R) =0,
\end{equation}
where
\begin{equation}\label{eq:cestf}
f(R) = \Omega^2 R^3 + (V^2-2A\Omega^2)R^2 + (4V(C-\frac{1}{2}AV\Omega)/\Omega+\Omega^2A^2-4)R +
4(C-\frac{1}{2}AV\Omega)^2/\Omega^2.
\end{equation}
We require the parameters $A,C,V,\Omega$ to be such that $f$ has two positive roots and one negative root,
and we then write
\begin{equation}\label{eq:cestf2}
f(R)=: \Omega^2(R-\alpha)(R-\beta)(R+\delta)
\end{equation}
where $0<\beta\leq \alpha$ and $0<\delta.$ The solution of equation \eqref{eq:firstnation5} is given in terms of the
Jacobi elliptic function ${\rm sn}$ by
\begin{equation}\label{eq:R}
R(s) = \alpha + (\beta-\alpha) {\rm sn}^2(\frac{1}{2}(\alpha+\delta)^\frac12 \Omega s | k)
\end{equation}
where $k := ((\alpha-\beta)/(\alpha+\delta))^\frac12\in [0,1).$   Inserting \eqref{eq:R} in \eqref{eq:firstnation3}
and then integrating in $s$ yields 
\begin{equation}\label{eq:Z}
z(s) = z(0) + \frac{\Omega}{2}(A+\delta)s - (\alpha+\delta)^\frac12 E(\frac{1}{2}(\alpha+\delta)^\frac12 \Omega s |
k ),
\end{equation}
where $E(u|k) := \int_0^u {\rm dn}^2(u',k)du' = \int_0^u (1-k^2{\rm sn}^2(u',k))du'$ is the incomplete elliptic integral of second kind. 
Inserting \eqref{eq:R} in \eqref{eq:firstnation4} and then integrating in $s$ yields 
$$
\theta(s) = \theta(0) + \frac{1}{2}Vs + \frac{2C-AV\Omega}{\alpha(\alpha+\delta)^\frac12 \Omega^2}
\Pi(\frac12(\alpha+\delta)^\frac12\Omega s | \frac{\alpha-\beta}{\alpha},k),
$$ 
where $\Pi(u|\ell,k):= \int_0^u (1-\ell{\rm sn}^2(u',k))^{-1}du'$ is the incomplete elliptic integral of third kind. 
Notice that if $\alpha=\beta$ then $k=0$, and if moreover $\alpha=\beta=A=1$ then the corresponding solution $\gamma$ is a circle
of radius one in the $(x_1,x_2)$-plane. For $0<k<<1$, 
 in particular if $\alpha+\delta\simeq 1$ and  $0<\alpha-\beta<<1,$ the Jacobi elliptic function ${\rm sn}(\cdot,k)$ resembles 
the $\sin$ function. We wish therefore to find coefficients $A,C,V,\Omega$ such that $f$ vanishes on the
negative axis, such that $\alpha+\delta\simeq 1,$ $0<\alpha-\beta<<1$, $\Omega>>1,$ and $(\alpha-\beta)^2\simeq
\Omega^{-3}$, so that the corresponding solutions $\gamma$ behave like the wrapped
helices which we discussed at the beginning of this section, with pitch $\simeq 2\pi\Omega^{-1}$ and radius $\simeq
\alpha-\beta.$ We need also that $\partial_s\gamma$ is well-defined
on $T^1.$ We explain the details now.

The function $R$ is periodic with period
\begin{equation}\label{eq:TR}
T_R := \frac{4K(k)}{(\alpha+\delta)^\frac12 \Omega}.
\end{equation}
The function $z$ satisfies
$$
z(s+T_R) = z(s) + \frac{\Omega}{2}(A+\delta)T_R - 2(\alpha+\delta)^\frac12 E(k),
$$
where $E(k):=E(K(k)|k)$ is the complete elliptic integral of the second kind. In particular
\begin{equation}\label{eq:schrimp0}
z'(s+T_R) = z'(s)\qquad \forall s\in \R.
\end{equation}
The function $\theta$ satisfies
\begin{equation}\label{eq:schrimp1}
\theta(s+T_R) = \theta(s) + \frac{1}{2}VT_R +
2 \frac{2C-AV\Omega}{\alpha(\alpha+\delta)^\frac12\Omega^2}\Pi(\frac{\alpha-\beta}{\alpha},k)
\end{equation}
where $\Pi(\ell,k):=\Pi(K(k)|l,k)$ is the complete elliptic integral of the third kind. 

We fix $m\in \N\setminus\{0\}$ and we require that
\begin{equation}\label{eq:coco1}
\theta(s+T_R) = \theta(s) + \frac{2\pi}{m},
\end{equation}
so that both $R$ and $\exp(i\theta)$ are $mT_R$-periodic. Since $z$ 
satisfies \eqref{eq:schrimp0}, we infer that if \eqref{eq:coco1} holds, then
$\partial_s \gamma$ is well-defined on $T^1_\ell,$ where 
$$
\ell := mT_R.
$$
At the end, we will also require that $\ell = 2\pi,$ but we found more convenient 
to let $\ell$ as a free parameter at this stage and to use scaling properties at the very last step.
In view of \eqref{eq:schrimp1}, our requirement \eqref{eq:coco1} translates into
\begin{equation}\label{eq:schrimp2}
\frac{2K(k)V}{(\alpha+\delta)^\frac12\Omega} + 2
\frac{2C-AV\Omega}{\alpha(\alpha+\delta)^\frac12\Omega^2}\Pi(\frac{\alpha-\beta}{\alpha},k) = \frac{2\pi}{m}.
\end{equation}
%while \eqref{eq:coco2} translates into
%\begin{equation}\label{eq:coco3}
%\Omega \equiv \Omega_m := \frac{2K(k)m}{\pi(\alpha+\delta)^\frac12}.
%\end{equation}
%Using \eqref{eq:coco3}, \eqref{eq:schrimp2} may be rephrased as
%\begin{equation}\label{eq:coco4}
%C- \frac{1}{2}AV\Omega_m = \frac{\alpha K(k) (2-V)}{2\Pi(\frac{\alpha-\beta}{\alpha},k)}.
%\end{equation}

Equating the coefficients of order zero in $R$ in \eqref{eq:cestf} and \eqref{eq:cestf2}, we obtain
\begin{equation}\label{eq:mouche0}
C-\frac12 AV\Omega = \frac{\Omega^2}{2}(\alpha\beta\delta)^\frac12.
\end{equation}
Inserting \eqref{eq:mouche0} in \eqref{eq:schrimp2}, we obtain
\begin{equation}\label{eq:ratioVOmega}
\frac{V}{\Omega} = \left[ \frac{2\pi}{m} -
2\frac{\beta^\frac12\delta^\frac12}{\alpha^\frac12(\alpha+\delta)^\frac12}
\Pi\Big(\frac{\alpha-\beta}{\alpha},k\Big)\right] \frac{(\alpha+\delta)^\frac12}{2K(k)},
\end{equation}
which determines the ratio $V/\Omega$ in terms of $\alpha,\beta$ and $\delta$ only\footnote{Recall that $k$ is a
function of $\alpha,\beta,\delta.$}.

Equating the coefficients of order two in $R$ in \eqref{eq:cestf} and \eqref{eq:cestf2}, we obtain
\begin{equation}\label{eq:crayon0}
A = \frac{1}{2}\left(\alpha+\beta-\delta +\frac{V^2}{\Omega^2}\right),
\end{equation}
and therefore using \eqref{eq:ratioVOmega} we obtain
\begin{equation}\label{eq:valeurA}
A = \frac{1}{2}\left( \alpha+\beta-\delta +\left[ \frac{2\pi}{m} -
2\frac{\beta^\frac12\delta^\frac12}{\alpha^\frac12(\alpha+\delta)^\frac12}
\Pi\Big(\frac{\alpha-\beta}{\alpha},k\Big)\right]^2 \frac{\alpha+\delta}{4K^2(k)},
\right) 
\end{equation} 
so that $A$ too is determined in terms of $\alpha,\beta$ and $\delta$ only.

Finally, equating the coefficients of order one in $R$ in \eqref{eq:cestf} and \eqref{eq:cestf2}, we obtain
\begin{equation}\label{eq:montre0}
\Omega^2(\alpha\beta-\alpha\delta-\beta\delta) = 4V(C-\frac{1}{2}AV\Omega)/\Omega + \Omega^2A^2 - 4,
\end{equation}
and therefore using \eqref{eq:mouche0}, \eqref{eq:ratioVOmega} and \eqref{eq:valeurA} we obtain the equation
\begin{equation}\label{eq:cestg}
g(\alpha,\beta,\delta)\Omega^2 = 4,
\end{equation}
where
\begin{equation*}\begin{split}
g(\alpha,\beta,\delta)&:= \frac{1}{4}\left( \alpha+\beta-\delta +\left[ \frac{2\pi}{m} -
2\frac{\beta^\frac12\delta^\frac12}{\alpha^\frac12(\alpha+\delta)^\frac12}
\Pi\Big(\frac{\alpha-\beta}{\alpha},k\Big)\right]^2 \frac{\alpha+\delta}{4K^2(k)}
\right)^2 \\
&+ 2\alpha^\frac12\beta^\frac12\delta^\frac12 \left[ \frac{2\pi}{m} -
2\frac{\beta^\frac12\delta^\frac12}{\alpha^\frac12(\alpha+\delta)^\frac12}
\Pi\Big(\frac{\alpha-\beta}{\alpha},k\Big)\right] \frac{(\alpha+\delta)^\frac12}{2K(k)}\\
& -\alpha\beta + \alpha\delta + \beta\delta.
\end{split}\end{equation*}
Hence, $\Omega$ is determined in terms of $\alpha,\beta,\delta,$ provided $g\neq 0.$

We assume $m\geq 2$ (we are ultimately interested only in the case $m\to +\infty$) and we start from the 
 solution\footnote{Notice that $\Pi(0,0)=K(0)=\pi/2.$} of \eqref{eq:cestg} given by
$$
\alpha=\beta:=1, \quad \delta:=\frac{1}{m^2 - 1},\qquad \Omega :=  \sqrt{m^2-1},
$$ 
which corresponds  to
$$
A=1,\qquad V=1,\qquad C = \sqrt{m^2-1},
$$
and for which $\ell=mT_R=2\pi$ and $g(\alpha,\beta,\delta)=g(1,1,\tfrac{1}{m^2-1})=\tfrac{4}{m^2-1}.$ By continuity, $g$ does not
vanish for $(\alpha,\beta,\delta)$ in a neighborhood\footnote{A neighborhood in the sector
$\{\alpha\geq\beta>0\}\cap\{\delta>0\}$, that is the domain of $g$.} of $(1,1,\tfrac{1}{m^2-1})$, and as long as it remains
 so, $\Omega$ and then $V$, $A$ and $C$ successively defined by \eqref{eq:cestg}, \eqref{eq:ratioVOmega}, \eqref{eq:valeurA} and
 \eqref{eq:mouche0} are coefficients such that both \eqref{eq:cestf} and \eqref{eq:cestf2} are satisfied and
moreover \eqref{eq:coco1} holds. Since we wish to be able to reach solutions of the form $(\alpha-\beta)^2\simeq
m^{-3}$ for $m$ large, we need slightly more than just a continuity argument.  For that purpose, we restrict our
attention to the one-parameter family 
$$
\alpha := \alpha_{\epsilon,m} \equiv 1,\qquad \beta:=\beta_{\epsilon,m} \equiv 1-\epsilon,\qquad \delta :=
\delta_{\epsilon,m}\equiv \frac{1}{m^2-1}, 
$$ 
where $\epsilon>0$. We compute after elementary algebra 
\begin{equation}\label{eq:onepara}\begin{split}
g(\alpha_{\epsilon,m},\beta_{\epsilon,m},\delta_{\epsilon,m}) &= \frac{4}{m^2-1} + \frac{\epsilon^2}{4} -
\frac{\epsilon}{m^2-1}\\
&+ \frac{1}{4(m^2-1)^2}\left[ \left(
\frac{\pi - \sqrt{1-\epsilon}\Pi(\epsilon,\epsilon\sqrt{m^2/(m^2-1)})}{K(\epsilon\sqrt{m^2/(m^2-1)})}\right)^2 -1\right]^2\\
&+\frac{1}{2}(2-\epsilon)\frac{1}{m^2-1}\left[ \left( \frac{\pi -
\sqrt{1-\epsilon}\Pi(\epsilon,\epsilon\sqrt{m^2/(m^2-1)})}{K(\epsilon\sqrt{m^2/(m^2-1)})}\right)^2 -1  \right]\\
&+ \frac{2}{m^2-1}\left[\frac{\pi -
\sqrt{1-\epsilon}\Pi(\epsilon,\epsilon\sqrt{m^2/(m^2-1)})}{K(\epsilon\sqrt{m^2/(m^2-1)})}\sqrt{1-\epsilon}
-1\right].
\end{split}
\end{equation}
Each of the squared brackets in the previous expression vanishes for $\epsilon = 0$, and therefore we may estimate
them as $O(\epsilon)$ as $\epsilon \to 0,$ independently of $m\geq 2$ (indeed $m^2/(m^2-1)$ remains bounded as
$m\to +\infty$ and $\Pi$ and $K$ are continuously differentiable functions up to $(0,0)$ and $0$ respectively). Hence,
$$
\Big|g(\alpha_{\epsilon,m},\beta_{\epsilon,m},\delta_{\epsilon,m})- \frac{4}{m^2-1}\Big| = O(\epsilon^2) +
O\Big(\frac{\eps}{m^2}\Big),\qquad \text{as } \epsilon\to 0 \text{ and } m\to +\infty.
$$ 
In particular, if
\begin{equation}\label{eq:validrange}
\epsilon \equiv \epsilon(m) = o\Big(\frac{1}{m}\Big) \qquad\text{as } m\to +\infty, 
\end{equation}
then setting $\alpha_m := \alpha_{\epsilon(m),m},$ $\beta_m := \beta_{\epsilon(m),m},$ $\delta_m :=
\delta_{\epsilon(m),m}$, we have 
$$
g_m:= g(\alpha_m,\beta_m,\delta_m) = \frac{1}{m^2-1}(4+o(1)) \qquad
\text{as } m\to +\infty,
$$
which yields
$$
\Omega_m = \sqrt{m^2-1} (1+o(1)) \qquad
\text{as } m\to +\infty,
$$
and therefore from \eqref{eq:TR}
$$
\ell_m = mT_R = 2\pi + o(1) \qquad
\text{as } m\to +\infty.
$$
We are now in position to complete the proof of Theorem \ref{thm:illposed}. For $\tilde \sigma_0>0$ fixed and for
$m\geq 2$, we set $\epsilon(m):=\tilde \sigma_0 m^{-\frac32}.$ Provided $m$ is sufficiently large, $g_m \neq 0$ and 
the above construction yields an initial data $\gamma_m$ for \eqref{eq:strongbfbis} such that $\partial_s\gamma_m$ 
is periodic with period $2\pi+o(1)$ and such that the corresponding solution of \eqref{eq:strongbfbis} 
is a rigid motion consisting of a parallel translation at speed $C_m$ along $e_3$ and a rotation of angular speed
$\Omega_m$ around $e_3.$ By an homothety centered at zero and of factor $1+o(1)$, we transform $\gamma_m$ to 
$\tilde \gamma_m$ such that $\partial_s\tilde \gamma_m$ 
becomes exactly $2\pi$-periodic. In view of \eqref{eq:vitessedecomp} and the fact that $\partial_s\tilde \gamma_m =
e_3\times \tilde\gamma_m + o(1)$ in our example, we are particularly interested in the asymptotic of $\Omega_m -
C_m$, which will yield $\sigma_0$ in the statement of Theorem \ref{thm:illposed}. Using \eqref{eq:mouche0}, we write
$$
\Omega_m - C_m = \Omega_m - \frac{1}{2}A_m\frac{V_m}{\Omega_m}\Omega_m^2 - \frac{\Omega_m^2}{2}(\alpha_m\beta_m\delta_m)^\frac12.
$$  
We expand using \eqref{eq:ratioVOmega}, \eqref{eq:valeurA}, \eqref{eq:onepara} and \eqref{eq:cestg},
$$
A_m = 1 - \frac{1}{2}\tilde\sigma_0m^{-\frac32} + O(m^{-\frac52}),\qquad \frac{V_m}{\Omega_m}=\frac{1}{m} +
O(m^{-\frac32}),\qquad \Omega_m^2 = m^2 - \frac{\tilde\sigma_0^2}{16}m + O(1),
$$
$$
\Omega_m = m -\frac{\tilde\sigma_0^2}{32} + O(m^{-1}), \qquad\text{and}\qquad (\alpha_m\beta_m\delta_m)^\frac12 = 1/m +
O(m^{-5/2}),
$$ 
as $m\to +\infty,$ so that finally
$$
\Omega_m - C_m = \frac{\tilde\sigma_0^2}{4} + o(1) \qquad \text{ as } m\to +\infty. 
$$
Therefore, for $\sigma_0:={\tilde\sigma_0}^2/4$ the statement of Theorem \ref{thm:illposed} follows defining
 $u_{m,\sigma_0} := \partial_s\tilde\gamma_{m,\tilde\sigma_0}.$ For negative $\sigma_0$, 
it suffices to use the $\Omega \mapsto -\Omega$, $C\mapsto -C$ symmetry, and for the case $\sigma_0 = 0$  
it suffices to consider the trivial example $u_{m,0}:=u^*.$ \qed

\appendix

\numberwithin{lem}{section}
\numberwithin{cor}{section}
\numberwithin{prop}{section}
\numberwithin{equation}{section}

%%%%%%%%%%%%%%%%%%%%%%%%%%%%%%%%%%%%%%%%%%%%%%%%%%%%%%%%%%%%%%%%%%%%%%%%%%%%%%%%%%%%%%%%%%%%%%%%%%%%%%%%%%%%%%%%
\section{On the flow map for initial data in $H^3$}\label{sect:cauchy}
%%%%%%%%%%%%%%%%%%%%%%%%%%%%%%%%%%%%%%%%%%%%%%%%%%%%%%%%%%%%%%%%%%%%%%%%%%%%%%%%%%%%%%%%%%%%%%%%%%%%%%%%%%%%%%%%

In this Appendix, we briefly recall some properties of the flow map of \eqref{eq:schrodimap} for initial
data in $H^3(T^1_\ell,S^2)$ for some $\ell>0.$ Most if not all of them are presumably well-known, and we certainly do not claim originality nor 
optimality here. Nevertheless, since for some of them we couldn't always find the arguments which we wanted in the 
literature, we included sketch of proofs at the end of the section for the sake of completeness. 
We state them first for solutions that are a priori infinitely smooth; then the stability estimate of Corollary 
\ref{cor:controldiff} allows to extend some of then to solutions with initial data in $H^3.$ 

\smallskip
 
We denote by $C$, followed possibly by some arguments between round parentheses, a numerical constant whose actual value may
 change from line to line, but to which a (sometimes complicated) explicit value could be attributed a posteriori, depending only 
on its arguments.

\begin{lem}\label{lem:conserved}
For a smooth solution $u$ of \eqref{eq:schrodimap}, we have the following conserved quantities\footnote{As
mentioned in the Introduction, there are actually an infinite number of independent conserved quantities. The
additional ones, involving higher and higher order derivatives of the solution, seem however to lack some form 
of coerciveness in order to control these derivatives.}   
$$
E(u) := \int_{T^1} |\partial_s u|^2\, ds
\quad\text{and}\quad
I(u) := \int_{T^1} \big|\partial_{t} u\big|^2 + \big|\partial_{ss} u\big|^2 -
\frac{3}{4}\big|\partial_s u\big|^4\, ds.
$$
\end{lem}

\begin{cor}\label{cor:controlH2}
For a smooth solution $u$ of \eqref{eq:schrodimap} with initial datum $u^0$ and for every $t\in \R,$
$$
\big\| \partial_{ss}u(t,\cdot) \big\|_{2}^2 \leq   4 \big\| \partial_{ss}u^0\big\|_{2}^2 + 
2 \big\| \partial_s u^0\big\|_{2}^6.
$$
\end{cor}

\begin{lem}\label{lem:controlH3}
For a smooth solution $u$ of \eqref{eq:schrodimap} and for every $t\in \R,$
$$
\frac{d}{dt} \left(\|\partial_{sss}u\|^2_2 + \|\partial_{ts}u\|_2^2\right) \leq C(\ell) \|\partial_{ss} u\|_2^2
\left(\|\partial_{sss}u\|^2_2 + \|\partial_{ts}u\|_2^2\right). 
$$
\end{lem}

\begin{cor}\label{cor:controlH3}
For a smooth solution $u$ of \eqref{eq:schrodimap} with initial data $u^0$, for every $T>0$ and for every $t\in [-T,T]$
\begin{equation*}
\begin{split}
\|\partial_{sss} u(t,\cdot)\|_{2}^2 &\leq C(\ell)\Big( \|\partial_{ss}u^0\|_2^4 + \|\partial_{sss} u^0\|_{2}^2 \Big) \exp\Big(
C(\ell)\Big[\|\partial_{s}u^0\|_2^6 + \|\partial_{ss} u^0\|_{2}^2\Big]|t|\Big)\\
&\leq C(\ell, \|\partial_{sss} u^0\|_{2},T).
\end{split}
\end{equation*}
\end{cor}

\begin{lem}\label{lem:difference}
Let $u\,,$ $\check u$ be two smooth solutions of \eqref{eq:schrodimap}. Set $w:=u-\check u.$ 
Then for every  $t\in \R$,
\begin{align*}
\frac{d}{dt} \big\|w\big\|_{2}^2 &\leq  C\big(\ell, \|\partial_{sss}u\|_2, \|\partial_{sss}\check u\|_2\big)
\Big(\big\|w\big\|_{2}^2 + \big\|\partial_s
w\big\|_{2}^2 \Big),\\
\frac{d}{dt} \big\|\partial_s w\big\|_{2}^2 &\leq  C\big(\ell, \|\partial_{sss}u\|_2, \|\partial_{sss}\check u\|_2\big)
 \Big(\big\|w \big\|_{2}^2 +  \big\|\partial_s
w\big\|_{2}^2 \Big),\\
\frac{d}{dt} \Big( \big\|\partial_{ss}w\big\|_{2}^2 + \big\|\partial_t w\big\|_2^2\Big)
&\leq C\big(\ell, \|\partial_{sss}u\|_2, \|\partial_{sss}\check u\|_2\big) \Big(
\big\|\partial_{ss}w \big\|_{2}^2 + \big\|\partial_t w\big\|_2^2\Big).
\end{align*}
\end{lem}

\begin{cor}\label{cor:controldiff}
Let $u\, ,$ $\check u$ be two smooth solutions of \eqref{eq:schrodimap} with initial data $u^0$ and $\check u^0$ respectively.
Then for every $T>0$ and for every $t\in [-T,T]$
\begin{align*}
\big\|u-\check u \big\|_{H^1}^2 &\leq C\big(\ell,\|\partial_{sss}u^0\|_2,\|\partial_{sss}\check u^0\|_2,T \big)
\big\|u^0-\check u^0\big\|_{H^1}^2,\\
\big\|u-\check u \big\|_{H^2}^2 &\leq C\big(\ell,\|\partial_{sss}u^0\|_2,\|\partial_{sss}\check u^0\|_2,T \big)
\big\|u^0-\check u^0\big\|_{H^2}^2.
\end{align*}
\end{cor}

Corollary \ref{cor:controldiff} implies that the restriction of the flow map to smooth initial data is Lipschtizian 
both for the $H^1$ and the $H^2$ norm, with a Lispchitz constant depending on the $H^3$ norm\footnote{It follows
from McGahagan \cite{McG2} that the Lipschitz constant of the flow map for the $H^1$ norm depends only on a bound 
on the $H^2$ norm (instead of the $H^3$ norm).  In particular this allows to claim uniqueness and Lipschitz
continuity of the flow map in $H^2$ for the $H^1$ norm. The proof in \cite{McG2} is more involved than the direct proof
of Corollary \ref{cor:controldiff} below, and the latter is sufficient for our needs here.}. By approximation, 
the same remains true for the restriction of the flow map to initial data in $H^3.$ In turn, this implies that
$E$ and $I$ are conserved quantities for initial data in $H^3$, and that Corollary \ref{cor:controlH2} and
\ref{cor:controlH3} hold for $u$ and $\check u$ in $H^3.$    

\smallskip

\noindent{\bf Proof of Lemma \ref{lem:conserved}.} This is direct differentiation.\qed

\smallskip

\noindent{\bf Proof of Corollary \ref{cor:controlH2}.}
Having in mind the conservation law for $I$, we first write
$$
\int_{T^1} \big|\partial_s u\big|^4\, ds \leq \int_{T^1} \big|\partial_s u^1\big|^4 +\big|\partial_s
u^3\big|^4+\big|\partial_s u^3\big|^4  \, ds, 
$$
and since $\partial_su^j$ has zero mean for each $j\in\{1,2,3\},$ the Gagliardo-Nirenberg inequality yields (see
Agueh \cite{Agu} to compute the optimal\footnote{And, in particular, ``some'' !} constant)
$$
\int_{T^1} \big|\partial_s u^j\big|^4 \leq \frac{1}{\sqrt{3}}\Big(\int_{T^1} \big|\partial_s u^j\big|^2\Big)^\frac32 \Big(\int_{T^1} \big|\partial_{ss}
u^j\big|^2\Big)^\frac12,
$$
so that after summation
$$
\int_{T^1} \big|\partial_s u\big|^4\, ds \leq \sqrt{3}  \Big(\int_{T^1} \big|\partial_s u\big|^2\Big)^\frac32 \Big(\int_{T^1} \big|\partial_{ss}
u\big|^2\Big)^\frac12.
$$
Hence, using the conservation law for $E$,  
\begin{equation*}\begin{split}
\big\| \partial_{ss}u(t,\cdot)\big\|_{2}^2 &\leq I(u) + \frac{3\sqrt{3}}{4} \big\|
\partial_{s}u(t,\cdot)\big\|_{2}^3 \big\| \partial_{ss}u(t,\cdot)\big\|_{2} \\
&\leq  I(u) + \frac{3^3}{2^5} \big\|
\partial_{s}u(t,\cdot)\big\|_{2}^6 + \frac{1}{2} \big\| \partial_{ss}u(t,\cdot)\big\|_{2}^2\\
&= I(u) + \frac{3^3}{2^5} \big\|
\partial_{s}u(0,\cdot)\big\|_{2}^6 + \frac{1}{2} \big\| \partial_{ss}u(t,\cdot)\big\|_{2}^2,
\end{split}
\end{equation*}
and we may absorb the last term on the right-hand side by the left-hand side. Since conversely
$$
I(u) = I(u(0,\cdot)) \leq \int_{T^1} \big|u(0,\cdot) \times \partial_{ss}u(0,\cdot)\big|^2 + \big|\partial_{ss}
u(0,\cdot)\big|^2\, ds \leq 2 \big\| \partial_{ss}u(0,\cdot)\big\|_{2}^2,
$$
the conclusion follows.
\qed

\smallskip

\noindent{\bf Proof of Lemma \ref{lem:controlH3}.} Differentiating \eqref{eq:schrodimap} with respect to $t$ and
using the identity $|u|\equiv1$ yields the equation
\begin{equation}\label{eq:degre2}
\partial_{tt}u + \partial_{ssss} u = -2 \partial_{ss}\left(|\partial_s u|^2 u\right) + \partial_s\left(
|\partial_su|^2\partial_s u\right).
\end{equation}
Differentiating  \eqref{eq:degre2} with respect to $s$, taking the $L^2$-scalar product of the resulting equation
by $\partial_{ts}u$ and integrating by parts leads to
\begin{equation}\label{eq:evol3d}
\frac{d}{dt} \int_{T^1} \frac{|\partial_{ts}{u}|^2}{2} + \frac{|\partial_{sss}{u}|^2}{2} \, ds \leq -2\int_{T^1}
\partial_{sss}\left(|\partial_s u|^2 u\right)\partial_{ts}u\, ds + \int_{T^1}\partial_{ss}\left(
|\partial_su|^2\partial_s u\right)\partial_{ts}u\, ds.
\end{equation}
We estimate the right-hand side of \eqref{eq:evol3d} by applying the chain-rule to the $\partial_{sss}$ and
$\partial_{ss}$ operators involved. Among the resulting terms, the only one that cannot be straightforwardly
bounded in absolute value by $C(\ell) \|\partial_{ss} u\|_2^2 \left(\|\partial_{sss}u\|^2_2 + \|\partial_{ts}u\|_2^2\right)$
using H\"older's inequality and Sobolev's embedding is
$$
D_1 := -2\int_{T^1} \partial_{sss}\left(|\partial_s u|^2\right) u \partial_{ts}u\, ds.
$$  
We integrate by parts
$$
D_1 = 2\int_{T^1} \partial_{ss}\left(|\partial_s u|^2\right)\partial_s\left( u \partial_{ts}u\right)\, ds
$$
and, in view of the equality $0 = \partial_t|u|^2 = u\partial_t u$,  we develop
$$
\partial_s\left( u \partial_{ts}u\right)= \partial_s u \partial_{ts}u + u\partial_{tss}u  = \partial_s u
\partial_{ts}u + u\partial_{tss}u - \partial_{ss}(u\partial_t u) 
= -\partial_s u \partial_{ts} u - \partial_{ss}u \partial_t u
$$
so that
$$
D_1 = -2\int_{T^1} \partial_{ss}\left(|\partial_s u|^2\right)\left(  \partial_s u \partial_{ts} u + \partial_{ss}u \partial_t u
  \right)\, ds
$$
and here again H\"older's inequality and Sobolev's embedding yield the claim. \qed 

\smallskip
\noindent{\bf Proof of Corollary \ref{cor:controlH3}.} This follows from Lemma \ref{lem:controlH3}, Corollary
\ref{cor:controlH2} and the inequality
$$
\|\partial_{ts}u(0,\cdot)\|_2^2 = \|\partial_s u^0 \times \partial_{ss} u^0 + u^0\times \partial_{sss} u^0\|_2^2
\leq C(\ell) \|\partial_{ss}u^0\|^4_2  + 2 \|\partial_{sss}u^0\|^2_2,
$$
where we used Sobolev's embedding.\qed

\smallskip

\noindent{\bf Proof of Lemma \ref{lem:difference}.}
We subtract \eqref{eq:schrodimap} for $u$ and $\check u$, and then also \eqref{eq:degre2} for $u$ and $\check u.$
This yield
\begin{equation}\label{eq:diff01}
\partial_t w = w \times \partial_{ss}z + z\times \partial_{ss} w,
\end{equation}
and  
\begin{equation}\label{eq:diff02}
\partial_{tt} w + \partial_{ssss} w = -2\partial_{ss}\left( Q
w\right) -4 \partial_{ss}\left( (\partial_sz\cdot \partial_s w) z\right) + \partial_s\left( Q
 \partial_s w\right) + 2\partial_s\left( (\partial_s z\cdot\partial_s w) \partial_s z\right),
\end{equation}
where $z:= (u+\check u)/2$ and $Q:= \frac{|\partial_s u|^2+|\partial_s \check u|^2}{2}.$
Multiplying \eqref{eq:diff01} by $w$ and integrating by parts on $T^1$ we obtain
$$
\frac{d}{dt}\int_{T^1}\frac{|w|^2}{2} = \int_{T^1} (w \times z)\cdot \partial_{ss} w = - \int_{T^1} (w\times
\partial_{s} z)\cdot \partial_s w
$$
so that H\"older's inequality and Sobolev's embedding yield the desired claim at the $L^2$ level. Next, 
differentiating \eqref{eq:diff01} with respect to $s$, multiplying the resulting equation by $\partial_s w$ and integrating by
parts on $T^1$ we obtain, after some cancellation, 
$$
\frac{d}{dt}\int_{T^1}\frac{|\partial_sw|^2}{2} = \int_{T^1} (w \times \partial_{sss}z)\times \partial_s w,
$$ 
and the desired claim at the $H^1$ level follows. Finally, multiplying \eqref{eq:diff02} by $\partial_t w$ and
integrating by parts on $T^1$ yields
$$
\frac{d}{dt}\int_{T^1}\left( \frac{|\partial_{ss}w|^2}{2}+ \frac{|\partial_tw|^2}{2}\right) =-2 E_1 -4E_2 +E_3+2E_4,  
$$
where
$$
\begin{array}{ll}
E_1:= \int_{T^1}\partial_{ss}\left( Q
w\right)\partial_t w,   & E_2:= \int_{T^1}\partial_{ss}\left( (\partial_sz\cdot \partial_s w) z\right)  \partial_t
w,\\  
 E_3:=\int_{T^1}\partial_s\left( Q
 \partial_s w\right) \partial_t w,
& E_4:= \int_{T^1}\partial_s\left( (\partial_s z\cdot\partial_s w) \partial_s z\right)\partial_t w.
\end{array}
$$
The terms $E_1,$ $E_3$ and $E_4$  can straightforwardly be treated by H\"older's inequality and Sobolev's embedding.  We
expand $\partial_{ss}$ in  $E_2$ through the chain-rule as
$$
E_2 = \int_{T^1} (\partial_{sss}w \partial_s z)(z\partial_t w) + E_{2,1} + E_{2,2}
$$ 
and once more $E_{2,1}$ and $E_{2,2}$ are  treated by H\"older's inequality and Sobolev's embedding. Finally,
since $0=|u|^2-|\check u|^2$, we have $z \partial_t w = -w\partial_t z$ and therefore
$$
\int_{T^1} (\partial_{sss}w \partial_s z)(z\partial_t w) = - \int_{T^1} (\partial_{sss}w \partial_s z)(w\partial_t
z) = \int_{T^1} \partial_{ss}w\partial_s \left( \partial_s z (w \partial_tz)\right),
$$  
which is now in a suitable form to be estimated as above. This yields the desired claim at the $H^2$ level. 
\qed

\smallskip

\noindent{\bf Proof of Corollary \ref{cor:controldiff}.}
It follows directly from Lemma \ref{lem:difference}, Corollary \ref{cor:controlH3} and the fact that
\begin{equation*}\begin{split}
\|\partial_t u(0,\cdot) -\partial_t \check u(0,\cdot)\|_2^2 &= \frac{1}{4} \|(u^0-\check u^0)\times \partial_{ss}(u^0+\check
u^0) + (u^0+\check u^0)\times \partial_{ss}(u^0-\check u^0) \|_2^2\\ 
&\leq
C(\ell,\|\partial_{sss}u^0\|_2,\|\partial_{sss}\check u^0\|_2) \|u^0-\check u^0\|_{H^2}^2.
\end{split}\end{equation*}
\qed

\medskip

\noindent{\bf Addresses and E-mails: }

\smallskip
\noindent
Robert Jerrard. Department of Mathematics, University of Toronto, Toronto, Ontario M5S 2E4, Canada. E-mail:
{\tt rjerrard@math.utoronto.ca}

\smallskip

\noindent
Didier Smets. Laboratoire Jacques-Louis Lions, Universit\'e Pierre \& Marie Curie, 4 place Jussieu BC 187, 75252
Paris Cedex 05, France. E-mail: {\tt smets@ann.jussieu.fr}

\end{document}